\documentclass[english]{amsart}
\title[Fundamental group of partial compactification of arrangements]{Fundamental groups of partial compactifications of arrangements and homology planes}
\author{Rodolfo Aguilar Aguilar}\thanks{Partially supported by ANR project Hodgefun ANR-16-CE40-0011-01 and by the National program ``Young scientists and Postdoctoral students", Funded by the  Ministry of Education and Science of Bulgaria}
\date{}

\usepackage{hyperref}

\usepackage[english]{babel}
\usepackage{amssymb}
\usepackage{amsthm}
\usepackage{mathrsfs}
\usepackage{amsmath}
\usepackage{amsfonts}
\usepackage{graphicx}
\usepackage{subcaption}
\graphicspath{ {C:\Users\aguilaro\OneDrive\Documentos\Papers\Presentation general LAC surface\Boundary manifold methods} }
\usepackage{tikz-cd}
\usepackage{tikz}

\usetikzlibrary{calc} 
\usetikzlibrary{arrows.meta,bending,decorations.markings,intersections} 

\tikzset{
    arc arrow/.style args={%
    to pos #1 with length #2}{
    decoration={
        markings,
         mark=at position 0 with {\pgfextra{%
         \pgfmathsetmacro{\tmpArrowTime}{#2/(\pgfdecoratedpathlength)}
         \xdef\tmpArrowTime{\tmpArrowTime}}},
        mark=at position {#1-\tmpArrowTime} with {\coordinate(@1);},
        mark=at position {#1-2*\tmpArrowTime/3} with {\coordinate(@2);},
        mark=at position {#1-\tmpArrowTime/3} with {\coordinate(@3);},
        mark=at position {#1} with {\coordinate(@4);
        \draw[-{Stealth[length=#2,bend]}]       
        (@1) .. controls (@2) and (@3) .. (@4);},
        },
     postaction=decorate,
     }
}

\theoremstyle{plain} 

\newtheorem{thm}{Theorem}[section]
\newtheorem{lem}[thm]{Lemma}
\newtheorem{prop}[thm]{Proposition}
\newtheorem{cor}[thm]{Corollary}

\newtheorem{ques}{Question}

\theoremstyle{definition}
\newtheorem{defn}{Definition}[section]
\newtheorem{exmp}{Example}[section]

\theoremstyle{remark}
\newtheorem{rem}[thm]{Remark}

\newtheoremstyle{case}{}{}{}{}{}{:}{ }{}
\theoremstyle{case}

\DeclareMathOperator{\Bl}{Bl}
\DeclareMathOperator{\Hom}{Hom}

\DeclareMathOperator{\Sing}{Sing}

\DeclareMathOperator{\Ceva}{Ceva}

\newcommand{\abs}[1]{\left\vert#1\right\vert}
\newcommand{\A}{\mathscr{A}}

\begin{document}
\begin{abstract}
Following our previous work, we develop an algorithm to compute a presentation of the fundamental group of certain partial compactifications of the complement of a complex arrangement of lines in the projective plane. It applies, in particular, to homology planes arising from arrangements of lines. This permits to determine the infinitude of the fundamental group of	 homology planes of log-general type that may be the first examples of this type.
\end{abstract}

\maketitle
\section{Introduction}
The computation of the fundamental group of smooth quasi-projective complex varieties is in general a difficult task. There exist, however, certain classes of these type of varieties where specific methods have been developed in order to obtain at least, a presentation for these groups. For example, for the complement $\mathbb{P}^2\setminus \A$ of an arrangement of lines $\A$ in $\mathbb{P}^2$, a presentation for $\pi_1(\mathbb{P}^2\setminus \A)$ is obtained in \cite{Randell}, \cite{arvola}, \cite{suciu}, \cite{florens}.

In \cite{aguilar2019fundamental}, by modifying the method in \cite{Randell}, we have developed a method to obtain a presentation for the fundamental group of certain partial compactifications of the complement $\mathbb{P}^2\setminus \A$ of an arrangement of lines $\A\subset \mathbb{P}^2$ under the hypothesis that the lines in $\A$ are defined by real linear forms.

The purpose of this note is to generalize our method in two directions:
\begin{itemize}
\item to admit a general arrangement $\A\subset \mathbb{P}^2$ defined by complex linear forms and 
\item to admit a more general class $M(\A,I,P)$ of partial compactifications of $\mathbb{P}^2\setminus \A$. See \ref{ss:partialCArrangements} for a precise definition.
\end{itemize}

In particular, some of these partial compactifications give rise to homology planes, see below. It also contains some partial compactifications constructed by Fowler in the study of rational homology disks in \cite{Fowler}, see also \cite{wahl2020complex,bartolo2021rational}.

We proceed in two different ways: firstly, following \cite{arvola} and \cite{suciu}, whose work generalize \cite{Randell} to complex arrangements, we define a wiring diagram $\mathcal{W}$ that encodes some over or under-crossing of the lines in $\A$ arising by the complex nature of the forms defining them. The graph $\mathcal{W}$ encodes enough information to obtain a presentation of $\pi_1(M(\A,I,P))$.

\begin{thm}\label{thm:Int1} A presentation for $\pi_1(M(\A,I,P))$ can be obtained from $\mathcal{W}$. The set of generators are in correspondence with the set of lines in $\A$ and the set of relations has two types of them: 
\begin{itemize}
\item those relations $R_p$ coming from a singular point $p$ of $\A$. These relations already appeared in a presentation of $\pi_1(M(\A))$ and 
\item for each element $\iota$ either in  $I$ or in $P$, a relation $R_\iota$ which is a product of conjugates of some generators depending on $\iota$. 
\end{itemize}
\end{thm}
As in \cite{aguilar2019fundamental}, a main step in the proof of Theorem \ref{thm:Int1} consists in the explicit computation of an expression for the meridians around certain exceptional divisors, obtained by blowing-up $\A$ in some singular points, in terms of the generators. We also obtain in Theorem \ref{thm:OrbifoldPresentation}, similar to \textit{loc. cit.}, a presentation for the fundamental group of an orbifold whose moduli space $\Bl_{P_0}\mathbb{P}^2$ is a blow-up of $\mathbb{P}^2$ at some singular points $P_0$ of the arrangement $\A$ and the non-trivial isotropy groups lie in a divisor $D$ contained in the total transform of $\A$ in $\Bl_{P_0}\mathbb{P}^2$.   

Secondly, let $U$ denote a closed regular tubular neighborhood of $\A$ in $\mathbb{P}^2$. We call $\partial U$ the boundary manifold of $\A$. In \cite{florens}, a presentation for $\pi_1(M(\A))$ is obtained from a presentation of $\pi_1(\partial U)$ by studying the map $\pi_1(\partial U)\to \pi_1(M(\A))$ induced by the inclusion $\partial U\hookrightarrow M(\A)$.  For the homological version see \cite{guerville}. 

It turns out that their methods can also be applied to determine a presentation for the fundamental group of some partial compactifications $M(\A,I,P)$. However, in order to study the boundary manifolds $\partial U_D$ of strict transforms $D$ of $\A$ in some birational model of $\mathbb{P}^2$, we start from a different presentation for the boundary manifold $\partial U$ of $\A$ than that used in \cite{florens}.

Indeed, when $D=\sum D_i$ is a connected, simple normal crossing divisor such that $\pi_1(D)$ is trivial, Mumford gave a presentation for $\pi_1(\partial U_D)$ in \cite{mumfordtopology}. This, together with the graph-manifold structure in the sense of Waldhausen \cite{Waldhausen1967}, permitted Westlund to give a presentation of $\pi_1(\partial U)$ in \cite{westlund} (see also \cite{Cohen_2008}). Here, by a choice of a surface birational to $\mathbb{P}^2$ where the strict transform of $\A$ satisfies the hypothesis for the presentation of Mumford, we obtain the same presentation of Westlund. See Theorem \ref{thm:Westlund}. Following this construction, we are able to give a presentation for the fundamental group of a boundary manifold $\partial U_D$ of a divisor $D$ lying in a surface $\bar{X}$ obtained by successive blows-up of $\mathbb{P}^2$ such that $M(\A,I,P)=\bar{X}\setminus D$.

We obtain in Theorem \ref{thm:Guerville} a presentation for $\pi_1(M(\A))$ by studying the map $i_*:\pi_1(\partial U)\to \pi_1(M(\A))$. Moreover, as the construction for $\pi_1(\partial U)$ depends of several choices, we can make them in such a way that the image under $i$ of the meridians of the lines in $\A$ lying in $\partial U$, whose homotopy class are part of the generators of $\pi_1(\partial U)$, lie in the same homotopy class as the meridians constructed for Theorem \ref{thm:Int1}. From this, we do not only obtain that the presentation of Theorem \ref{thm:Guerville} and \ref{thm:Int1} are equivalent, but that the image of the set of relation in the presentation of $\pi_1(M(\A))$ coincides with the relations as in Theorem \ref{thm:Int1}. From this, we can obtain a presentation for partial compactifications $\pi_1(M(\A,I,P)$, see Theorem \ref{thm:BoundaryLAC}.

The presentation of Theorem \ref{thm:Int1} applies in particular for homology planes, affine smooth surfaces with trivial reduced integral homology, arising from arrangement of lines. To the knowledge of the author, no general algorithm for computing a presentation for their fundamental groups were known.

The arrangements of lines giving rise to homology planes were classified by T. Tom Dieck and T. Petri, see \cite{Dieck1989,tomDieck1991, weko_39408_1} . There exists one infinite family $L(1,n+1)$ and six arrangement $L(2),\ldots, L(7)$ of at most ten lines.

As in algebraic surfaces, the study of homology planes is usually divided by its log-Kodaira dimension: the only homology plane $X$ with $\bar{k}(X)=-\infty$ satisfies $X\cong \mathbb{C}^2$ \cite[Theorem 3.2]{fujita1979zariski}, there are no homology planes with $\bar{k}(X)=0$ and all the homology planes with $\bar{k}(X)=1$ can be obtained from $L(1,n+1)$, see \cite{miyanishiopen}.

For homology planes $X$ with $\bar{k}(X)=2$, Tom Dieck and Petri gave examples arising from $L(2)$ in \cite{tomdieck1993} as well as a general algorithm to obtain homology planes out of arrangement of lines. As the fundamental group $\pi_1(\mathbb{P}^2\setminus L(2))$ is abelian, all homology planes arising from $L(2)$ are contractible.  In \cite{AST_1993__217__251_0},  Zaidenberg gave a countable number of contractible homology planes arising from the arrangement $L(3)$. Here we can obtain the following result.

\begin{thm} There exists homology planes of log-general type with infinite fundamental group.
\end{thm}
These homology planes arise as in \cite{tomdieck1993}. They may be the first examples of this type. The infinitude was obtained at first, using the presentation of Theorem \ref{thm:Int1}, but the nature of the groups suggested the following more general construction: using certain pencils of lines associated to the (sub)-arrangements $L(3),L(4),L(5)$, as presented by Suciu in \cite{inproceedings}, we can construct homology planes $X$ having a fibration that induces a surjective homomorphism to the fundamental group of an orbicurve which is infinite. In fact, all homology planes with infinite fundamental group that we obtain are of this form. We present examples where the general fiber is $\mathbb{C}$, see examples \ref{exmp:L4.1}, \ref{exmp:L5.1} and others where the general fiber is $\mathbb{P}^1\setminus \{3-\text{points}\}$, see examples \ref{exmp:L3}, \ref{exmp:L4.2}. This suggests the following question:

\begin{ques} Let $X$ be a homology plane with infinite fundamental group. Does $\pi_1(X)$ always admit an infinite quotient isomorphic to the fundamental group of an orbicurve?
\end{ques}

This is the case for the homology planes $X$ with $\bar{k}(X)=1$, see \cite{miyanishiopen}. Adapting these arguments, we show in Proposition \ref{prop:L4Structure}, that this is also the case for the homology planes arising from $L(3)$.

Finally, by using again Theorem \ref{thm:Int1}, we give examples of contractible homology planes arising from $L(6)$ and $L(7)$. Following the arguments of Zaidenberg in \cite{AST_1993__217__251_0}, we can show that these examples are not isomorphic to those obtained from $L(1,n+1)$ and $L(3)$. These arguments also show that these homology planes are of log-general type.

\noindent \textbf{Acknowledgments}. I would like to thank my advisor P. Eyssidieux for a careful reading and valuable comments of this note. I would also like to thank M. Zaidenberg who introduced me to homology planes and the referees of my thesis E. Artal Bartolo and M. Teicher.

\section{Preliminaries}
\subsection{Notations}
We will denote by $\mathbb{P}^2$ the complex projective plane. 

Let $\A=\{L_1,\ldots,L_{n+1}\}$ be an arrangement of $n+1$ lines in $\mathbb{P}^2$. The complement of the arrangement will be sometimes denoted by $M(\A):=\mathbb{P}^2\setminus \A$. 

Let $X$ be a complex manifold, for $p\in X$ we denote by $\pi:\Bl_p X\to X$ the blow up of $X$ at $p$. If $D\subset X$ is a divisor, we denote by $\abs{D}$ the reduced divisor with the same support as $D$ and by  $\Sing D$ the set of singular points of $D$. 

We will denote by  $a^b=b^{-1} a b$ if $a,b\in G$ with $G$ a group. If $a\in G$ and $b\in \mathbb{Z}$, we denote as well by $a^b$ the $b$-power of $a$.

\subsection{Meridians}
Let $X$ be a complex manifold and $H\subset X$ a hypersurface. Let $p\in H$ be a smooth point and $\Delta$ a disc cutting transversaly $H$ at $p$. A loop $\gamma$ in $\pi_1(X\setminus H)$ freely homotopic to the boundary of $\Delta$ with the natural orientation is called a \emph{meridian}.

The following Proposition is well-known, for a proof see \cite{shimada}.
\begin{prop}\label{prop:meridians}
Let $X$ be a complex manifold and $D=\sum D_i$ a divisor such that each irreducible reduced component $\abs{D_i}$ of $D$ is smooth. Let $\gamma_i$ be a meridian of $\abs{D_i}$, then every other meridian of $\abs{D_i}$ is a conjugate of $\gamma_i$ in $\pi_1(X\setminus D)$ and the kernel of the map $\pi_1(X\setminus D) \to \pi_1(X)$ is the normal subgroup generated by the meridians of its irreducible components.
\end{prop}

\subsection{Dual graph of a divisor and partial compactifications of its complement}\label{sbs:dualgraph}
Let $\bar{X}$ be a projective smooth surface and let $D=\sum_{i=1}^{N} D_i\subset \bar{X}$ be a reduced simple normal crossing divisor with the $D_i$ being the irreducible components of $D$ and denote by $w_i=D_i\cdot D_i$ the self-intersection number of $D_i$. Let $\Delta$ be the unoriented graph where the vertices $V(\Delta):=\{v_1,\ldots,v_{N}\}$ are in correspondence with the irreducible components $D_i$ of $D$ and the edges $\mathcal{E}(\Delta)$ correspond with the intersection of the irreducible components of $D$, this is, there is an unoriented edge joining $v_i$ and $v_j$ for each point in $D_i\cap D_j$. Denote by $X:=\bar{X}\setminus D$.

We want to define some partial compactifications of $X$. The idea goes as follows: we choose a subset of irreducible components of $D$ indexed by $I$ which are not to be removed from $\bar{X}$, we then select a subset $P$ of points in $\Sing \sum_{i \not \in I} D_i$ to be blown-up and remove the \emph{strict} transform of $\sum_{i\not \in I} D_i$ in $\Bl_P \bar{X}$.

More precisely, let $I\subset \{1,\ldots, N\}$, $P=\{p_1,\ldots, p_{s_1}\}\subset \Sing (\sum_{i\not \in I}D_i)$ and denote by $\pi:\Bl_P \bar{X}\to \bar{X}$ the composite of the blow-ups at the points in $P$. Denote by $\pi^* D=\sum_{i=1}^{N+s_1} D_i'$ the total transform of the divisor $D$ in $\Bl_P\bar{X}$, suppose that for $i=1,\ldots,N$, we have that $D_i'$ is a strict transform of $D_i$ and for $j=1,\ldots,s_1$, the $D_{N+j}'$ are exceptional divisors. Define the divisor $$D'(I,P)=\pi^*D -\sum_{i\in I}D_i'-\sum_{N<j}D_j'.$$ Note that $\Bl_P\bar{X}\setminus \pi^* D\hookrightarrow X'(I,P):=\Bl_P\bar{X}\setminus D'(I,P)$. By restricting $\pi$, we obtain an isomorphism $\Bl_P\bar{X}\setminus \pi^* D\overset{\sim}{\to} X$. We call $X'(I,P)$ a \emph{partial compactification} of $X=\bar{X}\setminus D$. By Proposition \ref{prop:meridians}, the induced homomorphism $\pi_1(X)\to \pi_1(X'(P,I))$ is surjective. 

We comment on the effects of this construction in the dual graph. Denote by $\Delta'(I,P)$ the dual graph of $D'(I,P)$. It is obtained from $\Delta$ by deleting the following vertices and edges: for the set $I$, we have a subset $V(I)\subset V(\Delta)$ of vertices corresponding to the lines $D_i$ for $i\in I$, remove these vertices from $\Delta$, together with all edges in $\mathcal{E}(\Delta)$ having an endpoint in $V(I)$. 
We also remove the edges corresponding to $P$: let $p_j\in P$, there exists $j_1,j_2\in \{1,\ldots,N\}$ such that $p_j=D_{j_1}\cap D_{j_2}$. In the dual graph of $\pi^* D$ the edge corresponding to $p_j$ in $\Delta$ has been divided in two, with a vertex in between corresponding to the exceptional divisor coming from $p_j$.

\subsubsection{Partial compactifications for an arrangement of lines}\label{ss:partialCArrangements}
We can carry the above construction for a divisor $D\subset \bar{X}$ coming from an arrangement of lines $\A=\{L_1,\ldots,L_{n+1}\}\subset \mathbb{P}^2$. In fact, this will be the only case we will be interested in.

Let $\A\subset \mathbb{P}^2$ be an arrangement of lines. Denote by $P_0:=\{p_1,\ldots, p_{s_0}\}\subset \Sing \A$ the points with multiplicity strictly bigger that $2$. Define $\pi:\bar{X}:=\Bl_{P_0}\mathbb{P}^2\to \mathbb{P}^2$ and denote by $D=\abs{\pi^* \A}=\sum_{i=1}^{n+1+s_0} D_i$ the reduced total transform of $\A$ in $\bar{X}$. Note that $D$ is simple normal crossing. For a divisor $D$ where the irreducible components are smooth rational curves, the set of edges $\mathcal{E}(\Delta)$ of the dual graph $\Delta$ can be described as $\mathcal{E}(\Delta)=\{(i,j)\in \{1,\ldots,n+1+s_0\}^2\mid D_i\cap D_j\not =\varnothing, i<j\}$ once the irreducible components of $D$ are numbered. We assume that $D_i$ is the strict transform of $L_i$.

Let $I\subset \{1,\ldots,N=n+1+s_0\}$ and $P=\{p_1',\ldots,p_{s_1}'\}\subset \Sing \sum_{i\not \in I}D_i$. Consider $\pi':\Bl_{P}\bar{X}\to \bar{X}$ and let $D'=\pi'^*(D)-\sum_{i\in I} D_i'-\sum_{N<j} D_j'$ as above. We write $M(\A,I,P):=X'(I,P)=\Bl_{P}\bar{X}\setminus D'$ for a \emph{partial compactification} of the complement of an arrangement $M(\A)=\mathbb{P}^2\setminus \A$.

We can iterate this construction in the following way; consider a sequence of blow-ups: 
$$\Bl_{P_k,\ldots,P_1}\bar{X}\overset{\pi^{(k)}}{\to} \Bl_{P_{k-1},\ldots,P_1}\bar{X}\overset{\pi^{(k-1)}}{\to}\ldots \overset{\pi^{(2)}}{\to} \Bl_{P_1}\bar{X}\overset{\pi^{(1)}}\to \bar{X}\overset{\pi^{(0)}}{\to}\mathbb{P}^2$$
with $P_l\subset \Sing ((\pi^{(0)}\circ\pi^{(1)}\circ \cdots\circ \pi^{(l-1)})^* \A)$ for $l=1,\ldots,k$ and $\pi^{(l)}:\Bl_{P_l,\ldots,P_1}\bar{X}\to \Bl_{P_{l-1},\ldots, P_1} \bar{X}$ denoting the blow-up of $\Bl_{P_{l-1},\ldots, P_1} \bar{X}$ at $P_l$.
	 We can suppose that  the irreducible components of the reduced divisor $$D':=\abs{(\pi^{(0)}\circ\cdots\circ \pi^{(k)})^* \A}=\sum_{1}^{\abs{\A}} D_i'+\sum_{\abs{\A}}^{\abs{\A}+\abs{P_0}} D_j'+\ldots +\sum_{\abs{\A}+\abs{P_0}+\ldots+\abs{P_{k-1}}}^{\abs{\A}+\abs{P_0}+\ldots +\abs{P_k}} D_l',$$
where $\abs{P}$ denotes the cardinality of the set $P$, are ordered in such a way that $\pi^{(l)}\circ\cdots\circ \pi^{(k)}$ contracts the curves $D_i'$ with $i>\abs{\A}+\ldots+ \abs{P_{l-1}}$ for $l=1,\ldots,k$. Let $I\subset \{1,\ldots,\abs{\A}+\ldots+\abs{P_k}\}$ and define $M(\A,I,P_1,\ldots,P_k):=\Bl_{P_k,\ldots,P_1}\bar{X}\setminus D'-\sum_{i\in I} D_i'$ as an \emph{iterated partial compactification} of $M(\A)$. 

\begin{lem} Let $(\bar{X}',D')$ be a smooth projective surface such that 
\begin{enumerate}
\item the divisor $D'$ is a simple normal crossing divisor,
\item there is a birational morphism $\bar{X}'\overset{\psi}{\to} \bar{X}$,
\item we have that $\psi^* D \supset D'$
\end{enumerate}
then there exists an iterated partial compactification $(\bar{X}'', D'')$ and a proper birational morphism $\bar{X}'\overset{\psi ''}{\to} \bar{X}''$ such that $(\psi'')^{-1}D''\supset D'$ and $\pi_1(\bar{X}''\setminus D'') \overset{\sim}{\leftarrow} \pi_1(\bar{X}'\setminus D')$ is an isomorphism.
\end{lem}

Here we will restrict the study to $M(\A,I,P)$ unless otherwise stated. The results are easily generalized to the above more general setting of iterated partial compactifications.

\begin{rem} We have that $\pi_1(\bar{X}'\setminus D')$ is a quotient group of $\pi_1(M(\A))$ by proposition \ref{prop:meridians}.
\end{rem}
\subsection{Boundary manifolds}\label{ss:mumford}



Let $\bar{X}$ be a projective smooth surface and $D=\sum_{i=1}^k D_i\subset \bar{X}$ be a connected divisor. We can construct a regular tubular neighborhood $U$ of $D$ in $\bar{X}$ which comes with a surjective continuous retraction $\varphi:U \to D$ such that $\varphi|_D=id_D$. The boundary $\partial U$ of $U$ is an oriented, connected, closed $3$-manifold (see \cite{mumfordtopology}). We call the $3$-manifold $\partial U$ the \emph{boundary manifold} of $D$ and denote by $\psi:\partial U \to D$ the restriction of $\varphi$ to $\partial U$. 

Suppose now that $(X,D)$ is simple normal crossing and assume that: 
\begin{itemize}
\item the divisor $D$ is connected,
\item the irreducible components $D_i$ of $D$ are rational curves ,
\item the dual graph of $D$ has no cycles, in particular $\# D_i\cap D_j=0 \text{ or } 1 \text{ if } i\not =j$. This dual graph is a tree that we denote by $\mathcal{T}$.
\end{itemize} For such a pair, a presentation of $\pi_1(\partial U)$ is given in \cite[p. 235]{mumfordtopology} (See also \cite{hirzerbruchbour}). As we shall need the notations, let us describe it.

 Fix a base point $Q_i\in D_i\setminus \cup_{i\not = m} D_m$ in every rational curve $i=1,\ldots,k$. Denote by $P_{im}'$ the unique point in $D_i\cap D_m$, if any. Select a simple contractible oriented curve $l_i\subset D_i$ containing $Q_i$ and passing through every point $P_{im}'\in D_i$ as in figure \ref{fig:meridiansMumford} and denote by $l=\cup l_i\subset D$. We can construct a continuous map $h:l\to \partial U$ such that $\psi\circ h|_{l_i}=id_{l_i}$ and $h(l_i)\cap h(l_m)\not =\varnothing$ if $P_{im}'=D_i\cap D_m=l_i\cap l_m\not =\varnothing$.

\begin{figure}[h]
\begin{subfigure}[b]{0.3\textwidth}
        \centering

        \begin{tikzpicture} 
\draw (0,0) circle (2);

\draw (-.8,1) node [left]{\small $\partial\Delta(P_{i{1}})$};
\draw (-.8,0) node [left]{\small $\partial\Delta(P_{i2})$};
\draw (-.8,-1) node [left]{\small $\partial \Delta(P_{i{k_1}})$};

\draw(0,1) node[anchor=east]  {\tiny $P_{i{1}} $};
\draw(0,0) node[anchor=east]  {\tiny $P_{i2} $};
\draw(0,-1) node[anchor=east]  {\tiny $P_{i{k_i}} $};

\draw (.3,-1.6) node[right] {\tiny $Q_i$};

\draw (.3,.5) node[right] {\tiny $l_i$};

\draw (1.4,-1.5) node[right] {$D_i$};


\draw [arc arrow=to pos 0.45 with length 2mm] (-.7,0) to[out=-90,in=-90] 
(.7,0) [arc arrow=to pos 0.95 with length 2mm] 
to[out=90,in=90] cycle;
\draw [arc arrow=to pos 0.45 with length 2mm] 
  (0,0) to[out=0,in=0] (0,1);

\draw [arc arrow=to pos 0.45 with length 2mm] (-.7,1) to[out=-90,in=-90] 
(.7,1) [arc arrow=to pos 0.95 with length 2mm] 
to[out=90,in=90] cycle;
\draw  (0,-1) to[out=0,in=0] (0,0);

\draw [arc arrow=to pos 0.45 with length 2mm] (-.7,-1) to[out=-90,in=-90] 
(.7,-1) [arc arrow=to pos 0.95 with length 2mm] 
to[out=90,in=90] cycle;

\draw  (.3,-1.6) to[out=90,in=0] (0,-1);

\foreach \Point in {(0,-1), (0,0), (0,1), (.3,-1.6)}{
    \node at \Point {\textbullet};
}


\end{tikzpicture}
\caption{Mumford generators}
\label{fig:meridiansMumford}     
    \end{subfigure}
~ ~ ~~~   
    \centering
     \begin{subfigure}[b]{0.3\textwidth}
        \centering

        \begin{tikzpicture}

\draw(0,1) node[anchor=east]  {\tiny $P_{i{1}} $};
\draw(0,0) node[anchor=east]  {\tiny $P_{i2} $};
\draw(0,-1) node[anchor=east]  {\tiny $P_{i{k_1}} $};

\draw (.3,-1.6) node[right] {\tiny $Q_i$};

\draw (.5,0) node[right] {\tiny $l_i'$};
\draw (.5,1.5) node[right] {\tiny $\beta_{i1}'$};

\draw (1.4,-1.5) node[right] {$D_i$};

\draw (0,0) circle (2);


\draw  (.3,-1.6) -- (.3,-1.35);
\draw (.3,-1.35) to[out=0,in=0] (.3,-.65);
\draw  (.3,-.65) -- (.3,-.35);
\draw [arc arrow=to pos 0.45 with length 2mm] 
 (.3,-.35) to[out=0,in=0] 
(.3,.35);
\draw  (.3,.35) -- (.3,.63);

\draw [arc arrow=to pos 0.45 with length 2mm] (-.7,1) to[out=-90,in=-90] 
(.7,1) [arc arrow=to pos 0.95 with length 2mm] 
to[out=90,in=90] cycle;


\foreach \Point in {(0,-1), (0,0), (0,1), (.3,-1.6)}{
    \node at \Point {\textbullet};
}


\end{tikzpicture}
\caption{Paths in $D_i^*$}
\label{fig:meridiansMumford2}   
    \end{subfigure}
    \caption{Generators}\label{fig:MerMum}
\end{figure}
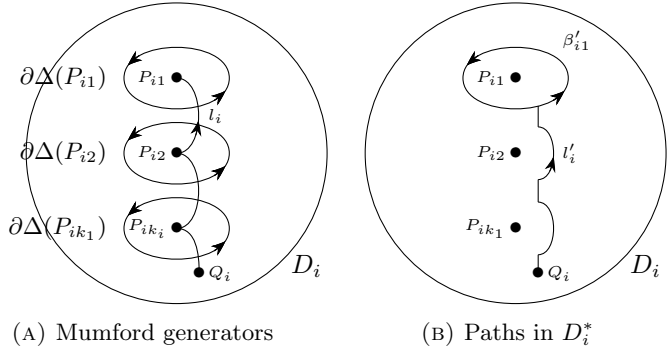

It is easy to see that $l$ is a homeomorphic image of a tree and deformation retracts to a point.

Label the points $P_{im}'\in D_i$ by the order they intersect $l_i$ as $P_{i1},\ldots, P_{i{k_i}}$, see Figure \ref{fig:meridiansMumford}.
Denote by $\psi_i:\partial U_i	\to D_i$ the boundary manifold of $D_i$. Let $D_i^*=D_i\setminus \cup_{m=1}^{k_i} \Delta(P_{im})$ with $\Delta(P_{im})$ a small open disk around $P_{im}$ in $D_i$. Define $\partial U_i^*:=\psi_i^{-1}(D_i^*)$. We may suppose that $\partial U\cap \partial U_i=\partial U_i^*$.


We may also assume that $Q_i\in D_i^*$. Define another contractible path $l_i'\subset D_i^*$ as follows: join every two connected components of $l_i\cap D_i^*$ touching the boundary of a disk $\partial \Delta(P_{im})$, by the segment of $\partial \Delta(P_{im})$ that connects these two points when traveled in the natural orientation, see figure \ref{fig:meridiansMumford2}. We assume $l_i$ and $\partial \Delta(P_{im})$ intersect transversally at two points for all $m =1,\ldots,k_i-1$.

 Consider the circle $\partial\Delta(P_{im})$ traveled in the natural orientation and connect it to $Q_i$ via a segment of $l_i'$. We obtain a path $\beta_{im}'\in \pi_1(D_i^*,Q_i)$, for $m=1,\ldots,k_i$, see figure \ref{fig:meridiansMumford2}. Note that $\beta_{i1}'\cdots \beta_{ik_i}'=1$ in $\pi_1(D_i^*)$.

We can construct continuous maps $h_i:\cup_{m=1}^{k_i} \beta_{im}'\to \partial U_i^*$ such that $\psi_i\circ h_i|_{\beta_{im}'}=id_{\beta_{im}'}$ for every $i=1,\ldots, k$. Let $h_i(Q_i)$ be a base point in $\partial U_i^*$, denote by $\gamma_{im}'=h_i(\beta_{im}')$ and let $\gamma_i'$ be a fiber $S^1$ at $Q_i$ of $\partial U_i^*$ traveled in the natural orientation.

By using the long homotopy sequence of a fiber bundle, Mumford obtained the following presentation in \cite{mumfordtopology}. See also \cite{hirzerbruchbour}.
\begin{lem}\cite[p. 236-237]{mumfordtopology} \label{lem:RelsInBundle} The fundamental group of $\partial U_i^*$ is given by the following presentation
\begin{equation}\label{eq:RelsInBundle}
\left\langle \gamma_{i1}',\ldots, \gamma_{i{k_i}}', \gamma_i' \mid [\gamma_{im}',\gamma_i'] \ m=1,\ldots,k_i, \gamma_i'^{-w_i}=\gamma_{i1}'\cdots \gamma_{ik_i}'   \right\rangle 
\end{equation}
with $w_i=D_i\cdot D_i$ the self-intersection number of $D_i$.
\end{lem}
\begin{rem}\label{rem:Mumford}
Note that $\partial U_i^*$ is non canonically homeomorphic to the trivial bundle $S^1\times D_i^*$, but the image of the paths $\gamma_{im}'$ are not longer identified with a path freely homotopic to one of the form $\{\text{point}\}\times \partial \Delta(P_{i_m})$. In fact, we need to  twist this image by a multiple of $\gamma_i'$ for it to be of such form. See \cite[p. 235]{mumfordtopology}.

\end{rem}

Now, to globalize this construction to $\partial U$, we can use $h(l)\subset \partial U$ as a skeleton to define paths generating $\pi_1(\partial U)$. Let $\gamma_i$ be the loop based at $h(Q_1)$ constructed as follows. Join $h(Q_1)$ to $h(Q_i)$ by a segment $\lambda$ of $h(l)$, follow $\gamma_i'$ and come back by $\lambda^{-1}$. Then it is homotopic to the canonical representative of $\gamma_i'$ in $\pi_1(\partial U_i^*\cup h(l),h(Q_1))$ using the natural isomorphism $\pi_1(\partial U_i^*\cup h(l),h(Q_1))\to \pi_1(\partial U_i^*,Q_i)$ thus obtained. Define similarly $\gamma_{im}$ for $1\leq m\leq k_i$. Then $\gamma_{im}=\gamma_{j_\mathcal{T}(i,m)}$ for some injective map $m \mapsto j_\mathcal{T}(i,m)$ from $\{1,\ldots,k_i\}$ to $\{1,\ldots, k\}$. 

 By gluing the $\partial U_i^*$ together and by using van Kampen theorem, Mumford obtained the following presentation for $\pi_1(\partial U)$.


\begin{thm}[\cite{mumfordtopology}]\label{thm:Mumford}  With the notations and assumptions as above, a presentation for $\pi_1(\partial U)$ is given by:
 $$\pi_1(\partial U)=\left\langle \gamma_1, \ldots, \gamma_{k}  \mid [\gamma_i,\gamma_{j_{\mathcal{T}}(i,m)}],  \ m=1,\ldots, k_i, \gamma_i^{-w_i}=\prod_{m=1}^{k_i} \gamma_{j_{\mathcal{T}}(i,m)},  1\leq i \leq k
 \right\rangle. $$
\end{thm}
\noindent where $w_i=D_i\cdot D_i$, $[a,b]=aba^{-1}b^{-1}$, $a^0=1$, the identity of the group.

\section{Wiring diagrams and a first presentation of the fundamental group of a partial compactification}\label{sec:Wirtinger}

We will describe the construction of a diagram permitting to express some me\-ri\-dians around different points in the lines of $\A$,  lying in a pencil of lines passing through a base point $R\in \mathbb{P}^2\setminus \A$, in terms of a fixed set of meridians lying in a fixed fiber of the pencil.

As an application we obtain a first presentation for the fundamental group of a partial compactification $M(\A,I,P)$. To do that we will use a modification of the presentation of the fundamental group of $M(\A)$ given in \cite{arvola} and \cite{suciu}. 

This diagram will also carry the information to compute the image of the cycles in the boundary manifolds of $\A$ into $M(\A)$. This will be done in section \ref{s:BoundaryM}
\subsection{Wiring diagram associated to a complex arrangement}\label{ss:111} Consider an arrangement of lines $\A$ in $\mathbb{P}^2$. 
Let us fix a base point $R\in\mathbb{P}^2\setminus \A$ and denote by $\pi_R:\Bl_R\mathbb{P}^2\to \mathbb{P}^2$ the blow-up at $R$. 
Let $\bar{f}:\Bl_R \mathbb{P}^2 \to \mathbb{P}^1$ be the morphism defined by the pencil of lines passing through $R$. In what follows, we assume that we have chosen $R$ in such a way that $\bar{f}|_{\Sing \A}:\Sing \A\to \mathbb{P}^1$ is injective. 

 Let $\ast\in \mathbb{P}^1$, consider a simple piece-wise linear path $\beta:([0,1],0)\to (\mathbb{P}^1,\ast)$ starting at $\ast$ and passing through every point $\bar{f}(p)$ for all $p\in \Sing \A$, being locally linear around these points.  

 By abuse of notation let us denote by $\A$ the union of the lines of arrangement in $\mathbb{P}^2$. 

\begin{defn}\label{def:wire} The \emph{wiring diagram} of $\A$ with respect to $\beta$ is $\mathcal{W}=\bigcup_{t\in [0,1]} (\A\cap \bar{f}^{-1}{(\beta(t)}))\subset \Bl_R\mathbb{P}^2$. The $i$-wire $W_i$ is $L_i\cap \mathcal{W}$. Here, we view $\A, L_i$ as subvarieties of $\Bl_R\mathbb{P}^2$ since $R\not \in \A$.
\end{defn}
By the choice of $\beta$, as it passes through the points $\bar{f}(p)$ for $p\in \Sing \A$, we have that $\Sing \A\subset \mathcal{W}$.


\begin{lem} Every wire is a piece-wise linear simple curve.
\end{lem}
\begin{proof}
As no line in $\A$ passes through $R$, every $L_i\in \A$ induces a section of $\Bl_R \mathbb{P}^2\to \mathbb{P}^1$ which is in fact an isomorphism. By the choice of $\beta$ the result follows.
\end{proof}

\subsubsection{Planar representation of the wiring diagram} \label{ss:PlanarWiring}
By considering the pullback $\beta^\ast(\mathcal{W})$ and a trivialization $\beta^*\Bl_R\mathbb{P}^2\cong [0,1]\times \mathbb{P}^1$, we can view $\beta^* (\mathcal{W})$ as a closed graph embedded in  $[0,1]\times \mathbb{P}^1$. Sometimes we will continuing writing $\mathcal{W}$ for $\beta^*(\mathcal{W})$. Moreover we can remove the exceptional divisor $\pi_R^{-1}(R)$ from $[0,1]\times \mathbb{P}^1$ and we can view  $\mathcal{W}$ as a closed graph embedded in  $[0,1]\times \mathbb{C}$ via a piece-wise linear isomorphism.

 There exists a complex coordinate $z$ in $\mathbb{C}$ such that the  projection $(p:[0,1]\times \mathbb{C}\to [0,1]\times\mathbb{R}, (t,z)\mapsto (t,\Re(z)))$ is generic, in the sense that the extra crossings in $p(\mathcal{W})$ arise as transversal intersection of only two wires $p(W_i)$ and $p(W_k)$ for certain $t\in [0,1]$ and wires $W_i, W_k$ that do not intersect in $\bar{f}^{-1}{(\beta(t))}$. We call these crossings virtual vertices. We obtain a planar diagram which can be represented as in the figure \ref{fig:Wiring}. 
 
 We assume that the order of the lines $L_1,\ldots, L_n$ is such that, at the very right of the planar representation of $\mathcal{W}$, the wire $W_1$ is at the bottom of $\mathcal{W}$, above it is the wire $W_2$ and then $W_3$, continuing in this way until $W_n$.
 
\begin{defn} Consider coordinates $(t,x,y)$ in $\mathbb{R}^3$.  We say that a wire $W_{i}$ passes above $W_{k}$ at a point $t'\in [0,1]$ if $(t',x,y_i)\in W_{i}, (t',x,y_k)\in W_{k}$ and $y_i<y_k$.
\end{defn}  
 In order to distinguish the virtual vertices arising in the projection we mark the projection $p(W_i)\cap p(W_k)$ to indicate if the wires over or under crossed in $\beta^\ast \mathcal{W}$ as in Figure \ref{fig:VirtualVertices}.  We call the first a positive braiding (or positive virtual vertex) and the second a negative braiding (or negative virtual vertex).
\begin{rem} As in the \cite{suciu}, we read the wiring diagram from right to left.
\end{rem}

\begin{figure}[h]
        \centering
      \begin{tikzpicture}
\draw(4.7,0)--(6,0)node [anchor=west] {$W_1$};
\draw (4.4,1)  -- (4.7,0);
\draw(4.1,1)--(4.4,1);
\draw (3.9,1.5)  -- (4.1,1);
\draw (3.1,1.5)--(3.9,1.5);
\draw(2.9,2)  -- (3.1,1.5);
\draw(0.4,2) node[anchor=east]{$W_1$} --(2.9,2);

\draw (2.1,.5)--(6,.5)node [anchor=west]{$W_2$} ;
\draw (1.9,1)  -- (2.1,.5);
\draw (1.1,1)--(1.9,1);
\draw(0.9,1.5)  -- (1.1,1);
\draw(0.4,1.5)node[anchor=east]{$W_2$}--(0.9,1.5);

\draw (5.5,1)--(6,1)node [anchor=west]{$W_3$};
\draw (5.2,2)  -- (5.5,1);
\draw (3.1,2)  -- (5.2,2);
\draw (2.9,1.5)  -- (3.1,2);
\draw (1.1,1.5)--(2.9,1.5);
\draw (.9,1)  -- (1.1,1.5);
\draw(0.4,1) node[anchor=east]{$W_3$} --(.9,1);

\draw(4.1,1.5)--(6,1.5)node[anchor=west]{$W_4$} ;
\draw(3.9,1)  -- (4.1,1.5);
\draw (2.1,1)--(3.9,1);
\draw(1.9,.5)  -- (2.1,1);
\draw (0.4,.5)node[anchor=east]{$W_4$}--(1.9,.5);

\draw(5.5,2)--(6,2)node[anchor=west]{$W_5$};
\draw(5.2,1)--(5.5,2);
\draw(4.7,1)--(5.2,1);
\draw(4.4,0)--(4.7,1);
\draw(0.4,0)node[anchor=east]{$W_5$}--(4.4,0);

\node at (3.2,-.2) {$\downarrow$};     

\draw(0.4,-.5) node[anchor=east]{$1$} -- (6,-.5)node[anchor=west]{$0$};
\draw(0.4,-.4) -- (0.4,-.6);
\draw(6,-.4) -- (6,-.6);

\end{tikzpicture}
        \caption{Wiring Diagram}
        \label{fig:Wiring}
\end{figure}
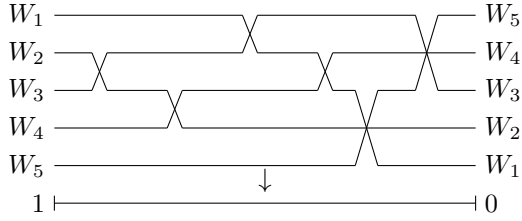

\begin{exmp}\label{ex:5linesWiring} Let $(z_1:z_2:z_3)$ be homogeneous coordinates of $\mathbb{P}^2$. Consider the arrangement consisting of two transverse pairs of parallel lines in $\mathbb{C}^2\cong \mathbb{P}^2 \setminus \{z_3=0\}$, defined by the equation $(z_2-z_1)(z_2-z_1+z_3) (z_2+z_1)(z_2+z_1-z_3)z_3=0$. The wiring diagram associated to this arrangement is shown in Figure \ref{fig:Wiring}. There are no virtual vertices since the arrangement is real and $\beta$ is a real segment.
\end{exmp}

\begin{rem} When no under or overcrossing is marked in a wiring diagram $\mathcal{W}$, it coincides with the notion of wiring diagram in \cite{matroids}. They are in correspondence with arrangement of "pseudo-lines", in particular there exists a wiring diagram of $9$ wires that does not comes from an arrangement of lines (the so called non-Pappus arrangement, see \cite[Proposition 8.3.1]{matroids}), however for $8$-wires or less they are in correspondence with the real arrangement of lines \cite[Thm 6.3.1]{matroids}.
\end{rem}

\subsection{Using the diagram to obtain presentations}

\subsubsection{Algorithm for computing a presentation of the fundamental group of $M(\A)$} \label{ss:AlgorithmComplement}

\medskip

We will use the following well-known Lemma.
\begin{lem}\label{lem:IsoFunda} Let $Z\subset X$ be an algebraic  subvariety of an algebraic smooth surface $X$. Fix a point $R\in X\setminus Z$. Denote by $\pi_R:\Bl_R X\to X$, then $\pi_1(X\setminus Z)\cong \pi_1(\Bl_R X \setminus \pi_R^*Z)$.
\end{lem}


This allow us to compute $\pi_1(M(\A))$ in the total space of the fiber bundle $\bar{f}:\Bl_R\mathbb{P}^2\to \mathbb{P}^1$. We will find suitable subspaces of the total space of this fiber bundle to apply the van-Kampen Theorem.

Let $\mathcal{W}\subset \Bl_R\mathbb{P}^2$ be a wiring diagram. Let $\beta^*(\mathcal{W})\subset [0,1]\times \mathbb{P}^1$ be as in \ref{ss:PlanarWiring}. Every vertical line $t\times \mathbb{P}^1$ in $[0,1]\times \mathbb{P}^1$ corresponds to the fiber $\bar{f}^{-1}{(\beta(t))}$. Recall that if $p\in \bar{f}^{-1}{(\beta(t_p))}$ for $p\in \Sing \A$ and $t_p\in [0,1]$, then no other point in $\Sing \A$ lies in the same fiber. Suppose that there are $s$ points $t_{p_1},\ldots,t_{p_s}$ corresponding to $p_1,\ldots, p_s$ in $\Sing \A$. 

By fixing a planar representation $p(\beta^*(\mathcal{W}))$ of $\beta^*(\mathcal{W})$ as in \ref{ss:PlanarWiring}, some under or over-crossing can arise. As the projection is generic, they correspond to a finite number $t_1',\ldots, t_\nu'$ of elements of $[0,1]$ distinct from the $t_{p_r}$.

Order the set $\{t_{p_1},\ldots,t_{p_s},t_1',\ldots,t_{\nu}'\}$ by increasing order and relabel them by $t_\kappa$ for $\kappa=1,\ldots, \nu+s$. Let $B_\kappa\subset \mathbb{P}^1$ be a neighborhood of $\beta(t_\kappa)$ homeomorphic to a disk in $\mathbb{C}$ such that $B_\kappa\cap B_j=\varnothing$ if $\abs{\kappa-j}>1$ and $B_\kappa\cap B_{\kappa+1}$ is homeomorphic to a disk. Consider $M_\kappa:=\bar{f}^{-1}(B_\kappa)\subset \Bl_R(\mathbb{P}^2)$ for  $\kappa=1,\ldots, \nu+s$ and denote by $M_\kappa(\A):=M_\kappa\setminus M_\kappa\cap \pi_R^*\A$.

\begin{lem}\label{lem:FreeGroup}
We have that 
$$\pi_1(M_\kappa(\A)\cap M_{\kappa+1}(\A))\cong F_n \text{ for } \kappa=1,\ldots, \nu+s-1 $$ 
with $F_n$ the free group in $n$ generators.
\end{lem}
\begin{proof}
First note that as $B_\kappa\cap B_{\kappa+1}\subset \mathbb{P}^1\setminus \{\bar{f}(p)\mid p\in \Sing \A\}$ we have that
$ M_\kappa(\A)\cap M_{\kappa+1}(\A)=\bar{f}^{-1}(B_\kappa\cap B_{\kappa+1})$
is the restriction of a fiber bundle to a contractible base. The fundamental group of any fiber in $B_\kappa\cap B_{\kappa+1}$ is a free group in $n$ generators.
\end{proof}

\begin{prop}\label{prop:VanKampen} We have that 
$$\pi_1(M(\A))\cong\pi_1(M_1(\A))\underset{\pi_1(M_1(\A)\cap M_2(\A))}{*}\cdots \underset{\pi_1(M_{\nu+s-1}(\A)\cap M_{\nu+s}(\A))}{*}\pi_1( M_{\nu+s}(\A)). $$
\end{prop}
\begin{proof}
By Lemma \ref{lem:IsoFunda}, we have that the morphism $\Bl_R\mathbb{P}^2\setminus \pi_R^* \A\to M(\A)=\mathbb{P}^2\setminus \A$ induces an isomorphism in the fundamental groups. 

Denote the restriction of $\bar{f}$ to $\Bl_R\mathbb{P}^2\setminus \pi_R^*\A$ by  $f:\Bl_R\mathbb{P}^2\setminus \pi_R^*\A\to \mathbb{P}^1$. Let $\infty\in \mathbb{P}^1\setminus \cup_{\kappa=1}^{\nu+s} B_\kappa$ and note that $f^{-1}(\mathbb{P}^1\setminus \{\infty\})$ is the complement in $\Bl_R\mathbb{P}^2\setminus \pi_R^*\A$ of a smooth irreducible divisor $D_\infty$ that is the restriction to $\Bl_R\mathbb{P}^2\setminus \pi_R^*\A$ of the strict transform of a line in $\mathbb{P}^2$ passing through $R$ .

By Proposition \ref{prop:meridians}, we have that 
$$\pi_1(\Bl_R\mathbb{P}^2\setminus \pi_R^* \A)=\pi_1\left(f^{-1}(\mathbb{P}^1\setminus \{\infty\})\right)/\langle\langle\gamma_{D_\infty}\rangle\rangle	 $$
where $\gamma_{D_\infty}$ is a meridian around $D_\infty$.

Note that, as $R\in \mathbb{P}^2\setminus \A$, we have that $\pi_R^{-1}(R)\subset \Bl_R\mathbb{P}^2\setminus \pi_R^*\A$ and its restriction to $f^{-1}(\mathbb{P}^1\setminus\{\infty\})$ is isomorphic to $\mathbb{C}$. The meridian $\gamma_{D_\infty}$ can be chosen to lie inside this restriction and therefore $\gamma_{D_\infty}=1$. We obtain that $\pi_1(f^{-1}(\mathbb{P}^1\setminus \{\infty\}))\cong \pi_1(M(\A))$.

Observe that $\cup_{\kappa=1}^{\nu+s}M_\kappa(\A)$ has the same homotopy as $(\Bl_R\mathbb{P}^2\setminus \A)\setminus f^{-1}(\infty)$. We conclude by successive applications of the van-Kampen Theorem: by construction $B_1\cap B_{\nu+s}=\varnothing$, we obtain that $\pi_1(\cup_{\kappa=1}^{\nu+s}M_\kappa(\A))$ is isomorphic to $$\pi_1(M_1(\A))\underset{\pi_1(M_1(\A)\cap M_2(\A))}{*}\cdots \underset{\pi_1(M_{\nu+s-1}(\A)\cap M_{\nu+s}(\A))}{*}\pi_1( M_{\nu+s}(\A))$$ 

\end{proof}

We want to compute now $\pi_1(M_\kappa(\A))$ for $\kappa=1,\ldots, \nu+s$ and the morphisms of amalgamation 
$\pi_1(M_\kappa(\A))\leftarrow \pi_1(M_\kappa(\A)\cap M_{\kappa+1}(\A))\to \pi_1(M_{\kappa+1}(\A)) $. In fact, if no point of $\Sing \A$ lies in $M_\kappa(\A)$ we will have that $\pi_1(M_\kappa(\A))\cong F_n$. However, some conjugations may arise in the meridians due to braiding of the wires in $\mathcal{W}$.

 We have to distinguish $3$ cases depending in the nature of $M_\kappa$: $M_\kappa$ contains a point of $\Sing \A$, it contains a \emph{positive braiding} of $\mathcal{W}$ or it contains a \emph{negative braiding}.

Let $\theta_\kappa<t_\kappa$ be sufficiently close so that $\beta(\theta_\kappa)\in B_\kappa\cap B_{\kappa-1}$ and denote by $x_1^{(\kappa)}, \ldots, x_{n+1}^{(\kappa)}$ the set of points in the planar representation $p(\beta^*(\mathcal{W}))$ of the wiring diagram $\mathcal{W}$ labeled from bottom to top corresponding to the points in $\bar{f}^{-1}{(\beta(\theta_\kappa)})\cap\A$.
\begin{defn}\label{def:generating set} A \emph{geometric generating set} $\Gamma^{(\kappa)}=\{\lambda_1^{(\kappa)},\ldots,\lambda_{n+1}^{(\kappa)}\}$ of the group $\pi_1(\bar{f}^{-1}{(\beta(\theta_\kappa)})\setminus (\bar{f}^{-1}{(\beta(\theta_\kappa)})\cap\A),q_\kappa)$ with $q_\kappa=\pi_R^{-1}(R)\cap \bar{f}^{-1}{(\beta(\theta_\kappa)})$ is the datum of $\lambda_1^{(\kappa)}, \ldots,$ $\lambda_{n+1}^{(\kappa)}$ meridians around $x_1^{(\kappa)},\ldots, x_{n+1}^{(\kappa)}$ respectively, all of them based at $q_\kappa$ such that $\lambda_{n+1}^{(\kappa)}\cdots\lambda_1^{(\kappa)}$ is nullhomotopic in $\bar{f}^{-1}{(\beta(\theta_\kappa)})\setminus \{x_1^{(\kappa)},\ldots,x_{n+1}^{(\kappa)}\}\cong \mathbb{P}^1\setminus \{(n+1)-\text{points}\}$.
\end{defn}

\begin{rem} A geometric generating set $\Gamma^{(\kappa)}=\{\lambda_1^{(\kappa)},\ldots,\lambda_{n+1}^{(\kappa)}\}$ induces a geometric base ${\Gamma^{(\kappa)}}'=\{\lambda_1^{(\kappa)},\ldots,\lambda_n^{(\kappa)}\}$ of $\pi_1(\mathbb{C}\setminus\{x_1^{(\kappa)},\ldots,x_n^{(\kappa)}\},q_\kappa)$.	
\end{rem}
 
  We consider here the geometric generating set  $\Gamma^{(\kappa)}=\{\lambda_1^{(\kappa)},\ldots, \lambda_{n+1}^{(\kappa)}\}$ as in figure \ref{fig:elgeobas}. As $\pi_1(\pi_R^{-1}(R))$ is trivial, we can fix a point $q\in \pi_R^{-1}(R)$ as a global base point for all the geometric generating set $\Gamma^{(\kappa)}$ with $\kappa=1, \ldots , \nu+s$ by joining $q_\kappa$ to $q$ by a simple path in $\pi_R^{-1}(R)$.

\begin{figure}[t]
\begin{subfigure}[b]{0.3\textwidth}
        \centering
\begin{tikzpicture} 
\draw (-.4,-1) node [left]{$\lambda_j^{(\kappa)}$};
\draw (-.4,0) node [left]{$\lambda_{j+1}^{(\kappa)}$};
\draw (-.4,1) node [left]{$\lambda_m^{(\kappa)}$};

\draw (0,-1) circle (.3);

\draw (0,-1.5) -- (0,-1.3);

\draw [arc arrow=to pos 0.35 with length 2mm] (-.5,0) to[out=-90,in=-90] 
(.5,0) [arc arrow=to pos 0.85 with length 2mm] 
to[out=90,in=90] cycle;
\draw  (0,-1.5) to[out=0,in=0] (.5,1);

\draw [arc arrow=to pos 0.35 with length 2mm] (-.5,1) to[out=-90,in=-90] 
(.5,1) [arc arrow=to pos 0.85 with length 2mm] 
to[out=90,in=90] cycle;
\draw  (0,-1.5) to[out=0,in=-90] (.5,0);

\foreach \Point in {(0,-1),(0,0), (0,1)}{
    \node at \Point {\textbullet};
}


\end{tikzpicture}
\caption{$\Gamma^{(\kappa)}$ in $\bar{f}^{-1}{(\beta(\theta_\kappa)})$}
\label{fig:elgeobas}        
    \end{subfigure}
~ ~ ~~~   
    \centering
     \begin{subfigure}[b]{0.3\textwidth}
        \centering
     
     \begin{tikzpicture} 

\draw (-.1,-1.5) node {\small $\lambda_m^{(\kappa)}$};
\draw (-.1,-0.5) node {\small $\lambda_{j+1}^{(\kappa)}$};
\draw (-.1,.5) node {\small $\lambda_j^{(\kappa)}$};

\draw (0,-1) circle (.3);

\draw  (.5,-1.5) to[out=20,in=-20] (.4,.8);

\draw [arc arrow=to pos 0.1 with length 2mm] (-.5,0) to[out=-90,in=-90] 
(.5,0) [arc arrow=to pos 0.6 with length 2mm] 
to[out=90,in=90] cycle;
\draw  (.5,-1.5) to[out=0,in=0] (0,2);
\draw  (0,2) to[out=180,in=180] (-.3,-1);

\draw [arc arrow=to pos 0.1 with length 2mm] (-.5,1) to[out=-90,in=-90] 
(.5,1) [arc arrow=to pos 0.6 with length 2mm] 
to[out=90,in=90] cycle;
\draw  (.5,-1.5) to[out=0,in=0] (0,1.6);
\draw  (0,1.6) to[out=180,in=180] (-.5,0);

\foreach \Point in {(0,-1),(0,0), (0,1)}{
    \node at \Point {\textbullet};
}


\end{tikzpicture}
\caption{$\Gamma^{(\kappa)}$ in $\bar{f}^{-1}{(\beta(\theta_{\kappa+1})})$ .}
\label{fig:DoubleBase}
    \end{subfigure}
    \caption{Geometric generating set in different fibers	}\label{fig:DifConj}
\end{figure}
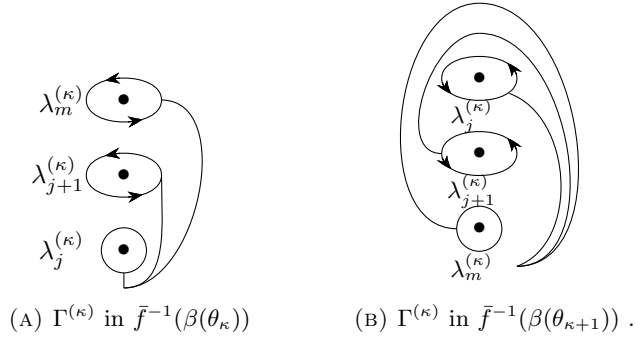

   We des\-cri\-be how the meridians change when we move the generators of $\Gamma^{(\kappa)}$ to the fiber $\bar{f}^{-1}{(\beta(\theta_{\kappa+1})})$ and express them in the generators $\Gamma^{(\kappa+1)}$, see figure \ref{fig:DoubleBase}. We record as well the relations arising in between. 
 
  Suppose that $p\in \Sing \A\cap  M_\kappa$, and let $\Gamma^{(\kappa)}$ be as above. Denote by $j$ the first index of the meridians of $\Gamma^{(\kappa)}$ corresponding to a line passing through $p$, and by $m$ the last. We have that $\lambda_k^{(\kappa+1)}=\lambda_k^{(\kappa)}$ for $k<j$ and $k>m$ as we can deform continuously $\lambda_k^{(\kappa)}$ to $\lambda_k^{(i+1)}$ having the same homotopy type in $\pi_1(M_\kappa(\A))$. 
  
Let $R_\kappa=[\lambda_m^{(\kappa)},\lambda_{m-1}^{(\kappa)},\ldots,\lambda_j^{(\kappa)}]$ denote the set of equations of the form $\lambda_m^{(\kappa)}\cdot \lambda_{m-1}^{(\kappa)}\cdots\lambda_j^{(\kappa)}= \lambda_{\sigma(m)}^{(\kappa)}\lambda_{\sigma(m-1)}^{(\kappa)}\cdots\lambda_{\sigma(j)}^{(\kappa)}$ where $\sigma$ varies in the set of cyclic permutations in $m-j+1$ elements.

\begin{figure}[h]
    \centering
      \begin{tikzpicture}
      
\node at (0,0) [anchor=north] {$p$};

\draw (-1,.75) node [anchor=east] {$a$}    -- (1,-.75) node [anchor=west] {$a$};
\draw (-1,0.25) node [anchor=east] {$b^a$} -- (1,-.25)node [anchor=west] {$b$} ;
\draw (-1,-.25) node [anchor=east] {$c^{ba}$}  -- (1,.25)node [anchor=west] {$c$};
\draw (-1,-.75) node [anchor=east] {$d^{cba}$}  -- (1,.75)node [anchor=west] {$d$};

\draw [dashed,->] (1,-1.1) node [anchor=north] {$\Gamma^{(\kappa)} $} -- (1,1.32);
\draw [dashed,->] (-1,-1.1) node [anchor=north] {$\Gamma^{(\kappa+1)} $} -- (-1,1.32);

\end{tikzpicture}
        \caption{Actual vertex}
   \label{fig:ActualVertices}
\end{figure}
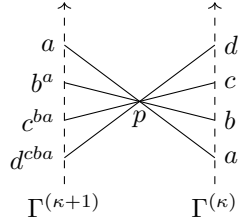

\begin{lem}\label{lem:ActualVertex} Let $p\in \Sing \A\cap M_\kappa$. Then $\pi_1(M_\kappa(\A),q_\kappa)$ is generated by the elements of $\Gamma^{(\kappa)}$ and $\Gamma^{(\kappa+1)}$ together with the relations $R_\kappa$, $\lambda_k^{(\kappa+1)}=\lambda_k^{(\kappa)}$ for $k<j$ or $m<k$, $\lambda_{n+1}^{(\kappa)}\cdots\lambda_1^{(\kappa)}=1$ and
\begin{align*}\label{eq:ConjugationActual}
&\lambda_m^{(\kappa+1)}=\lambda_j^{(\kappa)},\\
&\lambda_{m-1}^{(\kappa+1)}={\lambda_{j+1}^{(\kappa)}}^{{\lambda_j}^{(\kappa)}},\\
& \lambda_{m-2}^{(\kappa+1)}={\lambda_{j+2}^{(\kappa)}}^{{\lambda_{j+1}}^{(\kappa)}{\lambda_j}^{(\kappa)}},\\
&\vdots\\
&\lambda_j^{(\kappa+1)}={\lambda_m^{(\kappa)}}^{\lambda_{m-1}^{(\kappa)}\cdots\lambda_j^{(\kappa)}}.
\end{align*}
\end{lem}
(see figures \ref{fig:DifConj} and \ref{fig:ActualVertices}.)
\begin{proof}
Let $V_p$ be a neighborhood around $p$ homeomorphic to a product $B_\kappa\times D$ with $D$ a disk not intersecting $L_k\in \A$ with $k<j$ or $k>m$.  The local fundamental $\pi_1(V_p\setminus \A)$ equals the fundamental group of the link associated to the singularity $p$ which is a Hopf link of $m-j+1$ circles (see \cite[Lemma 5.75]{Orlik}).

For the complement $M_\kappa(\A)\setminus V_p$ we have $\pi_1(M_\kappa(\A)\setminus V_p)\cong F_{n-(m-j)}$ and if $V'$ is a small neighborhood of $V_p$ we have that as $V_p\setminus \A$ retracts to $\partial V_p \setminus \A$ then $\pi_1((M_\kappa(\A)\setminus V_p)\cap V')\cong \pi_1(V_p\setminus \A)$. By van-Kampen we obtain the relation $\lambda_{n+1}^{(\kappa)}\cdots\lambda_1^{(\kappa)}=1$.
\end{proof}

\begin{lem}\label{lem:PosVirtual} Suppose that there is a \emph{positive braiding} of the wires $j$ and $j+1$ in $M_\kappa(\A)$. Then the group $\pi_1(M_\kappa(\A),q_\kappa)$ admits the presentation
$$\left\langle \lambda_1^{(\kappa)}, \ldots, \lambda_{n+1}^{(\kappa)}, \lambda_j^{(\kappa+1)}, \lambda_{j+1}^{(\kappa+1)} \mid \lambda_{j+1}^{(\kappa+1)}=\lambda_j^{(\kappa)},\lambda_j^{(\kappa+1)}={\lambda_{j+1}^{(\kappa)}}^{\lambda_j^{(\kappa)}} \right\rangle $$
(See fig. \ref{fig:PosBraiding}.)
\end{lem}
\begin{proof}
As in lemma \ref{lem:ActualVertex}, we have that we can deform $\lambda_{k}^{(\kappa+1)}$ to $\lambda_k^{(\kappa)}$ for $k<j$ or $j+1<k$ without changing the homotopy type.

The result follows from the Wirtinger presentation of a braid interchanging the $j$ and the $j+1$ wire: consider the meridians $\lambda_{j}^{(\kappa)},\lambda_{j+1}^{(\kappa)}$ in $\bar{f}^{-1}(\beta(\theta_{\kappa+1}))$ as in figure \ref{fig:DoubleBase}. Note that in $\pi_1(\bar{f}^{-1}(\beta(\theta_{\kappa+1})),q_{\kappa+1})$ these meridians satisfy the relations: 
$$\lambda_{j+1}^{(\kappa+1)}=\lambda_j^{(\kappa)}, \quad \lambda_j^{(\kappa+1)}={\lambda_{j+1}^{(\kappa)}}^{\lambda_{j+1}^{(\kappa+1)}}={\lambda_{j+1}^{(\kappa)}}^{\lambda_j^{(\kappa)}}.$$
This can be seen directly from Figure \ref{fig:DoubleBase}. (C.f. \cite[Lemmas 5.73, 5.74]{Orlik}.)
\end{proof}

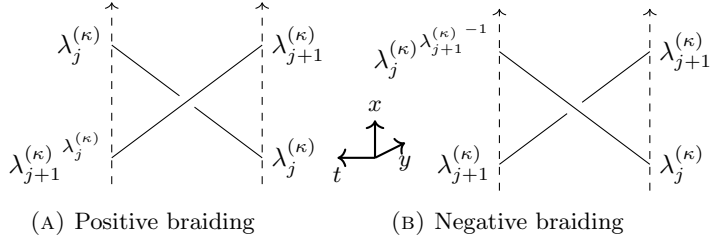
\begin{figure}[t]
\begin{subfigure}[b]{0.3\textwidth}
        \centering
         \begin{tikzpicture}
      \draw [thick, ->] (2.5,-0.75) -- (2,-0.75);
  \draw [thick, ->] (2.5,-0.75) -- (2.5,-.25);
    \draw [thick, ->] (2.5,-0.75) -- (2.9,-.55);      

  \node [below] at (2,-0.75) {\small $t$};
    \node [above] at (2.5,-.25) {\small $x$};
        \node [below] at (2.9,-.6) {\small $y$};

\draw (-1,0.75) node [anchor=east] {${\lambda_j^{(\kappa)}}$}    -- (-.1,.075);      
\draw (.1,-0.075)  -- (1,-.75) node [anchor=west] {$\lambda_j^{(\kappa)}$};      

\draw (-1,-.75) node [anchor=east] {${\lambda_{j+1}^{(\kappa)}}^{\lambda_j^{(\kappa)}}$}  -- (1,.75)node [anchor=west] {$\lambda_{j+1}^{(\kappa)}$};

\draw [dashed,->] (1,-1.1)  -- (1,1.32);
\draw [dashed,->] (-1,-1.1) -- (-1,1.32);
\end{tikzpicture} 
        \caption{Positive braiding}
        \label{fig:PosBraiding}
    \end{subfigure}
~ ~ ~~~   
    \centering
     \begin{subfigure}[b]{0.3\textwidth}
        \centering
     
      \begin{tikzpicture}

\draw (-1,.75) node [anchor=east] {${\lambda_j^{(\kappa)}}^{{\lambda_{j+1}^{(\kappa)}}^{-1}}$}    --(1,-.75) node [anchor=west] {$\lambda_j^{(\kappa)}$};      

\draw (-1,-.75) node [anchor=east] {${\lambda_{j+1}^{(\kappa)}}$}  -- (-.1,-0.075);
\draw (.1,0.075) -- (1,.75)node [anchor=west] {$\lambda_{j+1}^{(\kappa)}$};

\draw [dashed,->] (1,-1.1)  -- (1,1.32);
\draw [dashed,->] (-1,-1.1)  -- (-1,1.32);

\end{tikzpicture}
        \caption{Negative braiding}
        \label{fig:NegativeBraiding}
    \end{subfigure}
    \caption{Braiding in $\mathcal{W}$}\label{fig:VirtualVertices}
\end{figure}

\begin{lem}\label{lem:NegVirtual} Suppose that there is a \emph{negative braiding} in $M_\kappa(\A)$, then the group $\pi_1(M_\kappa(\A),q_\kappa)$ admits the presentation
$$\left\langle \lambda_1^{(\kappa)}, \ldots, \lambda_{n+1}^{(\kappa)}, \lambda_j^{(\kappa+1)}, \lambda_{j+1}^{(\kappa+1)} \mid \lambda_{j+1}^{(\kappa+1)}={\lambda_j^{(\kappa)}}^{{\lambda_{j+1}^{(\kappa)}}^{-1}},
\lambda_{j}^{(\kappa+1)}=\lambda_{j+1}^{(\kappa)} \right\rangle $$
(See fig. \ref{fig:NegativeBraiding}.)
\end{lem}

We can summarize the information carried by a wiring diagram $\mathcal{W}$ and the changes in the geometric sets $\Gamma^{(\kappa)}$ as they cross a vertex in $\mathcal{W}$ as follows. 

 For every $t_\kappa\in\{t_1,\ldots, t_{\nu+s}\}$ there exists a crossing $p_\kappa$ in the planar representation of $\mathcal{W}$, let $\Pi^{(\kappa)}=\{\sigma^{(\kappa)}(1)<\ldots<\sigma^{(\kappa)}(n+1)\}$ be an ordered set, with $\sigma^{(\kappa)}$ a permutation of $\{1,\ldots,n+1\}$ such that the $k$-th element $\sigma^{(\kappa)}(k)$ records the position of the wire $W_{\sigma^{(\kappa)}(k)}$ in the fiber $\bar{f}^{-1}{(\beta(\theta_\kappa)})$, when $\mathcal{W}$ is read from bottom to top, with $\theta_\kappa$ as in Definition \ref{def:generating set}. This is, $x_k^{(\kappa)}\in W_{\sigma^{(\kappa)}(k)}$ for $k=1,\ldots, n+1$. Note that $\sigma^{(1)}=id$.

 The order in $\Pi^{(\kappa)}$  records the \emph{local position} of the wires of $\mathcal{W}$ in $\bar{f}^{-1}{(\beta(\theta_\kappa)})$, while the order $\{1,\ldots,n+1\}$ induced from the order of the lines in $\A$ is a \emph{global order}. For a wire $W_k$ of $\mathcal{W}$, we write ${\sigma^{(\kappa)}}^{-1}(k)$ to indicate that the wire $W_k$ is in the ${\sigma^{(\kappa)}}^{-1}(k)$  position in the fiber $\bar{f}(\beta(\theta_\kappa))$.
 
 Consider the free group $F_{n+1}^{(\kappa)}$ generated by the meridians in $\Gamma^{(\kappa)}$ and let $\tau^{(\kappa)}:\{1,\ldots, n+1\}\to F_{n+1}^{(\kappa)}$ defined as follows: 

Suppose that the crossing $p_\kappa$ corresponding to $t_\kappa$ satisfies $p_\kappa=W_{\sigma^{(\kappa)}(j)}\cap W_{\sigma^{(\kappa)}(j+1)}\cap \ldots \cap W_{\sigma^{(\kappa)}(m)}$, then $$\tau^{(\kappa)}(k)=\left\lbrace \begin{array}{ll}
e  & \text{ for } k=1,\ldots,j,m+1,\ldots,n+1,\\
\lambda_{k-1}^{(\kappa)}\cdots\lambda_j^{(\kappa)}   & \text{ for } j<k\leq m,

\end{array} \right. $$  if $t_\kappa$ is an actual vertex, 
$$\tau^{(\kappa)}(k)=\left\lbrace \begin{array}{ll}
e  & \text{ for } k=1,\ldots,j,j+2,\ldots,n+1,\\
\lambda_{j}^{(\kappa)}   & \text{ for } k=j+1,

\end{array} \right. $$
if $t_\kappa$ is a positive virtual vertex, and 
$$\tau^{(\kappa)}(k)=\left\lbrace \begin{array}{ll}
e  & \text{ for } k=1,\ldots,j-1,j+1,\ldots,n+1,\\
(\lambda_{j+1}^{(\kappa)})^{-1}   & \text{ for } k=j,

\end{array} \right. $$
if $t_\kappa$ is a negative virtual vertex. 

The  Lemmas \ref{lem:ActualVertex}, \ref{lem:PosVirtual} and \ref{lem:NegVirtual} imply the following proposition.
\begin{prop}\label{prop:ConjTogether} Let $\Gamma^{(\kappa)}=\{\lambda_1^{(\kappa)},\ldots, \lambda_{n+1}^{(\kappa)}\},\Gamma^{(\kappa+1)}=\{\lambda_1^{(\kappa+1)},\ldots,\lambda_{n+1}^{(\kappa+1)}\}$ be geometric generating set as in \ref{ss:AlgorithmComplement} and suppose that $p_\kappa\in M_\kappa$. Then we have that in $\pi_1(M_\kappa(\A),q_\kappa)$:
$$\lambda_{{\sigma^{(\kappa+1)}}^{-1}(\sigma^{(\kappa)}(k))}^{(\kappa+1)}=(\lambda_k^{(\kappa)})^{\tau^{(\kappa)}(k)}\quad \text{ for } k=1,\ldots, n+1, $$
or equivalently,
$$\lambda_k^{(\kappa+1)}=(\lambda_{{\sigma^{(\kappa)}}^{-1}(\sigma^{(\kappa)}(k))}^{(\kappa)})^{\tau^{(\kappa)}\left({\sigma^{(\kappa)}}^{-1}(\sigma^{(\kappa+1)}(k))\right)} \quad  \text{ for } k=1,\ldots, n+1. $$
\end{prop}
Note that if $p_\kappa=W_{\sigma^{(\kappa)}(j)}\cap \ldots\cap W_{\sigma^{(\kappa)}(m)}$ we have that
$${\sigma^{(\kappa+1)}}^{-1}(\sigma^{(\kappa)}(k))= \left\lbrace \begin{array}{ll}
k & \text{for } k=1, \ldots,j-1,m+1,\ldots,n+1,\\
m-\iota & \text{for } k=j+\iota \text{ and } \iota=0, \ldots, m-j.
\end{array} \right. $$

As the fundamental group of $M(\A)$ is generated by the meridians around each line, we fix the geometric generating set $\Gamma^{(1)}=\{\lambda_1^{(1)},\ldots,\lambda_{n+1}^{(1)}\}=\{\lambda_1,\ldots, \lambda_{n+1}\}\subset M_1(\A)$.

\begin{thm}\label{thm:Arvola}
Let $\A=\{L_1,\ldots, L_{n+1}\}$ be a complex arrangement of lines in $\mathbb{P}^2$ and let $\Gamma^{(1)}$ be a geometric generating set as above. A
 presentation for the fundamental group of $M(\A)$ is given by
$$\pi_1(M(\A),q)=\left\langle \lambda_1,\ldots,\lambda_{n+1} \mid \bigcup_\kappa R_\kappa, \lambda_{n+1}\cdots \lambda_1 \right\rangle$$
with $R_\kappa$ as Lemma \ref{lem:ActualVertex} and each $\kappa$ corresponding to a point $p_\kappa\in\Sing \A$.
\end{thm}
\begin{rem}\label{rem:GammaKinGamma1} The relations $R_\kappa$ are  expressed in terms of the geometric generating set $\Gamma^{(1)}$ by substituting $\lambda_k^{(\kappa)}$ by a conjugate of $\lambda_{\sigma^{(\kappa)}(k)}^{(1)}$ by elements of $\Gamma^{(1)}$ by repeated applications of Proposition \ref{prop:ConjTogether}.
\end{rem}
\begin{proof}
From Proposition \ref{prop:VanKampen} we know that $\pi_1(M(\A))=\pi_1(M_1(\A))*_{F_n}\cdots*_{F_n}\pi_1(M_{\nu+s}(\A))$. Now, the groups $\pi_1(M_\kappa(\A))$ are presented in generators $\Gamma^{(\kappa)}$ and $\Gamma^{(\kappa+1)}$, and relations which are words in these letters (see Lemmas \ref{lem:ActualVertex}, \ref{lem:PosVirtual}, \ref{lem:NegVirtual}). The geometric generating set $\Gamma^{(\kappa)}$ is chosen in such a way that it lies in a fiber over a point of $B_{\kappa-1}\cap B_\kappa$, and therefore, we can assume that the amalgamation $\pi_1(M_{\kappa-1}(\A))*_{F_n}\pi_1(M_\kappa(\A))$ permits to see $\Gamma^{(\kappa)}$ in $M_\kappa(\A)$ and $M_{\kappa-1}(\A)$ simultaneously.

Note that $\lambda_m^{(\kappa+1)}\cdots \lambda_j^{(\kappa+1)}=\lambda_m^{(\kappa)}\cdots\lambda_j^{(\kappa)}$, hence $\lambda_{n+1}^{(\kappa+1)}\ldots \lambda_1^{(\kappa+1)}=\lambda_{n+1}^{(\kappa)}\cdots \lambda_1^{(\kappa)}$ for every $\kappa=1,\ldots, \nu+s-1$.

The relations in $\pi_1(M_\kappa(\A))$ when there is a positive or virtual vertex in $M_\kappa(\A)$, can be omitted in the presentation of $\pi_1(M(\A))$ by writing every meridian of $\Gamma^{(\kappa+1)}$ in terms of $\Gamma^{(\kappa)}$ as in Lemmas \ref{lem:PosVirtual}, \ref{lem:NegVirtual}.

When there is an actual vertex in $M_\kappa(\A)$, the relation $R_\kappa=[\lambda_m^{(\kappa)},\lambda_{m-1}^{(\kappa)},\ldots,\lambda_j^{(\kappa)}]$ will appear in the presentation of $\pi_1(M(\A))$. This relation can be expressed in terms of $\Gamma^{(1)}$ in a recursive way, by expressing $\Gamma^{(\kappa)}$ in terms of $\Gamma^{(\kappa-1)}$ by using the amalgamation of $\pi_1(M_{\kappa-1}(\A))$ and $\pi_1(M_\kappa(\A))$ over $B_{\kappa-1}\cap B_\kappa$ and the presentation of $M_{\kappa-1}(\A)$ given by proposition \ref{prop:ConjTogether}. More precisely, we have that $$\lambda_k^{(\kappa)}=\lambda_{\sigma^{(\kappa)}(k)}^{\tau^{(1)}\left(\sigma^{(\kappa)}(k)\right)\cdot \tau^{(2)}\left({\sigma^{(2)}}^{-1}(\sigma^{(\kappa)}(k)\right)\cdots \tau^{(\kappa-1)}\left({\sigma^{(\kappa-1)}}^{-1}(\sigma^{(\kappa)}(k))\right)} \text{ for } k=1,\ldots, n+1,$$
and every $\tau^{(r)}\left({\sigma^{(r)}}^{-1}(\sigma^{(\kappa)}(k))\right)$ can be expressed in terms of $\Gamma^{(1)}$ in a recursive way for $r=1,\ldots, i-1$.
\end{proof}

\subsubsection{Algorithm for determining the presentation for a partial compactification  $M(\A,I,P)$ }\label{sss:AlgorithmPartial}

Let $\mathcal{W}$ be a wiring diagram and $\left\langle \lambda_1,\ldots,\lambda_{n+1} \mid \cup_k R_k, \lambda_{n+1}\cdots \lambda_1\right\rangle$ a presentation of $\pi_1(M(\A))$ as in Theorem \ref{thm:Arvola}.

Consider a partial compactification $M(\A,I,P)$ of $\mathbb{P}^2\setminus \A$ as in \ref{ss:partialCArrangements}. Here, we let $P_0=\{p_1,\ldots,p_{s_0}\}\subset \Sing \A$ denote the points of multiplicity strictly bigger than two, consider $\pi:\bar{X}=\Bl_{P_0}\mathbb{P}^2\to \mathbb{P}^2$ and denote by $D=\sum_{i=1}^{n+1+s_0}D_i=\pi^* \A$. Select $I\subset \{1,\ldots, n+1+s_0\}$ and $P=\{p_1',\ldots, p_{s_1}'\}\subset \Sing \sum_{i\not \in I} D_i$. Consider another blow-up $\pi':\Bl_{P} \bar{X}\to \bar{X}$ and write ${\pi'}^* D=\sum_{i=1}^{n+1+s_0+s_1}D_i'$. Define $D'= {\pi'}^*(D)-\sum_{i\in I}D_i' - \sum_{i>n+1+s_0} D_i'$ and $M(\A,I,P)=\Bl_{P_1}\bar{X}\setminus D'$.

 From Proposition \ref{prop:meridians}, we have that a presentation for the fundamental group $\pi_1(M(\A,I,P))$ can be obtained from $\left\langle \lambda_1,\ldots,\lambda_{n+1} \mid \cup_k R_k, \lambda_{n+1}\cdots \lambda_1\right\rangle$ by adding as relations certain words $\lambda(D_i')$ representing some meridians around the irreducible components $D_i'$ with either $i\in I$ or $n+s_0+1<i$. In order to do so, we have to distinguish four cases for these irreducible components $D_i'$ of ${\pi'}^* D$:
\begin{enumerate}
\item \label{en:Mer1} $D_i'$ is the strict transform of a line in $\A$. In this case $i\leq n+1$.
\item \label{en:Mer2} $D_i'$ is the strict transform of an exceptional divisor $D_i$ in $\bar{X}$. In this case $n+1<i\leq n+1+s_0$,
\item \label{en:Mer3} $D_p'$ is an exceptional divisor coming from a double point $p$ in $\Sing \A$. 
\item \label{en:Mer4} $D_p'$ is an exceptional divisor obtained by blowing-up a point $p=D_r\cap D_k$ with $r\leq n+1$ and $n+1<k\leq n+1+s_0$.
\end{enumerate}

For the lines as  in (\ref{en:Mer1}) we let $\lambda(D_i')=\lambda_i$.


For the lines as in (\ref{en:Mer2}), suppose that $D_i\subset \bar{X}$ is an exceptional divisor coming from a point $p\in \Sing \A$ and suppose that $p\in M_\kappa$, this is, $t_p$ is the $\kappa$-element in the ordered set of vertices $t_1,\ldots,t_{\nu+s}$ of a planar representation of $\mathcal{W}$ as in \ref{ss:AlgorithmComplement}. In other words $t_p=t_\kappa\in [0,1]$ satisfies $\beta(t_\kappa)=\bar{f}(p)$ and consider $\Gamma^{(\kappa)}=\{\lambda_1^{(\kappa)},\ldots,\lambda_{n+1}^{(\kappa)}\}$ the geometric generating set of $\bar{f}^{-1}{(\beta(\theta_\kappa)})\setminus (x_1^{(\kappa)},\ldots, x_{n+1}^{(\kappa)})\subset M_\kappa(\A)$ and suppose that $p=W_{\sigma^{(\kappa)}(j)}\cap W_{\sigma^{(\kappa)}(j+1)}\cap \ldots\cap W_{\sigma^{(\kappa)}(m)}$ with the local index $\Pi^{(\kappa)}=\{\sigma^{(\kappa)}(1)<\ldots<\sigma^{(\kappa)}(n+1)\}$ as in \ref{ss:AlgorithmComplement}. Associate to $D_i$ and to its strict transform $D_i'$, the word $\lambda(D_i)=\lambda(D_i')=\lambda_m^{(\kappa)}\cdot\lambda_{m-1}^{(\kappa)}\cdots \lambda_{j+1}^{(\kappa)}\cdot\lambda_j^{(\kappa)}. $

\begin{lem}\label{lem:ExcMerProduct} Let $D_i'$ be a line as in (\ref{en:Mer2}). Then $\lambda(D_i')=\lambda_m^{(\kappa)}\cdot\lambda_{m-1}^{(\kappa)}\cdots \lambda_{j+1}^{(\kappa)}\cdot\lambda_j^{(\kappa)}$ represents a meridian around $D_i$, and by pull-back, also around $D_i'$.
\end{lem}
\begin{proof} Let $\psi:U\to D$ and $\psi_i:U_i\to D_i$ be the boundary manifolds of $D$ and $D_i$ in $\bar{X}$ respectively. Note that we can use the meridians $\lambda_j^{(\kappa)},\ldots,\lambda_{m}^{(\kappa)}$ to give a presentation of $\pi_1(\partial U^*)$, with $\partial U_i^*=\partial U\cap \partial U_i$ as in \ref{ss:mumford}, as follows:
 the projection $\pi(\partial U_i)$ to $\mathbb{P}^2$ can be seen as the boundary of a $4$-real ball $B_p$ centered at $p$. There exists $R_p\in \partial B_p$ such that for each $j\leq k \leq m$ the loop $\alpha_k:=\lambda_{k}^{(\kappa)}$ is homotopic to a product $\alpha_{k_1}^{\alpha_{k_2}}$ with 
\begin{itemize}
\item The loop $\alpha_{k_1}$ starting at $R_p$, lying completely in $\partial B_p$ and surrounding the line $L_{\sigma^{(\kappa)}(k)}$.
\item The loop $\alpha_{k_2}$ is a simple path connecting $R_p$ and the point $R\in \mathbb{P}^2\setminus \A$.
 \end{itemize}
By pulling-back the meridians $\alpha_{j_1},\ldots,\alpha_{m_1}$ to $\bar{X}$ we can see them as lying in $\partial U$. By construction of the geometric generating set $\Gamma^{(\kappa)}$, the product $\alpha_{m_1}\cdots\alpha_{j_1}$ is homotopic to a path encircling the lines $L_{\sigma^{(\kappa)}(j)},\ldots,L_{\sigma^{(\kappa)}(m)}$ and therefore the projection ${\psi_i}_*(\alpha_{m_1}\cdots\alpha_{j_1})=e$ in $\pi_1(D_i^*,\psi_i(R_p))$. We can construct a continuous map $h_i:\cup_{k=j}^{m} \psi_i(\alpha_{k_1})\to \cup_{k=j}^{m} \alpha_{k_1}$ such that $h_i(\psi(\alpha_{k_1}))=\alpha_{k_1}$ and therefore the loops $\alpha_{j_1},\ldots, \alpha_{m_1}$ together with a fiber $\alpha_i$ of $\partial U_i$ generate the group $\pi_1(\partial U_i^*)$ as in Lemma \ref{lem:RelsInBundle}. Moreover, as $D_i\cdot D_i=-1$, we have the relation $\alpha_i=\alpha_{m_1}\cdots\alpha_{j_1}$ in $\pi_1(\partial U_i^*,R_p)$.

By construction of $\Gamma^{(\kappa)}$, we have that every two $\alpha_{k_2}$ and $\alpha_{k_2'}$ with $j\leq k,k'\leq m$ are homotopic. Therefore, by connecting $\alpha_i$ to $R$ via $\alpha_{j_2}$, we obtain the relation $\alpha_i^{\alpha_{j_2}}=\lambda_m^{(\kappa)}\cdots \lambda_j^{(\kappa)}$ in $\pi_1(\mathbb{P}^2\setminus \A)$.

By pulling-back $\alpha_i^{\alpha_{j_2}}$ to $\Bl_{P_1}$ we obtain that it is homotopic to a meridian around $D_i'$.
\end{proof}

For the lines $D_p'$ as in (\ref{en:Mer3}), suppose that $p=D_r\cap D_k$ with $r,k\leq n+1$. Consider the unique index $1\leq \kappa\leq \nu +s$ such that $p\in M_\kappa$ and let $\Gamma^{(\kappa)}=\{\lambda_1^{(\kappa)},\ldots, \lambda_{n+1}^{(\kappa)}\}$ be a geometric generating set of $\pi_1(M_\kappa(\A))$.  We denote $$\lambda(D_p'):=\lambda_{{\sigma^{(\kappa)}}^{-1}(r)}^{(\kappa)}\lambda_{{\sigma^{(\kappa)}}^{-1}(k)}^{(\kappa)}.$$ Recall that ${\sigma^{(\kappa)}}^{-1}(r)$ and ${\sigma^{(\kappa)}}^{-1}(k)$ record the local position of the wires $W_r, W_k$ respectively, in the local order of the wires of $\mathcal{W}$ in $\bar{f}^{-1}{(\beta(\theta_\kappa)})$ given by $\Pi^{(\kappa)}=\{\sigma^{(\kappa)}(1)<\ldots<\sigma^{(\kappa)}(n+1)\}$.

Finally, let $D_p'$ be as in (\ref{en:Mer4}) with $p\in P$. We have that $p=D_r\cap D_k$ with $r\leq n+1$ and $D_k$ an exceptional divisor coming from a point $p(k)\in P_0$. Let us suppose that $p(k)\in M_\kappa$. Denote by $\Gamma^{(\kappa)}=\{\lambda_1^{(\kappa)},\ldots,\lambda_{n+1}^{(\kappa)}\}\subset M_\kappa(\A)$ the geometric generating set as above. We can suppose that $p(k)=W_{\sigma^{(\kappa)}(j)}\cap \ldots\cap W_{\sigma^{(\kappa)}(m)}$. As $n+1\leq k\leq n+1+s_0$, we can consider the word $\lambda(D_{k})=\lambda_{m}^{(\kappa)}\cdots\lambda_{j}^{(\kappa)}$ as in Lemma \ref{lem:ExcMerProduct} above. 

\begin{lem}\label{lem:doubleBlowUp} A meridian of $D_p'$ is given by $\lambda(D_p')=\lambda_{{\sigma^{(\kappa)}}^{-1}(r)}^{(\kappa)}\lambda(D_{k})$. Moreover, $\lambda_{{\sigma^{(\kappa)}}^{-1}(r)}^{(\kappa)}$ commutes with $\lambda(D_k)$.
\end{lem}
\begin{proof}
Recall that by construction, $\lambda_{{\sigma^{(\kappa)}}^{-1}(r)}^{(\kappa)}$ is the meridian of $L_r$ lying in the geometric generating set $\Gamma^{(\kappa)}$.

Let $\psi_{D_{k}}:\partial U_{D_k} \to D_k$ be the boundary manifold of $D_k$ in $\bar{X}$. For $k'=j,\ldots, m$, let us decompose the loops $\alpha_{k'}=\lambda_{k'}^{(\kappa)}$ in two parts $\alpha_{k_1'}, \alpha_{k_2'}$, as in the proof of the Lemma \ref{lem:ExcMerProduct}, such that $\alpha_{k'}$ is homotopic to $\alpha_{k_1'}^{\alpha_{k_2'}}$. The proof of the same Lemma and \ref{lem:RelsInBundle} give us that
$$\pi_1(\partial U_{D_k}^*,R_k)=\langle \alpha_{j_1},\ldots, \alpha_{m_1}, \alpha_{k} \mid [\alpha_k, \alpha_{k_1'}],\alpha_{k}=\alpha_{m_1}\cdots \alpha_{j_1}  \rangle. $$
for a point $R_{k}\in \partial U_{D_{k}}^*$ and $\alpha_k$ a fiber of $\partial U_{D_k}^*$. We can globalize the relations in this presentation by considering $\alpha_{k_1'}^{\alpha_{k_2'}}$ and obtain that $\lambda(D_{k})$ commutes with $\lambda_{k'}^{(\kappa)}$ for $k'=j, \ldots, m$, in particular as $D_r$ intersect $D_k$, we have that $\lambda_{{\sigma^{(\kappa)}}^{-1}(r)}^{(\kappa)}$ commutes with $\lambda(D_k)$.

Furthermore, the point $R_k$ can be chosen to lie in the boundary $\partial B_{p}$ of a ball $B_{p}\subset \bar{X}$ around $p$. Let $\psi_{D_p'}:\partial U_{D_p'}\to D_{p}'$ be the boundary manifold of $D_{p}'$ in $\Bl_{P}\bar{X}$ and $\Delta_1, \Delta_2$ a pair of disks about the points $D_{p}'\cap D_r'$ and $D_{p}'\cap D_{k}'$ respectively. Denote $\partial U_{D_{p}'}^*=\psi_{D_{p}'}^{-1}(D_{p}'\setminus (\Delta_1\cup \Delta_2))$. By working in local coordinates, it can be seen that $\alpha_{k}, \alpha_{{\sigma^{(\kappa)}}^{-1}(r)_1}$ and a fiber $\alpha_{p}$ of $\partial U_{D_p'}$ at $R_{k}$ generate the group $\pi_1(\partial U_{D_{p}'}^*)$ and that 
$$\pi_1(\partial U_{D_{p}'}^*)=\left\langle \alpha_{k},\alpha_{{\sigma^{(\kappa)}}^{-1}(r)_1},\alpha_{p} \left\lvert \begin{array}{l}
[\alpha_{k},\alpha_{p}],  \ [\alpha_{{\sigma^{(\kappa)}}^{-1}(r)_1},\alpha_{p}], \\
\alpha_{p}=\alpha_{k}\cdot \alpha_{{\sigma^{(\kappa)}}^{-1}(r)_1} 
\end{array}\right.  \right\rangle $$ by Lemma \ref{lem:RelsInBundle} and because $D_{p}'\cdot D_{p}'=-1$.
\end{proof}

\begin{thm}\label{thm:KnotPresentation} Let $\A\subset \mathbb{P}^2$ be an arrangement of lines, $\mathcal{W}$ a wiring diagram and $M(\A,I,P)$ a partial compactification. Then 
$$\pi_1(M(\A,I,P),q)=\left\langle \lambda_1,\ldots, \lambda_{n+1}\mid \bigcup_k R_k, \lambda_{n+1}\cdots\lambda_1, \bigcup_{i\in I} \lambda(D_i'),\bigcup_{p\in P}\lambda(D_p'),\right\rangle $$
is a presentation for the fundamental group of the partial compactification.
\end{thm}

\begin{proof}
We only have to justify the expression for those meridians around lines as in (\ref{en:Mer1}) and (\ref{en:Mer3}). For the meridians of lines as in (\ref{en:Mer2}) and (\ref{en:Mer4}), the expression $\lambda(D_i')$ and $\lambda(D_p')$ is explained by Lemmas \ref{lem:ExcMerProduct} and \ref{lem:doubleBlowUp} respectively. We will conclude by Proposition \ref{prop:meridians}.

For the meridians around lines as in (\ref{en:Mer1}), it is immediate by the biholomorphism property of the blow-up outside the exceptional divisor.

Consider a line $D_i$ as in (\ref{en:Mer3}) and suppose that it comes from a point $p=D_r\cap D_k$ with $r,k\leq n+1$. Note that there is essentially no difference with a line as in (\ref{en:Mer2}) besides the change of local indexation to a global one,  and therefore, we can proceed as in the proof of Lemma \ref{lem:ExcMerProduct} to obtain that $\lambda_{{\sigma^{(\kappa)}}^{-1}(r)}^{(\kappa)}\lambda_{{\sigma^{(\kappa)}}^{-1}(k)}^{(\kappa)}$ is homotopic to a fiber of $\partial U_{i}^*$ connected to the global base point $R$.

\end{proof}

As $D=\sum_{i=1}^{N} D_i$ is a simple normal crossing divisor with $N=n+1+s_0$, we can consider an orbifold structure in $(\Bl_{P_0} \mathbb{P}^2, D)$ (see \cite{Eyssidieux} for the notation) by choosing weights $r=(r_1,\ldots,r_{N})\in (\mathbb{N}^*\cup\{+\infty\})^{N}$.

\begin{thm}\label{thm:OrbifoldPresentation} Let $\A$ be a complex arrangement of lines, $\mathcal{W}$ a wiring diagram and consider the weights $r$ of $D$ as above. The fundamental group $\pi_1(\mathcal{X}(\Bl_{P_0}\mathbb{P}^2, D, r ))$ of the orbifold $\mathcal{X}(\Bl_{P_0}\mathbb{P}^2, D, r )$ admits the following presentation: 
$$\left\langle \lambda_1,\ldots, \lambda_{n+1}\mid \bigcup_k R_k, \lambda_{n+1}\cdots\lambda_1, \bigcup_{i=1}^{N}\lambda(D_i)^{r_i}\right\rangle $$
where the relation $\lambda(D_i)^{r_i}$ is omitted if $r_i=+\infty$.
\end{thm}

\section{Boundary Manifolds methods}\label{s:BoundaryM}
In this Section we use the results of Mumford as stated in \ref{ss:mumford} in order to study the fundamental group of the boundary manifold $\partial U$ of an arrangement of lines $\A$.

The notion of wiring diagram defined in the previous section will play an important role, a  presentation of $\pi_1(M(\A,I,P))$ will be obtained as a quotient of the presentation of $\pi_1(\partial U)$ and compared with Theorem \ref{thm:KnotPresentation}.

\subsection{Boundary manifold of an arrangement of lines}
\subsubsection{Fundamental group of the boundary manifold of an arrangement of lines}\label{ss:FundaBoundary}
Let $\A=\{L_1,\ldots,L_{n+1}\}\subset \mathbb{P}^2$ be an arrangement of lines and denote by $\pi:\bar{X}\to \mathbb{P}^2$ the blow-up of the projective plane at the $s_0$ points of $\Sing \A$ of multiplicity equal or higher than $3$ as in \ref{ss:partialCArrangements}. Recall that $D=\abs{\pi^* D}=\sum_{i=1}^{n+s_0+1} D_i$ is the reduced total transform of $\A$ in $\bar{X}$ and let $\psi:\partial U\to D$ be its boundary manifold.

Using the description of Mumford (Theorem \ref{thm:Mumford}) and that of a weighted graph, Westlund gave a presentation of the fundamental group $\pi_1(\partial U)$ of $\partial U$ \cite{westlund} (see also \cite{Cohen_2008}). Let us describe this presentation.

Denote by $\Delta$ the dual graph of $D$  and by $\mathcal{E}$ the set of edges of $\Delta$ as in \ref{sbs:dualgraph} above. Associate to each vertex $v_i$ a weight $w_i$ corresponding with the self-intersection number of the associated line $D_i$ in $\bar{X}$. 

Let $\mathcal{T}$ be a maximal tree of $\Delta$ (a subgraph of $\Delta$ containing no cycles and all the vertices of $\Delta$) and denote by $\mathcal{C}=\Delta\setminus \mathcal{T}$. Note that $g=\abs{\mathcal{C}}=b_1(\Delta)$ equals the number of independent cycles in $\Delta$. 

The edges in $\mathcal{C}$ correspond to $g$ points $\{p_1,\ldots,p_g\}$ in $\Sing D$. Let us denote by $\pi^{(1)}:\Bl_{p_1,\ldots,p_g}\bar{X}\to \bar{X}$ the blow-up at these points. Denote by $D'=\sum_{i=1}^{n+s_0+1}D_i'$ the strict transform of $D$ in $\Bl_{p_1,\ldots,p_g}\bar{X}$  and let $\psi':\partial U'\to D'$ be the boundary manifold of $D'$. Note that the dual graph of $D'$ is a tree that can be identified with $\mathcal{T}$ by removing from $\Delta$ the edges in $\mathcal{C}$. In particular, $D'$ and $\partial U'$ are connected. Let ${\pi^{(1)}}^*(D)=D'+\sum_{k=1}^g E_k$ be the total transform of $D$ with $E_1,\ldots, E_k$ exceptional divisors.


Now, if $(i,j)\in \mathcal{C}$ corresponds to the point $p_k$ for some $1\leq k \leq g$, there exists an exceptional divisor $E_k\in \Bl_{p_1,\ldots,p_g}\bar{X}$ and $D_i',D_j'$ strict transforms of irreducible components $D_i,D_j$ of $D$ respectively such that $E_k\cap D_i'\not = \varnothing$, $E_k\cap D_j'\not = \varnothing$ and $D_i\cap D_j=p_k$. Denote its boundary manifold by $\psi_{E_k}:\partial U_{E_k}\to E_k, \psi_i':\partial U_i'\to D_i', \psi_j':\partial U_j'\to D_j'$.

Select a base point $Q_i \in D_i'\setminus (\cup_{j\not =i}D_j'\bigcup \cup_{k=1}^g E_k )$ as in \ref{ss:mumford}  and a simple curve $l_i\subset D_i'$ containing $Q_i$ and every intersection of the form:
\begin{enumerate}
\item \label{en:1} $D_i'\cap D_j', \text{ with } (i,j) \text{ an edge in } \mathcal{T}$,
\item $D_i'\cap E_k$, with $E_k$ coming from a point $p_k$ corresponding to an edge $(i,j)$ in $\mathcal{C}$.
\end{enumerate}  
Let us label these points by the order they intersect $l_i$ as $P_{i1}, \ldots, P_{i{k_i'}}$. Note that for every $P_{im}$ there corresponds a unique edge $(i,j_\Delta(i,m))$ in $\Delta$. This defines an injective function $m\mapsto j_{\Delta}(i,m)$ from $\{1,\ldots,k_i'\}$ to $\{1,\ldots, n+s_0+1\}$.

We also label only the points as in (\ref{en:1}) by the order they intersect $l_i$ as $P_{i1}',\ldots, P_{i{k_i}}'$ and define a function $m\mapsto j_{\mathcal{T}}(i,m)$ from $\{1, \ldots, k_i\}$ to $\{1,\ldots, n+s_0+1\}$ as in  \ref{ss:mumford}.

Let $l=\cup l_k\subset D'$ and $h':l\to \partial U'$ be a continuous function such that $\psi'\circ h'=id_l$. For an exceptional divisor $E_k$ corresponding to an edge $(i,j)$ in $\mathcal{C}$, we let $l_{E_k}\subset E_k$ be a simple path connecting $E_k\cap D_i'$ to $E_k\cap D_j'$ and $h_{E_k}:l_{E_k}\to \partial U_{E_k}$ such that $\psi_{E_k}\circ h_{E_k}=id_{l_{E_k}}$,  $h_{E_k}(l_{E_k})\cap h'(l_i)\not =\varnothing$ and $h_{E_k}(l_{E_k})\cap h'(l_j)\not = \varnothing$. This create a cycle $c_k=c_{ij}$ in the boundary manifold of ${\pi^{(1)}}^*(D)$, which we orient passing first by $h'(l_i)$, following $h_{E_k}(l_{E_k})$ and coming back by $h'(l_j)$. We denote by $\gamma_1,\ldots,\gamma_{n+1+s_0}$ the meridians around $D_1',\ldots, D_{n+1+s_0}'$ obtained as in \ref{ss:mumford} using $h'(l)$.

\begin{thm}[Westlund]\label{thm:Westlund}
A presentation for $\pi_1(\partial U)$ is given by 
$$\pi_1(\partial U)=\left\langle 
 \begin{array}{llll}
\gamma_1, \ldots, \gamma_{n+s_0+1} & \big\lvert [\gamma_i,\gamma_j^{s_{ij}}], & (i,j)\in \mathcal{E} \\
c_1, \ldots, c_g & \bigg\lvert \gamma_i^{-w_i}=\prod_{m=1}^{k_i'} \gamma_{j_{\Delta}(i,m)}^{s_{ij_{\Delta}(i,m)}} & 1\leq i \leq n+s_0+1
\end{array} \right\rangle $$
where 
$$s_{ij}=\left\lbrace\begin{array}{ll}
c_k^{-1} & \text{ if } (i,j) \text{ equals the $k$-th element in } \mathcal{C},\\
c_k & \text{ if } (j,i) \text{ equals the $k$-th element in } \mathcal{C},\\
1 & \text{ if } (i,j) \text{ is an edge of } \mathcal{T}.\\	 
\end{array} \right. $$ 
\end{thm}
\begin{proof}  From Theorem \ref{thm:Mumford} we know that $$\pi_1(\partial U')=\left \langle \gamma_1,\ldots\gamma_{n+s_0+1}\bigg \lvert [\gamma_i,\gamma_{j_{\mathcal{T}(i,m)}}]\ m=1,\ldots, k_i, \gamma_i^{-w_i'}=\prod_{m=1}^{k_i}\gamma_{j_{\mathcal{T}(i,m)}} \right \rangle $$
where $w_i'$ is the intersection number of the strict transform $D_i'$ of $D_i$ in $\Bl_{p_1,\ldots,p_g}\bar{X}$. Note that $(i,l)$ is an edge of $\mathcal{T}$ if and only if $l=j_{\mathcal{T}}(i,m)$ for some $m\in \{1,\ldots, k_i\}$ and therefore the set of relations $A=\{[\gamma_i, \gamma_{j_{\mathcal{T}}(i,m)}]\mid m=1,\ldots, k_i, i=1,\ldots, n+s_0+1\}$ is the same as $B=\{[\gamma_i,\gamma_{l}]\mid (i,l) \text{ an edge of }\mathcal{T}\}$.

Let $E_k$ be an exceptional divisor corresponding to an edge $(i,j)$ in $\mathcal{C}$ as above. We can remove two disks $\Delta_1'\subset D_i',\Delta_2'\subset D_j'$ in $D'$ around the points $E_k\cap D_i'$ and $E_k\cap D_j'$ respectively, and obtain a pair of torus $T_i',T_j'$ as boundary from $\partial U'^\circ=\psi'^{-1}(D'\setminus \Delta_1'\cup \Delta_2')$. Let $\gamma(E_k)_i',\gamma_i$ and $\gamma(E_k)_j',\gamma_j$ be generators of $\pi_1(T_i')$ and $\pi_1(T_j')$ with $\gamma(E_k)_i', \gamma(E_k)_j'$ constructed from $\partial \Delta_1'$, $\partial \Delta_2'$ as in \ref{ss:mumford}. We obtain the following presentation for $\pi_1(\partial U'^\circ)$:
$$\left \langle \gamma_1,\ldots\gamma_{n+s_0+1}, \gamma(E_k)_i', \gamma(E_k)_j' \Bigg \lvert \begin{array}{l}
A, {[\gamma_i,\gamma(E_k)_i'],[\gamma_j, \gamma(E_k)_j']} \\
\gamma_l^{-w_l'}=\prod_{m=1}^{k_l}\gamma_{j_{\mathcal{T}}(l,m)}  \text{ for } l\not = i,j,\\
\gamma_i^{-w_i'}=\gamma_{j_{\mathcal{T}}(i,1)}\cdots \gamma(E_k)_i' \cdots \gamma_{j_{\mathcal{T}}(i,k_i)}, \\
 \gamma_j^{-w_j'}=\gamma_{j_{\mathcal{T}}(j,1)}\cdots \gamma(E_k)_j' \cdots \gamma_{j_{\mathcal{T}}(i,k_j)}  
\end{array} \right \rangle $$
where the products in the lowest row of the relations are taken in such a way that $1=\psi_i(\gamma_{i1}')\cdots \psi_i(\gamma(E_k)_i')\cdots \psi_i(\gamma_{i{k_i}}')$ holds in $\pi_1({D_i'}^*\setminus \Delta_1)$ and similarly $1=\psi_j(\gamma_{j1}')\cdots \psi_j(\gamma(E_k)_j')\cdots \psi_j(\gamma_{j{k_j}}')$ in $\pi_1({D_j'}^*\setminus \Delta_2)$ for generators $\gamma_r, \gamma_{r1}',\ldots, \gamma_{r{k_r}}'$ generators of $\pi_1(\partial {U_r'}^*)$ for $r=i,j$ as in \ref{ss:mumford}.

  Let $E_k^*$ denote the submanifold of $E_k$ obtained by removing another pair of disks $\Delta_1,\Delta_2$ of $E_k$ about the points $E_k\cap D_i'$ and $E_k\cap D_j'$  as in \ref{ss:mumford}. Write $\partial U_{E_k}^*$ for $\psi_{E_k}^{-1}(E_k^*)$. Note that the boundary of $\partial U_{E_k}^*$ consists also of a pair of torus $T_i,T_j$ corresponding to $\Delta_1$ and $\Delta_2$ respectively. Let $\gamma_i', \gamma(E_k)$ and $\gamma_j', \gamma(E_k)$  be generators of $\pi_1(T_i)$ and $\pi_1(T_j)$ respectively. 
By (\ref{eq:RelsInBundle}), we have that $\pi_1(\partial U_{E_k}^*)=\langle \gamma(E_k),\gamma_i',\gamma_j'\mid [\gamma(E_k),\gamma_i'],[\gamma(E_k),\gamma_j'], \gamma(E_k)=\gamma_i'\gamma_j' \rangle \cong \langle \gamma_i', \gamma_j' \mid [\gamma_i',\gamma_j'] \rangle \cong \mathbb{Z}^2$ because $E_k\cdot E_k=-1$.

 We can glue $\partial U'^\circ$ to $\partial U_{E_k}^*$ by first gluing $T_i$ to $T_i'$ by  a longitude-to-meridian orientation-preserving attaching map $f$, and similarly $T_j$ to $T_j'$ by a map $g$. 
 
 First, by the van Kampen Theorem we obtain that $\gamma_i=\gamma_i'$ and $\gamma(E_k)=\gamma(E_k)_i'$. Then, from HNN extension we get $\gamma_j=c_k^{-1} \gamma_j' c_k$ and $\gamma(E_k)_j'=c_k^{-1} \gamma(E_k) c_k$.
 
 We obtain the following presentation of $\pi_1(\partial U'^\circ\cup_{f,g} \partial U_{E_k}^*)$ by replacing $\gamma_i'=\gamma_i,\gamma_j'=\gamma_j^{c_k^{-1}},\gamma(E_k)=\gamma_i\gamma_j^{c_k^{-1}},$ $\gamma(E_k)_i'=\gamma_i\gamma_j^{c_k^{-1}},\gamma(E_k)_j'=\gamma_i^{c_k}\gamma_j$ in terms of $\gamma_i,\gamma_j,c_k$
 
 $$\left \langle \gamma_1,\ldots\gamma_{n+s_0+1}, c_k \bigg \lvert \begin{array}{l} {[\gamma_i,\gamma_j^{c_k^{-1}}]}, [\gamma_i,\gamma_l] \text{ with } (i,l)\in \mathcal{T}  \\
\gamma_l^{-w_l'}=\prod_{m=1}^{k_l}\gamma_{j_{\mathcal{T}}(l,m)}  \text{ for } l\not = i,j,\\
\gamma_i^{-w_i'}=\gamma_{j_{\mathcal{T}}(i,1)}\cdots \gamma_i \gamma_j^{c_k^{-1}} \cdots \gamma_{j_{\mathcal{T}(i,k_i)}},\\
\gamma_j^{-w_j'}=\gamma_{j_{\mathcal{T}(j,1)}}\cdots \gamma_i^{c_k}\gamma_j \cdots \gamma_{j_{\mathcal{T}(j,k_j)}}  
\end{array} \right \rangle $$
Note that the row of the relations corresponding to $i$ can be simplified to \begin{equation}\label{eq:above}
\gamma_i^{-(w_i'+1)}=\gamma_{j_{\mathcal{T}}(i,1)}\cdots\gamma_j^{c_k^{-1}} \cdots \gamma_{j_{\mathcal{T}(i,k_i)}},
\end{equation} as $\gamma_i$ commutes with every $\gamma_{j_\mathcal{T}(i,m)}$. A similar simplification can be made for the relation corresponding to $j$.

We repeat the above process for every $E_k$ with $k=1,\ldots, g$. After this, the order for the product as in (\ref{eq:above}), is given by the function $m\mapsto j_\Delta(i,m)$ and the conjugations $s_{ij}$ as in the statement of the Theorem. We get that $\gamma_i^{-(w_i'+(k_i'-k_i))}=\prod_{m=1}^{k_i'}\gamma_{j_\Delta(i,m)}^{s_{ij_\Delta(i,m)}}$. Note that $k_i'-k_i$ equals the number of points in $\{p_1,\ldots, p_g\}\cap D_i$, and therefore $w_i'+(k_i'-k_i)=w_i$.

 This gives a presentation for the fundamental group of the boundary manifold of the total transform of $D$, which is homeomorphic to $\partial U$.
\end{proof}

A central computation in our work is the expression of the meridians around the exceptional divisors $D_{n+2},\ldots, D_{n+s_0+1}$ in $D=\sum_{i=1}^{n+s_0+1}D_i$ in terms of meridians of the lines in $\A$.  As a partial result we obtain an expression in the following corollary. The cycles $c_k$ will be expressed in terms of meridians of the lines in \ref{ss:BoundaryToGlobal}.
\begin{cor}\label{cor:ExcDiv} For $r=n+2, \ldots, n+s_0+1$, we have that in $\pi_1(\partial U)$,
$$\gamma_r=\prod_{m=1}^{k_r'}\gamma_{j_{\Delta}(r,m)}^{s_{rj_{\Delta}(r,m)}} \text{ with } s_{rj}=\left\lbrace \begin{array}{cc}
c_k^{-1} & \text{ if } (r,j) \text{ equals the $k$-th element in }\mathcal{C},\\
c_k & \text{ if } (j,r) \text{ equals the $k$-th element in }\mathcal{C},\\
1 & \text{ if } (r,j) \text{ is an edge in } \mathcal{T}.\\
\end{array} \right.$$
\end{cor}
\begin{proof}
It follows from the relation $\gamma_r^{-w_r}=\prod_{m=1}^{k_r'} \gamma_{j_{\Delta}(i,m)}^{s_{rj_\Delta(r,m)}}$ in the presentation of $\pi_1(\partial U)$ in Theorem \ref{thm:Westlund}, the fact that $w_r=-1$ because $D_r$ is an exceptional divisor and hence $j_{\Delta}(r,m)\in\{1,\ldots,n+1\}$.
\end{proof}
\subsubsection{Choice of a maximal tree}\label{sss:MaximalTree}
In what follows, we will define a maximal tree $\mathcal{T}'$ of the dual graph $\Delta$ of $D$ as defined in \cite[Section 3.3]{Cohen_2008}. 

In the arrangement $\A=\{L_1,\ldots,L_{n+1}\}$, we will fix the line $L_{n+1}$ as the line at infinity, recall that we denote by $D_i$ the strict transform by $L_i$ for $i\leq n+1$ in $D=\sum D_i\subset \bar{X}$.

Consider the following subset of edges $\mathcal{E}'\subset \mathcal{E}$ which defines a maximal tree $\mathcal{T}'\subset \Delta$ of the dual graph $\Delta$ of $D$: 
\begin{enumerate}
\item Let $(j,n+1),(n+1,j)\in \mathcal{E}'$ if $D_{n+1}\cap D_j\not = \varnothing$. This is, all the edges having as an endpoint the vertex corresponding to $D_{n+1}$.	
\item  Let $(i,j)\in \mathcal{E}'$ if $n+1<j$ ($D_j$ is an exceptional divisor) with either
\begin{itemize}
\item $D_j\cap D_{n+1}=\varnothing$ and $i=\min\{l\mid D_l\cap D_j\not =\varnothing\}$. Note that $D_j$ comes from a point in $\Sing \A\setminus L_{n+1}$.
\item or $D_j\cap D_{n+1}\not =\varnothing$ and $D_i\cap D_j\not =\varnothing$. The line $D_i$ corresponds then to a line $L_i$ touching $L_{n+1}$ in a point of multiplicity $>2$.
\end{itemize}
\end{enumerate}

Note that $\mathcal{E}\setminus \mathcal{E}'$ consists either:
\begin{itemize}
\item of edges corresponding to double points $L_i\cap L_j$ with $i,j<n+1$,
\item or, if $p=L_{i_1}\cap \ldots \cap L_{i_l}$ with $i_1<\ldots<i_l<n+1$, $2<l$, and $E_j$ denotes the exceptional divisor obtained by blowing up at $p$, of edges of the form $(i_r, j)$ with $r=2,\ldots,l$.
\end{itemize}
Let us consider the presentation of $\pi_1(\partial U)$ as in Theorem \ref{thm:Westlund}. If $(i,j)$ equals the $k$-th element in $\Delta\setminus \mathcal{T}'$ as in the first point above, we denote the cycle $c_k$ by $c_{i,j}$. Recall that if $i<j$, we pass first through $h'(l_i)$ and then through $h'(l_j)$.

For the cycles created by the edges in the second point, let us suppose that the irreducible component of $D$ are ordered in such a way that $D_{n+1}\cap D_j\not = \varnothing$ for $j=n+2,\ldots,s'$ and $D_{n+1}\cap D_k=\varnothing$ for $k>s'$.

For $s'<\iota\leq n+1+s_0$, we have that, as $D_{\iota}$ is an exceptional divisor,  $1\leq j_\Delta(\iota,m)\leq n$ for $1\leq m\leq k_i'$, and  $\gamma_{\iota}^{-w_\iota}=\prod_{m=1}^{k_\iota'} \gamma_{j_\Delta(\iota,m)}^{s_{\iota j_\Delta(\iota,m)}} $ holds as in Theorem \ref{thm:Westlund}. Note that if $(j_\Delta(\iota,m),\iota)$ equals the $k$-th element in $\Delta\setminus \mathcal{T}'$, we have that $s_{\iota j_\Delta(\iota,m)}=c_k$. In this case, we denote $c_k$ by $c_{j_\Delta(\iota,m),\iota}$. As $\mathcal{T}'$ is a maximal tree, the edges  corresponding to $(j_\Delta(\iota,m),\iota)$ for $1<m\leq k_i'$, give rise to $k_i'-1$ independent cycles $c_{j_\Delta(\iota,m),\iota}$ in $\Delta$.

Using the tree $\mathcal{T}'$ and corollary \ref{cor:ExcDiv}, we can express the meridian around an exceptional divisor in terms of the meridians of the lines and the cycles $c_{s_t'',s''}$:
\begin{equation}\label{eq:gammaExc}
\gamma_{\iota}=\gamma_{j_{\Delta}(\iota,1)}^{c_{j_\Delta(\iota,1),\iota}}\gamma_{j_\Delta(\iota,2)}^{c_{j_\Delta(\iota,2),\iota}}\cdots \gamma_{j_\Delta(\iota,k_\iota')}^{c_{j_\Delta(\iota,k_\iota'),\iota}} \quad \text{ for } s'<\iota\leq n+1+s_0
\end{equation}

with $c_{j_\Delta(\iota,r),\iota}=1$ if $r=\min \{j_\Delta(\iota,m)\mid m=1,\ldots, k_i'\}$.

\subsection{From a presentation for the boundary manifold of an arrangement of lines to a presentation of its complement}\label{ss:BoundaryToGlobal}
 Let $\A\subset \mathbb{P}^2$ be an arrangement of lines and $\partial U$ its boundary manifold. We identify $\partial U$ with the boundary manifold of the total transform $D$ of $\A$ in $\pi:\bar{X}\to \mathbb{P}^2$, the blow-up of $\mathbb{P}^2$ at the points of $\Sing \A$ of multiplicity higher than two.  Denote by $i:\partial U \hookrightarrow \mathbb{P}^2\setminus \A$ the inclusion map and by $i_*:\pi_1(\partial U)\to \pi_1(\mathbb{P}^2\setminus \A)$ the induced homomorphism. 

Consider the presentations $\langle\gamma_1, \ldots, \gamma_{n+1}, c_1, \ldots, c_g \mid R' \rangle$ of $\pi_1(\partial U)$ with $R'$ the set of relations as in Theorem \ref{thm:Westlund} and $\langle \lambda_1, \ldots, \lambda_{n+1} \mid \cup R_k,  \lambda_{n+1}\cdots \lambda_1\rangle$ of  $\pi_1(\mathbb{P}^2\setminus \A)$ as in Theorem \ref{thm:Arvola}.

Recall that the construction of the me\-ri\-dian $\gamma_k$ around the irreducible component $D_k$ of $D=\sum_{k=1}^{n+1+s_0} D_k$ depends on a choice of a maximal tree $\mathcal{T}$ of the dual graph $\Delta$ of $D$, contractible paths $l_k\subset D_k$, and a section $h:l=\cup l_k\to \partial U$, see \ref{ss:mumford}. We choose the maximal tree $\mathcal{T}'$ constructed at \ref{sss:MaximalTree}. For $p=L_{\eta_1}\cap \ldots \cap L_{\eta_r}\in \Sing \A\setminus L_{n+1}$, we have a unique cycle $c_{\eta_1,\eta_2}$ if $r=2$ and $r-1$ cycles if $r>2$, in this case let us denote by $D_\iota$ the corresponding exceptional divisor in $\bar{X}$, therefore we have the cycles $c_{\eta_2,\iota},\ldots, c_{\eta_r,\iota}$. See \ref{sss:MaximalTree}.

Consider a wiring diagram $\mathcal{W}$ of $\A$ as in \ref{ss:111}. There exists $\kappa\in \mathbb{N}^*$ such that $p\in M_\kappa$. Consider the geometric generating set $\Gamma^{(\kappa)}=\{\lambda_1^{(\kappa)},\ldots, \lambda_{n+1}^{(\kappa)}\}$. Recall that, as in remark \ref{rem:GammaKinGamma1}, there exists a word $\xi_j^{(\kappa)}$ in $\lambda_1^{(1)}, \ldots, \lambda_{n+1}^{(1)}$ (see also \ref{ss:CyclesInMeridians}), such that 
$$\lambda_{{\sigma^{(\kappa)}}^{-1}(j)}^{(\kappa)}={\lambda_{j}^{(1)}}^{\xi_j^{(\kappa)}} \quad  \text{ for } j=1,\ldots,n+1. $$

The main objective of this subsection is to prove the following Theorem.
\begin{thm}\label{thm:Guerville} 
The paths $l_1, \ldots, l_{n+1+s_0}$, the map $h:l\to \partial U$ and the wiring diagram $\mathcal{W}$ of $\A$ can be chosen in such a way that
\begin{enumerate}
\item\label{en:Guerville1} The generator $\gamma_k$ of $\pi_1(\partial U)$ lies in the same homotopy class as $\lambda_k$ in $\mathbb{P}^2\setminus \A$ for $k=1,\ldots, n+1$. 
\item If $p=L_{\eta_1}\cap \ldots \cap L_{\eta_r}\in \Sing \A\setminus L_{n+1}$ and $p\in M_\kappa$ as above, then 
\begin{itemize}
\item 
if $r=2$, the cycle $c_{\eta_1,\eta_2}$ is homotopic in $\mathbb{P}^2\setminus \A$ to $\xi_{\eta_1}^{(\kappa)}(\xi_{\eta_2}^{(\kappa)})^{-1}$  and 
\item if $r>2$, the cycle $c_{\eta_a,\iota}$ is homotopic to $\xi_{\eta_a}^{(\kappa)}{\xi_{\eta_1}^{(\kappa)}}^{-1}$, for $a=2, \ldots, r$. 	
\end{itemize}
By the point (\ref{en:Guerville1}), we can also consider each $\xi_{\eta_a}^{(\kappa)}$ as a word in $\gamma_1,\ldots,\gamma_{n+1}$.

\item  If $p=L_{\eta_1}\cap \ldots \cap L_{\eta_r}\in \Sing \A\setminus L_{n+1}$, denote by $R'(p)$ the set of relations:
\begin{itemize}
\item  $\{c_{\eta_1,\eta_2}^{-1}\xi_{\eta_1}^{(\kappa)}(\xi_{\eta_2}^{(\kappa)})^{-1}\}$ if $r=2$ , or
\item  $\{c_{\eta
_a,\iota}^{-1}\xi_{\eta_a}^{(\kappa)}(\xi_{\eta_1}^{(\kappa)})^{-1}\mid a=2,\ldots, r\}$ if $r>2$.
\end{itemize} We have that $\langle \gamma_1,\ldots, \gamma_{n+1}, c_1, \ldots, c_g \mid R', \cup_{p\in \Sing\A\setminus L_{n+1}} R'(p)\rangle$ and $\langle \lambda_1, \ldots, \lambda_{n+1} \mid \cup R_k, \lambda_{n+1}\cdots \lambda_1\rangle$ are Tietze-equivalent  presentations of $\pi_1(\mathbb{P}^2\setminus \A)$.

\end{enumerate}
\end{thm}

By using a different presentation of $\pi_1(\partial U)$ and different techniques, the image of the generators of $\pi_1(\partial U)$  under the map $i_*$ was computed in \cite{florens} (See Proposition 2.13 and Theorem 4.5 of loc. cit.). The proof of Theorem \ref{thm:Guerville} is inspired by the ideas of \cite{florens}.

\subsubsection{Constructing equivalent generators}\label{sss:EquivalentGenerators}  

Let us choose the point $R\in \mathbb{P}^2\setminus \A$ close to $L_{n+1}$, consider the blow-up $\pi_R:\Bl_R\mathbb{P}^2\to \mathbb{P}^2$ and denote by $\bar{f}:Bl_R\mathbb{P}^2\to \mathbb{P}^1$ the associated pencil as in \ref{ss:111}.

Let $\beta:[0,1]\to \mathbb{P}^1$ be as in \ref{ss:111} such that it passes first through the projection of the points $\Sing \A\cap L_{n+1}$ to $\mathbb{P}^1$ via $\bar{f}$. Take its associated wiring diagram $ \mathcal{W}$ co\-rres\-pon\-ding to the arrangement $\A$ and fix a planar representation $p(\beta^* \mathcal{W})$ as in \ref{ss:PlanarWiring}. 

Let us order the representation of all the singular points $\Sing \A=\{p_1,\ldots, p_s\}$ in $p(\beta^* \mathcal{W})$ together with the virtual vertices $\{p_1',\ldots,p_\nu' \}\in p(\beta^*\mathcal{W})$, by the order they are crossed by the fiber $p(\beta^* \mathcal{W})|_{t}$ with $t$ increasing in $[0,1]$, and let $t_1,\ldots,t_{s+\nu}\in (0,1)$  be such that either an actual or a virtual vertex lies in $p(\beta^*(\mathcal{W}))|_{t_\kappa}$, for all $\kappa=1,\ldots,s+\nu$. By abuse of notation we will also denote by $t_\kappa$ the crossings in $p(\beta^*\mathcal{W})$ at the fiber $p(\beta^*(\mathcal{W}))|_{t_\kappa}$ and we will write $\mathcal{W}$ for $p(\beta^* \mathcal{W})$. Let $\Gamma^{(\kappa)}=\{\lambda_1^{(\kappa)},\ldots,\lambda_{n+1}^{(\kappa)}\}$ be the geometric generating set defined in \ref{ss:AlgorithmComplement}, for $\kappa=1,\ldots, s+\nu$.  

Recall that we have assumed that the order of the lines $L_1,\ldots, L_n$ is such that, at the very right of the planar representation of $\mathcal{W}$, the wire $W_1$ is at the bottom of $\mathcal{W}$, above it is the wire $W_2$ and then $W_3$, continuing in this way until $W_n$.

For an irreducible component $D_k$ of $D\subset \bar{X}$, denote its boundary manifold by $\psi_k:\partial U_k\to D_k$ and recall that we can consider $\psi_k|_{\partial U_k^*}:\partial U_k^*=\partial U_k\cap \partial U \to D_k^*$ (see \ref{ss:mumford}).
A set of generators for $\pi_1(\partial U_k^*)$ was constructed by fixing a base point $Q_k\in D_k^*$, simple paths $l_k\subset D_k$ from which we obtain paths $l_k'\subset D_k^*$ (see figure \ref{fig:MerMum}) and $h:\cup l_k\to \partial U$ as in \ref{ss:mumford}. The generators $\gamma_1,\ldots, \gamma_{n+1+s}$ were constructed by joining the different generators of $\pi_1(\partial U_k^*)$ to a common base point $Q$ via the contractible path $h(\cup l_k)$ in $\partial U$. 

Recall that the first $n+1$ irreducible components $D_1,\ldots, D_{n+1}$ of $D$ correspond to the lines $L_1,\ldots,L_{n+1}$ respectively and that, as in the end of \ref{ss:FundaBoundary}, there exists $s'$ such that for $j=n+2, \ldots, s'$, we have that $D_{n+1}\cap D_j\not = \varnothing$ and for $j>s'$, we have $D_{n+1}\cap D_j=\varnothing$.

\begin{lem}\label{lem:EqMeridians} For $k=n+1,\ldots, s'$, we can choose $l_k'\subset D_k^*$, a continuous map $h_k':l_k'\to \partial U_k^*$ and a base point $Q\in h_{n+1}(l_{n+1}')$ for the fundamental group $\pi_1(\partial U)$ in such a way that $i_*(\gamma_r)$ lies in the same homotopy class as $\lambda_r^{(1)}$ for $r=1,\ldots, n+1$.
\end{lem}
\begin{proof}

We begin by defining those $l_k'$ for $k=n+2,\ldots, s'$. Essentially, we arrange the choices in an appropriate way to obtain the stated in the lemma.

More precisely, let $D_k$ be an exceptional divisor corresponding to a point $p=p(k)\in \Sing \A\cap L_{n+1}$ with multiplicity higher or equal to three. Suppose that $p=L_j\cap L_{j+1}\cap \ldots\cap L_{m-1}\cap L_m\cap L_{n+1}$ (which can be written in this way by the order of the lines chosen above). Consider the boundary manifold $\psi_k:\partial U_k\to D_k$ in $\bar{X}$. We will also write $\partial U_p$ for the image $\pi(\partial U_k)\subset \mathbb{P}^2$ under the map $\pi:\bar{X}\to \mathbb{P}^2$.  For $r=j,\ldots, m$, each meridian $\lambda_r^{(1)}$ (see figure \ref{fig:elgeobas}) is homotopic to a meridian ${\lambda_r'}^{(1)}$ (see figure \ref{fig:decomposing1}) that can be decomposed in the following way: ${\lambda_r'}^{(1)}=\lambda_{p}\lambda_{r_1}\lambda_{p}^{-1}$ with $\lambda_{r_1}\subset \partial U_p$ a meridian of $L_r$ based at a point $q_p\in \partial U_p$ and $\lambda_p$ a path connecting $R$ and $q_p$.

We can further decompose each $\lambda_{r_1}$ as the boundary of a disk $\Delta$ around a point in $L_r$ and a path $\lambda_{r_2}$ connecting the point $q_p$ to $\partial \Delta$. Define the path $l_k'$ in $D_k^*$ as the projection $\psi_k(\cup_{r=j}^m \lambda_{r_2})$. We define $h|_{l_k'}$ such that $h|_{l_k'}\circ \psi_k|_{\cup_{r=j}^m \lambda_{r_2}} =id|_{\cup_{r=j}^m \lambda_{r_2}}$.

\begin{figure}[h]
        \centering
\begin{tikzpicture} 
\draw (-.4,-1) node [left]{${\lambda_j'}^{(1)}$};
\draw (-.4,0) node [left]{${\lambda_{j+1}'}^{(1)}$};
\draw (-.4,1) node [left]{${\lambda_m'}^{(1)}$};

\draw (0,1.3) node [above]{\tiny $\partial \Delta$};

\draw (.7,1) node [right]{\tiny ${\lambda_{m_2}}$};
\draw (.85,-.65) node [left]{\tiny ${\lambda_{j_2}}$};

\draw (.7,-1) node [right]{\small ${\lambda_p}$};

\draw (0,-1.5) node [left]{$R$};
\draw (1,0) node [right]{$q_p$};

\draw (0,-1) circle (.3);

\draw  (1,0) to[out=-100,in=0] (.3,-1);


\draw [arc arrow=to pos 0.35 with length 2mm] (-.5,0) to[out=-90,in=-90] 
(.5,0) [arc arrow=to pos 0.85 with length 2mm] 
to[out=90,in=90] cycle;
\draw  (0,-1.5) to[out=0,in=-90] (1,0);
\draw  (1,0) to[out=90,in=0] (.5,1);

\draw [arc arrow=to pos 0.35 with length 2mm] (-.5,1) to[out=-90,in=-90] 
(.5,1) [arc arrow=to pos 0.85 with length 2mm] 
to[out=90,in=90] cycle;
\draw  (1,0) to[out=-180,in=0] (.5,0);

\foreach \Point in {(0,-1),(0,0), (0,1),(1,0)}{
    \node at \Point {\textbullet};
}


\end{tikzpicture}
\caption{Decomposing a meridian}
\label{fig:decomposing1}

\end{figure}
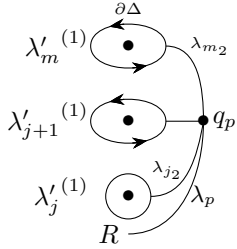

Note that, up to a slight change in $R$, the paths $\lambda_p$ are homotopic to paths $\lambda_p'$ lying in $\partial U_{n+1}^*$. 

Now, let $p=L_j\cap L_{n+1}$ be a double point in $\Sing \A\cap L_{n+1}$. The meridian $\lambda_j^{(1)}$ can be decomposed as $\lambda_j' \cdot \partial \Delta_j \cdot  \lambda_j'^{-1}$, with $\partial \Delta_j$ a fiber of $\partial U_j^*$ and $\lambda_j'\subset \partial U_{n+1}^*$ a path starting at $R$ and finishing at point $q_j\in \partial U_j^*\cap \partial U_{n+1}^*$.

Finally, for $k=n+1$ we define $l_k\in D_k$ as the image of $\beta([0,1])$ under the section of the pencil $\bar{f}:\Bl_R\mathbb{P}^2\to \mathbb{P}^1$ with range $D_k$. By construction $l_k$ passes over all the points in $\Sing \A\cap L_{n+1}$. We let $h|_{l_{n+1}}$  be a continuous function such that $\psi_{n+1}|_{h|_{l_{n+1}}(l_{n+1})}\circ h|_{l_{n+1}}=id_{l_{n+1}}$,  $h|_{l_{n+1}}(l_{n+1})$ is a simple path passing through each $q_j$ with $L_j\cap L_{n+1}$ a double point, touching each $\lambda_{(n+1)_2}(p)$ for each point $p\in \Sing \A\cap L_{n+1}$ of multiplicity greater or equal to two, and such that each $\lambda_p\cdot \lambda_{j_2}$ is homotopic to a segment of $h|_{l_{n+1}}(l_{n+1})$.


By the construction of the maximal tree $\mathcal{T}'$, these paths are sufficient to construct $\gamma_j$ for $j=1,\ldots, n+1$ and by construction, they lie in the same homotopy class as $\lambda_j^{(1)}$. 
\end{proof}

\begin{cor} The morphism $i_*:\pi_1(\partial U)\to \pi_1(\mathbb{P}^2\setminus \A)$ is surjective.
\end{cor}
\begin{proof}
The group $\pi_1(\mathbb{P}^2\setminus \A)$ is generated by the elements $\Gamma^{(1)}=\{\lambda_1^{(1)},\ldots, \lambda_{n+1}^{(1)}\}$, as $i_*(\gamma_k)=\lambda_k^{(1)}$ for $k=1,\ldots, n+1$ the result follows.
\end{proof}
Suppose that $p\in M_\kappa\cap (\Sing\A\setminus L_{n+1})$ is of multiplicity higher or equal to three and that $p=L_{\sigma^{(\kappa)}(j)}\cap L_{\sigma^{(\kappa)}(j+1)}\cap \ldots\cap L_{\sigma^{(\kappa)}(m)}$. Denote by $\psi_p:\partial U_p\to E_p$ the boundary manifold of the exceptional divisor $E_p\subset \bar{X}$ obtained by blowing-up $p$. We select $l_p'\subset E_p^*$ in a similar way as in the proof of the precedent Lemma for a point of multiplicity higher than two lying in $L_{n+1}$: decompose each $\gamma_j^{(\kappa)},\ldots,\gamma_m^{(\kappa)}$ into a path $\lambda_p^{(\kappa)}$ connecting $R\subset \mathbb{P}^2$ and a point $q_p\in \partial U_p$, and ${\lambda_{r_1}}^{(\kappa)}$ with $r\in \{j,\ldots,m\}$ based at $q_p$ and generating $\pi_1(\partial U_p^*)$ as in figure \ref{fig:decomposing1}. Decompose further ${\lambda_{r_1}}^{(\kappa)}$ into a boundary of a disk $\partial \Delta_r$ around a point of the line $L_{\sigma^{(\kappa)}(r)}$ and a path ${\lambda_{r_2}}^{(\kappa)}$ connecting $q_p$ and $\partial \Delta_r$. We take $l_p'=\psi_p(\cup {\lambda_{r_2}}^{(\kappa)} )$ and define $h|_{l_p'}$ such that $h|_{l_p'}\circ \psi_p|_{\cup {\lambda_{r_2}}^{(\kappa)}}=id|_{\cup {\lambda_{r_2}}^{(\kappa)}}$.

For every $k=1,\ldots, n$, we define $l_k\subset D_k$ as the image of $\beta([0,1])$ under the section of $\bar{f}:\Bl_R\mathbb{P}^2\to \mathbb{P}^1$ that has as range $D_k$. We define $h|_{l_k}$ such that it is continuous, $\psi_k|_{h|_{l_k}(l_k)} \circ h_{l_k}=id_{l_k}$, $h_{l_k}(l_k)$ intersects $\cup \lambda_{r_2}^{(\kappa)}$ in a point if $p\in M_\kappa\cap (\Sing\A \cap L_k)$ with the notations as in the paragraph above, $h_{l_k}(l_k)\cap h_{l_{k'}}(l_{k'})\not = \varnothing$ if $L_k\cap L_{k'}\not = \varnothing$ is a double point and $h_{l_k}(l_k)$ is not homotopic to a multiple of a fiber $S^1$ of $\partial U_k$.

\subsubsection{Expressing the cycles in terms of the meridians}\label{ss:CyclesInMeridians}
Let $t_\eta\in \mathcal{W}$ be an actual vertex and suppose that $t_\eta=W_{\eta_1}\cap W_{\eta_2}\cap \ldots\cap W_{\eta_r}$  with the global order of the wires of $\mathcal{W}$ such that $\eta_1<\eta_2<\ldots<{\eta_r}< n+1$. By definition of the maximal tree $\mathcal{T}'$, to each $\eta_a$, with $a>1$,  corresponds a cycle $c_{\eta_a,t_\eta}$ which is a generator of $\pi_1(\partial U)$, see \ref{sss:MaximalTree}. This cycle is constructed by connecting $h|_{\eta
_a}(l_{\eta_a})\cdot h|_{t_\eta}(l_{t_\eta}')\cdot h|_{\eta_1}(l_{\eta_1})$ to $R$ if $r>2$ and by connecting $ h|_{\eta_1}(l_{\eta_1})\cdot h|_{\eta
_a}(l_{\eta_a})$ to $R$ if $r=2$. 

 For every $\kappa\leq \eta$, consider the geometric generating set $\Gamma^{(\kappa)}=\{\lambda_1^{(\kappa)},\ldots, \lambda_{n+1}^{(\kappa)}\}$ as in \ref{ss:AlgorithmComplement} and recall the construction of the functions $\tau^{(\kappa)}:\{1,\ldots,n+1\}\to F_{n+1}^{(\kappa)}$ as defined before Proposition \ref{prop:ConjTogether}. For $1\leq a \leq r$,  denote by \begin{equation}
 \xi_{\eta_a}^{(\kappa)}=\tau^{(1)}(\eta_a)\cdot \tau^{(2)}\left({\sigma^{(2)}}^{-1}(\eta_a)\right)\ldots \tau^{(\kappa-1)}\left({\sigma^{(\kappa-1)}}^{-1}(\eta_a)\right ). \end{equation}

\begin{prop}\label{prop:CyclestoMeridians} Let $1<a\leq r$. The image of the cycle $c_{t_\eta, \eta_a}$ under the map $i_*$ equals $\xi_{\eta_1}^{(\eta)}(\xi_{\eta_a}^{(\eta)})^{-1}$ if $r=2$ or $\xi_{\eta_a}^{(\eta)}(\xi_{\eta_1}^{(\eta)})^{-1}$  if $r>2$.
\end{prop}

We consider the points $\theta_\kappa<t_\kappa$ very close to $t_\kappa$ as before Definition \ref{def:generating set}. 

\begin{lem}\label{lem:CyclestoMer} Let $\Gamma^{(\kappa)}=(\lambda_1^{(\kappa)},\ldots,\lambda_{n+1}^{(\kappa)})$ be a generating set as above. Then, for $\iota=1,\ldots,n+1$ we have that $(\lambda_\iota^{(\kappa)})^{{\xi_{\eta_a}^{(\kappa)}}^{-1}}$ is homotopic to a meridian of $L_{\sigma^{(\kappa)}(\iota)}$ at the point $x_{\iota}^{(\kappa)}=\bar{f}^{-1}{(\beta(\theta_\kappa)})\cap L_{\sigma^{(\kappa)}(\iota)}$  constructed by:
\begin{enumerate}
\item  following $h|_{\eta_a}(l_{\eta_a})$ until $\bar{f}^{-1}{(\beta(\theta_\kappa)})$,
\item  then joining it to a circle in $\bar{f}^{-1}{(\beta(\theta_\kappa)})$ about $x_{\iota}^{(\kappa)}$ and,
\item coming back via $h|_{\eta_a}(l_{\eta_a})$.
\end{enumerate}
See figure \ref{fig:DoubleBase2}.
\end{lem}

\begin{figure}[t]
        \centering
     
     \begin{tikzpicture} 

\draw (2.5,.5) node {\tiny $(\gamma_{\iota}^{(2)})^{{\tau^{(1)}(\eta_a)}^{-1}}$};

\draw (0,-0.2) node {\tiny $\vdots$};
\draw (0,-1.2) node {\tiny $x_{{\sigma^{(2)}}^{-1}(\eta_a)}^{(2)}$};

\draw (.1,1.6) node {\tiny $x_{\iota}^{(2)}$};

\draw (0,0.4) node {\tiny $x_{\iota-1}^{(2)}$};

\draw  (.5,-1.5) to[out=0,in=0] (0,2);
\draw  (0,2) to[out=180,in=180] (0,-.8);
\draw  (0,-.8) to[out=0,in=-20] (.5,1);

\draw [arc arrow=to pos 0.1 with length 2mm] (-.5,1) to[out=-90,in=-90] 
(.5,1) [arc arrow=to pos 0.6 with length 2mm] 
to[out=90,in=90] cycle;

\foreach \Point in {(0,-1),(0,0), (0,1)}{
    \node at \Point {\textbullet};
}


\end{tikzpicture}
\caption{A meridian follows another boundary manifold}
\label{fig:DoubleBase2}
  
\end{figure}
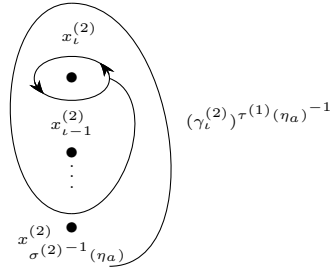
\begin{proof}
Note that if $\sigma^{(\kappa)}(\iota)=\eta_a$ and as $\lambda_{\iota}^{(\kappa)}=\lambda_{\eta_a}^{\xi_{\eta_a}^{(\kappa)}}$, by successive applications of Proposition \ref{prop:ConjTogether}, we can choose a meridian in the homotopy class of $(\lambda_{\iota}^{(\kappa)})^{{\xi_{\eta_a}^{(\kappa)}}^{-1}}=\lambda_{\eta_a}$ that  satisfies the properties stated in the Lemma (see figure \ref{fig:DoubleBase}).

Now, if $\sigma^{(\kappa)}(\iota)\not=\eta_a$, we proceed by induction. Let $t_1=W_j\cap W_{j+1}\cap \ldots\cap W_m$ and consider $$\xi_{\eta_a}^{(2)}=\tau^{(1)}(\eta_a)=\left\lbrace \begin{array}{ll}
\lambda_{\eta_a-1}^{(1)}\cdots \lambda_j^{(1)} & \text{if } n_a\in\{j+1,\ldots,m\},\\
1 & \text{if }\eta_a\not \in \{j+1,\ldots, m\}.\end{array}\right.$$
By construction, for $k\in \{j+1,\ldots,m\}$ the product $\lambda_k^{(1)}\ldots\lambda_j^{(1)}$ is freely homotopic to a circle containing the points $x_j^{(1)},\ldots,x_k^{(1)}$.

Now, note that the paths in $\Gamma^{(1)}$ are homotopic to paths in the fiber $\bar{f}^{-1}({\beta(\theta_2)})$ as in Figure \ref{fig:DoubleBase}. Such representative of the homotopy class of $\lambda_\iota^{(1)}$ can be seen as lying in the boundary manifold $\partial U_\iota^*$. 

By considering $(\lambda_\iota^{(2)})^{{\xi_{\eta_a}^{(2)}}^{-1}}$ we obtain a path as in figure \ref{fig:DoubleBase2} if $\tau^{(1)}(\eta_a)\not=1$. This meridian can be decomposed as stated.

For a general $\Gamma^{(\kappa+1)}$, note that as $\xi_{\eta_a}^{(\kappa+1)}=\xi_{\eta_a}^{(\kappa)}\cdot\tau^{(\kappa)}(\sigma^{{(\kappa)}^{-1}}(\eta_a))$ and by repeating the above procedure, we can decompose $(\lambda_\iota^{(\kappa+1)})^{{\tau^{(\kappa)}}^{-1}({\sigma^{(\kappa)}}^{-1}(\eta_a))}$ as a meridian of $L_{\sigma^{(\kappa+1)}(\iota)}$ that follows $l_{\eta_a}'$ between $\bar{f}^{-1}{(\beta(\theta_\kappa)})$ and $\bar{f}^{-1}({\beta(\theta_{\kappa+1})})$ (see figure \ref{fig:DoubleBase2}). By applying induction, we obtain that $(\lambda_\iota^{(\kappa+1)})^{{\xi_{\eta_a}^{(\kappa+1)}}^{-1}}$ can be decomposed as stated in the lemma.
\end{proof}

\begin{proof}[Proof of Proposition \ref{prop:CyclestoMeridians}]
Note that we have that \begin{align*}
\xi_{\eta_a}^{(\eta)}&=\tau^{(1)}(\eta_a)\cdots\tau^{(\eta-1)}({\sigma^{(\eta-1)}}^{-1}(\eta_a))\\
&=\tau^{(\eta-1)}\left({\sigma^{(\eta-1)}}^{-1}(\eta_a)\right)^{{\xi_{\eta_a}^{(\eta-1)}}^{-1}}\cdots\tau^{(2)}\left({\sigma^{(2)}}^{-1}(\eta_a)\right)^{{\xi_{\eta_a}^{(2)}}^{-1}}\tau^{(1)}(\eta_a).
\end{align*} 
by using $\xi_{\eta_a}^{(\kappa)}=\xi_{\eta_a}^{(\kappa-1)}\cdot \tau^{(\kappa-1)}\left({\sigma^{(\kappa-1)}}^{-1}(\eta_a))\right)$ for $\kappa=2,\ldots, \eta$.

Now, if $\tau^{(\kappa)}({\sigma^{(\kappa)}}^{-1}(\eta_a))\not =1$, it is homotopic to a path in $\bar{f}^{-1}{(\beta(\theta_\kappa)})$ encircling the points $x_j^{(\kappa)}, \ldots,$ $x_{{\sigma^{(\kappa)}}^{-1}(\eta_a)}^{(\kappa)}$ and by applying Lemma \ref{lem:CyclestoMer} to each factor of $\tau^{(\kappa)}({\sigma^{(\kappa)}}^{-1}(\eta_a))=\lambda_{{\sigma^{(\kappa)}}^{-1}(\eta_a)}^{(\kappa)}\ldots\lambda_j^{(\kappa)}$ we obtain that $\tau^{(\kappa)}({\sigma^{(\kappa)}}^{-1}(\eta_a))$ can be decomposed in three parts as in Lemma \ref{lem:CyclestoMer}.


Recall that we have constructed  the $l_{\eta_a}\subset D_{\eta_a}$ from a section of the map $\bar{f}$. By the choice of $h|_{\eta_a}$ we can suppose that $h_{\eta_a}\subset \bar{f}^{-1}(\beta[0,1])\cap \partial U_{\eta_a}$.

By considering $Y=\bar{f}^{-1}(\beta([0,1]))\setminus \pi_R^{-1}(R)\subset \Bl_R\mathbb{P}^2$, we can see the cycles $c_{t_\eta,\eta_a}\subset Y\subset \mathbb{R}^3$. Moreover, we can choose coordinates in $\mathbb{R}^3$ and define that $h|_{\eta_a}(l_{\eta_a})$ passes above $h|_{l_k}(l_k)$ (or $h|_{l_k}(l_k)$ passes below $h|_{l_{\eta_a}}(l_{\eta_a})$) in some fiber $\bar{f}^{-1}({\beta(\theta')})$ with $\theta'\in [t_\kappa-\varepsilon,t_\kappa+\varepsilon]$ with $\varepsilon>0$ sufficiently small, if the wires $W_{\eta_a}\cap W_k\not= \emptyset$ in a planar representation of the fiber $\bar{f}^{-1}({\beta(t_\kappa)})$ and ${\sigma^{(\kappa)}}^{-1}(\eta_a)<{\sigma^{(\kappa)}}^{-1}(k)$.

We can see then $(\tau^{(\kappa)}({\sigma^{(\kappa)}}(\eta_a)))^{{\xi_{\eta_a}^{(\kappa)}}^{-1}}$ as a path encircling the lines $L_k$ corresponding to those $h|_{l_k}(l_k)$ passing below $h|_{l_{\eta_a}}(l_{\eta_a})$ in some fiber $\bar{f}^{-1}({\beta'(\theta')})$ with $\theta'\in [t_\kappa-\varepsilon, t_\kappa+\varepsilon]$. By construction, $\xi_{\eta_a}^{(\eta)}$ is homotopic to a path encircling all the lines $L_k$ such that $h|_{l_k}(l_k)$ lies below $h_{l_{\eta_a}}(l_{\eta_a})$ at some point in $\beta([0,t_\eta])$.

Therefore, we can decompose $\xi_{\eta_a}^{(\eta)}$ in three parts:
\begin{enumerate}
\item The first path starting at $Q\in h(l_{n+1})$ and following $h(l_{\eta_a})$ until $\bar{f}^{-1}({\beta(\theta_\eta)})$. Then, 
\item a simple path starting at $h(l_{\eta_a})\cap \bar{f}^{-1}({\beta(\theta_\eta)})$, lying completely in $\bar{f}^{-1}({\beta(\theta_\eta)})$ and finishing at $\bar{f}^{-1}({\beta(\theta_\eta)})\cap \pi_R^{-1}(R)$, and
\item a path connecting $\pi_R^{-1}(R)\cap \bar{f}^{-1}({\beta(\theta_\eta)})$  to $Q\subset \pi_R^{-1}(R)$.
\end{enumerate}

By decomposing in a similar fashion $\xi_{\eta_1}^{(\eta)}$, it follows that the cycle $c_{t_\eta,\eta_a}$ is homotopic in $\mathbb{P}^2\setminus \A$ to $\xi_{\eta_1}^{(\eta)}(\xi_{\eta_a}^{(\eta)})^{-1}$  if $r=2$ and to $\xi_{\eta_a}^{(\eta)}(\xi_{\eta_1}^{(\eta)})^{-1}$ if $r>2$.
\end{proof}

\subsubsection{Expressing the relations in terms of the generators}\label{sss:ExpressingRelations} For every $r=n+2,\ldots,n+s_0+1$, we let $R_r'$ be the subset of the set of relations $R'$ of the presentation of $\pi_1(\partial U)$ as in Theorem \ref{thm:Westlund} such that $$R_r'=\{[\gamma_k,\gamma_r^{s_{kr}}], \gamma_r^{-w_r}=\prod_{m=1}^{k_r'}\gamma_{j_\Delta(r,m)}^{s_{rj_\Delta(r,m)}}\mid  (k,r)\in \mathcal{E}\} $$
with $$s_{kr}=\left\lbrace\begin{array}{ll}
c_\iota^{-1} & \text{ if } (k,r) \text{ equals the $\iota$-th element in } \mathcal{C},\\
c_\iota & \text{ if } (r,k) \text{ equals the $\iota$-th element in } \mathcal{C},\\
1 & \text{ if } (k,r) \text{ is an edge of } \mathcal{T}'.\\	 
\end{array} \right. $$ 
\begin{prop}\label{thm:guerville} Consider an exceptional divisor $E_\kappa=D_r\subset \bar{X}$ coming from a singular point $t_\kappa\in \Sing \A\cap M_\kappa$ of multiplicity higher or equal to $3$. The image of the set of relations $R_r'$ as above, under the map $i_*$, equals the set of relations $R_\kappa=[\lambda_m^{(\kappa)},\ldots,\lambda_j^{(\kappa)}]$ as in Lemma \ref{lem:ActualVertex}.
\end{prop}
\begin{proof}Let $t_\kappa=W_{\sigma^{(\kappa)}(j)}\cap W_{\sigma^{(\kappa)}(j+1)}\cap \ldots\cap W_{\sigma^{(\kappa)}(m)}$ with the local order given by $\Pi^{(\kappa)}=\{\sigma^{(\kappa)}(1)<\ldots<\sigma^{(\kappa)}(n+1)\}$. As $w_r=-1$ and by the local order of the wires we have that  $\gamma_r=\gamma_{\sigma^{(\kappa)}(m)}^{c_{\sigma^{(\kappa)}(m),t_\kappa}}\cdots \gamma_{\sigma^{(\kappa)}(j)}^{c_{\sigma^{(\kappa)}(j),t_\kappa}}$, that $[\gamma_r,\gamma_{\sigma(\kappa)}^{c_{\sigma(\kappa),t_\kappa}}]$ and that $c_{\sigma^{(\kappa)}(k),t_\kappa}=c_{t_\kappa,\sigma^{(\kappa)}(k)}^{-1}$.

Let us omit the superscript $\lambda_{\kappa}=\lambda_\kappa^{(1)}$ for the elements in $\Gamma^{(1)}$. 

By considering the image under $i_*$ of the elements in $R_r'$, we have by Lemma \ref{lem:EqMeridians} and by Proposition \ref{prop:CyclestoMeridians} that
$$i_*(\gamma_r)=\lambda_{\sigma^{(\kappa)}(m)}^{\xi_{\sigma^{(\kappa)}(m)}^{(\kappa)}{\xi_{\iota}^{(\kappa)}}^{-1}}\cdots \lambda_{\sigma^{(\kappa)}(j+1)}^{\xi_{\sigma^{(\kappa)}(j+1)}^{(\kappa)}{\xi_{\iota}^{(\kappa)}}^{-1}}\lambda_{\sigma^{(\kappa)}(j)}^{\xi_{\sigma^{(\kappa)}(j)}^{(\kappa)}{\xi_{\iota}^{(\kappa)}}^{-1}}, \quad [i_*(\gamma_r),\lambda_{\sigma^{(\kappa)}(k)}^{\xi_{\sigma^{(\kappa)}(k)}^{(\kappa)}{\xi_{\iota}^{(\kappa)}}^{-1}}]. $$
with $\iota=\min\{\sigma^{(\kappa)}(j),\sigma^{(\kappa)}(j+1),\ldots,\sigma^{(\kappa)}(m)\}$.

The commutators can also be written as $[i_*(\gamma_r)^{\xi_{\iota}^{(\kappa)}},\lambda_{\sigma^{(\kappa)}(k)}^{\xi_{\sigma^{(\kappa)}(k)}^{(\kappa)}}] $ . But as 
$i_*(\gamma_r)^{\xi_{\iota}^{(\kappa)}}=  \lambda_{\sigma^{(\kappa)}(m)}^{\xi_{\sigma^{(\kappa)}(m)}^{(\kappa)}}\cdots \lambda_{\sigma^{(\kappa)}(j+1)}^{\xi_{\sigma^{(\kappa)}(j+1)}^{(\kappa)}}\lambda_{\sigma^{(\kappa)}(j)}^{\xi_{\sigma^{(\kappa)}(j)}^{(\kappa)}}, $
 we have that the relations $[i_*(\gamma_r)^{\xi_{\iota}^{(\kappa)}},\lambda_{\sigma^{(\kappa)}(k)}^{\xi_{\sigma^{(\kappa)}(k)}^{(\kappa)}}]$ can be condensed as $[\lambda_{\sigma^{(\kappa)}(m)}^{\xi_{\sigma^{(\kappa)}(m)}^{(\kappa)}},\cdots, \lambda_{\sigma^{(\kappa)}(j+1)}^{\xi_{\sigma^{(\kappa)}(j+1)}^{(\kappa)}},\lambda_{\sigma^{(\kappa)}(j)}^{\xi_{\sigma^{(\kappa)}(j)}^{(\kappa)}}]$. 
  Now, if $R_\kappa=[\lambda_m^{(\kappa)}, \ldots,\lambda_{j+1}^{(\kappa)},\lambda_j^{(\kappa)}]$ denotes the relation given in Theorem \ref{thm:Arvola} for the point $t_\kappa$ as in Lemma \ref{lem:ActualVertex}, recall that we have the equality $\lambda_k^{(\kappa)}=\lambda_{\sigma^{(\kappa)}(k)}^{\xi_{\sigma^{(\kappa)}(k)}}$. By replacing it in the commutators above, the result follows.
\end{proof}

\begin{prop}\label{prop:Redundant} For $r=1,\ldots, n+1$, we have the equality $$i_*(\gamma_r^{w_r}\prod_{m=1}^{k_r'}\gamma_{j_\Delta(r,m)}^{s_{r j_\Delta(r,m)}})=\gamma_{n+1}\cdots\gamma_1
$$ in $\pi_1(\mathbb{P}^2\setminus \A)$ with $s_{rj}$ as above. 
\end{prop}
\begin{proof}

Fix $L_r\in \A$. 
Let $\{p_1,\ldots,p_b\}\subset \Sing \A\cap L_r$ be the singular points of the arrangement lying in $L_r$. Note that $b=k_r'$. Indeed, we can find a partition  $A\cup B=\{1,\ldots, k_r'\}$, with $A$ a set indexing the double points of $\Sing \A\cap L_r$, and $B$ indexing the points of multiplicity strictly bigger than two. Let $\pi:\bar{X}\to \mathbb{P}^2$ be the blow-up of $\mathbb{P}^2$ at the points of $\Sing \A$ of multiplicity strictly bigger than two and let $D_r$ denote the strict transform of $L_r$ in $\bar{X}$. We have that $A$ also indexes all the strict transforms of lines in $\A$ which have no empty intersection with $D_r$, and $B$ the exceptional divisors of $\bar{X}$ crossing $D_r$.	It is clear then that $b=k_r'$.

It follows that $\gamma_{j_\Delta(r,m)}$ is a meridian of an irreducible component $D_{j_\Delta(r,m)}$ of $D=\pi^* D$ for $m=1,\ldots, k_r'$. Recall that $\gamma_r$ commutes with $\gamma_{j_\Delta(r,m)}^{s_{r j_\Delta(r,m)}}$ and note that the self-intersection number $w_r$ of $D_r$ is $1-\abs{B}$. 

Let us study the geometric meaning of the product $\gamma_{r}^{-1}\gamma_{j_\Delta(r,m)}^{s_{r j_\Delta(r,m)}}$ with $m\in B$. Let us write $\iota=j_\Delta(r,m)$, denote by $D_{\iota}=D_{j_{\Delta}(r,m)}$ the exceptional divisor that $\gamma_{j_\Delta(r,m)}$ surrounds, and let $D_{\iota_1}, \ldots, D_{\iota_k}$ be the irreducibles components of $D=\pi^* \A$ that intersect $D_{\iota}$ ordered in such a way that, if we denote by $\gamma_{\iota_j}$ the meridians around $D_{\iota_j}$ used for the presentation of  $\pi_1(\partial U)$, $\gamma_{\iota}^{-w_{\iota}}=\gamma_{\iota_1}^{s_{\iota \iota_1}}\cdots\gamma_{\iota_k}^{s_{\iota \iota_k}}$ holds. As $D_{\iota}$ is an exceptional divisor, we have that $w_{\iota}=-1$.
By Theorem \ref{thm:Westlund}, we have that $[\gamma_{\iota},\gamma_{\iota_j}^{s_{\iota \iota_{j}}}]$ for $j=1,\ldots, k=k(\iota)$.

Replacing the expression $\gamma_{\iota}$ as above in $[\gamma_{\iota},\gamma_{\iota_j}^{s_{\iota \iota_j}}]$, we can show that these commutators relations are equivalent to $$\gamma_{\iota_1}^{s_{\iota \iota_1}}\gamma_{\iota_2}^{s_{\iota \iota_2}}\cdots\gamma_{\iota_k}^{s_{\iota \iota_k}}=\gamma_{\sigma(\iota_1)}^{s_{\iota\sigma(\iota_1)}}\gamma_{\sigma(\iota_2)}^{s_{\iota\sigma(\iota_2)}}\cdots\gamma_{\sigma(\iota_k)}^{s_{\iota\sigma(\iota_k)}}$$
where $\sigma$ runs over the cyclic permutations of the elements $\{\iota_1,\ldots,\iota_k\}$. Hence there exists some cyclic permutation $\sigma'$ such that $\sigma'(\iota_1)=r$ because $D_r$ intersects $D_{j_\Delta(r,m)}=D_{\iota}$. Note that $s_{r \iota}= s_{\iota r}^{-1}=s_{\iota \sigma'(\iota_1)}^{-1}$ and hence $\gamma_{r}^{-1}\gamma_{\iota}^{s_{r \iota}}=(\gamma_{\sigma'(\iota_2)}^{s_{\iota\sigma'(\iota_2)}}\cdots\gamma_{\sigma'(\iota_k)}^{s_{\iota\sigma'(\iota_k)}})^{s_{r\iota}} $ represents a loop which surrounds the lines $L_{\sigma'(\iota_2)},\ldots, L_{\sigma'(\iota_k)}$ following $l_r'$ by construction of the cycle $s_{r\iota}$.

Now, the product $\gamma_r^{w_r}\prod_{m=1}^{k_r'}\gamma_{j_\Delta(r,m)}^{s_{r j_\Delta(r,m)}}$ can be written as $\gamma_r \prod_{m=1}^{k_r'} \Upsilon_m$ with 
$$\Upsilon_m=\left\lbrace \begin{array}{ll}
\gamma_{j_\Delta(r,m)}^{s_{r j_\Delta(r,m)}} & \text{if } m\in A,\\
\gamma^{-1}\gamma_{j_\Delta(r,m)}^{s_{r j_\Delta(r,m)}} & \text{if } m\in B.
\end{array} \right. $$ 
by commuting $\gamma_r$ with $\gamma_{j_\Delta(r,m)}^{s_{r j_\Delta(r,m)}}$. Note that, for $\Upsilon_m$ with $m\in A$, the path $\Upsilon_m$ is a meridian around the other line that intersects $D_r$ in the double point corresponding to $m\in A$. Hence, by the precedent paragraph, $\gamma_r \prod_{m=1}^{k_r'} \Upsilon_m$ is a product of the meridians of all the lines in $\A$ ordered in the way they intersect $L_r$.

Now, by choosing a line $L$ sufficiently close to $L_r$ we have that the product  $\gamma_r^{w_r}\prod_{m=1}^{k_r'}\gamma_{j_\Delta(r,m)}^{s_{r j_\Delta(r,m)}}$ is a path encircling $L\setminus (L\cap \A)$ and therefore it is equivalent to $\lambda_{n+1}\cdots \lambda_1$ in $\pi_1(\mathbb{P}^2\setminus \A)$.
\end{proof}

\subsubsection{End of proof of the Theorem \ref{thm:Guerville}}

The point (1) of the Theorem is obtained by Lemma \ref{lem:EqMeridians}.

The point (2) follows from Proposition \ref{thm:guerville}.

For the point (3), recall that $R'$ denotes the set of relations for the presentation of $\pi_1(\partial U)$ as in Theorem \ref{thm:Westlund}. Using the notation of \ref{sss:ExpressingRelations} we have that $$R'\setminus \cup 
R_k'=\{\gamma_r^{w_r}\prod_{m=1}^{k_r'}\gamma_{j_\Delta(r,m)}^{s_{r j_\Delta(r,m)}}, [\gamma_k,\gamma_r^{s_{kr}}] \mid r=1,\ldots, n+1, D_k\cap D_r\in \Sing \A\setminus P_0\}$$
this is, $D_k\cap D_r$ is a double point.

By Proposition \ref{prop:Redundant}, we have that $i_*(\gamma_r^{w_r}\prod_{m=1}^{k_r'}\gamma_{j_\Delta(r,m)}^{s_{r j_\Delta(r,m)}})=\lambda_{n+1}\ldots \lambda_1$.

By proceeding as in Proposition \ref{thm:guerville}, it can be seen that for a double point $p_\kappa=D_k\cap D_r$, the relation $[\gamma_k,\gamma_r^{s_{kr}}]$ correspond to the relation $R_\kappa$ as in Theorem \ref{thm:Arvola}. 

Hence, in $\langle \gamma_1,\ldots, \gamma_{n+1}, c_1, \ldots, c_g \mid R', \cup_{p\in \Sing\A\setminus L_{n+1}} R'(p)\rangle$ the set of relations $R'$ is equivalent to the set of relations $\cup R_\kappa\cup \{\lambda_{n+1}\cdots \lambda_1\}$. 

This concludes the proof of Theorem \ref{thm:Guerville}.

\subsubsection{Independence of the maximal tree}\label{sss:Independence}
Let $D=\sum_{k=1}^{n+1+s_0}D_k$ the total transform of the arrangement $\A$ in $\bar{X}$ and denote by $\Delta$ the dual graph of $D$ as above. 

Let $\mathcal{T}\subset \Delta$ be an arbitrary maximal tree and denote by $G(\mathcal{T})=\{\gamma_1(\mathcal{T}),\ldots,$  $\gamma_{n+1+s_0}(\mathcal{T}),c_1(\mathcal{T}),\cdots, c_g(\mathcal{T})\}$, the set of generators of $\pi_1(\partial U)$ as in Theorem \ref{thm:Westlund}. Recall that these are constructed using $\mathcal{T}$. Denote by $R(\mathcal{T})$ the set of relations given in the same Theorem.

Consider also the maximal tree $\mathcal{T}'$ defined as in \ref{sss:MaximalTree} and denote by $\gamma_1,\ldots, \gamma_{n+1+s_0},$  $c_1,\ldots, c_g$ the generators of $\pi_1(\partial U)$ as in Theorem \ref{thm:Guerville} and by $R$ the set of relations.

Consider the inclusion $i:\partial U \hookrightarrow \mathbb{P}^2\setminus \A$ and fix $i_*(\gamma_1)=\lambda_1,\ldots, i_*(\gamma_{n+1+s_0})=\lambda_{n+1+s_0}$ as a set of generators for $\pi_1(\mathbb{P}^2\setminus \A)$ with $\Gamma^{(1)}=\{\lambda_{1}, \ldots, \lambda_{n+1}\}$ as in Theorem \ref{thm:Guerville}. For $\iota=1,\ldots, n+1+s_0$, we have that $i_*(\gamma_\iota(\mathcal{T}))$ and $\lambda_\iota$ are meridians of the same smooth curve $D_\iota$, therefore, we can express $i_*(\gamma_\iota(\mathcal{T}))$ as a conjugate of $\lambda_\iota$ by elements in $\lambda_1,\ldots, \lambda_{n+1}$.	We let $\delta_\iota$ denote the word in $\pi_1(\mathbb{P}^2\setminus \A)$ representing $i_*(c_\iota)$ in the letters $\lambda_1,\ldots,\lambda_{n+1+s_0}$ and by $\delta_\iota'$ the same word in the letters $\gamma_1,\ldots, \gamma_{n+1+s_0}$ as in Theorem \ref{thm:Guerville}.

Reciprocally, by fixing $i_*(\gamma_1(\mathcal{T})), \ldots, i_*(\gamma_{n+1+s_0}(\mathcal{T}))$ as generators of $\pi_1(\mathbb{P}^2\setminus \A)$, we can express $\lambda_\iota$ as a conjugate of $i_*(\gamma_\iota(\mathcal{T}))$ by elements in $i_*(\gamma_1(\mathcal{T})), \ldots,$  $i_*(\gamma_{n+1+s_0}(\mathcal{T}))$ for $\iota=1,\ldots, n+1$. The image $i_*(c_\iota(\mathcal{T}))$ of the cycle $c_\iota(\mathcal{T})$ can be expressed in terms of $i_*(\gamma_1(\mathcal{T})), \ldots, i_*(\gamma_{n+1+s_0}(\mathcal{T}))$ for $\iota=1,\ldots,g$. We let $\delta_\iota(\mathcal{T})$ be this expression when it is written in terms of $\gamma_1(\mathcal{T}), \ldots, \gamma_{n+1+s_0}(\mathcal{T})$ such that $\delta_\iota(\mathcal{T})\in \langle G(\mathcal{T})\mid R(\mathcal{T})\rangle$.

\begin{prop}\label{prop:EqPresentations} A presentation   of $\pi_1(\mathbb{P}^2\setminus \A)$ can be obtained as follows
$$\pi_1(\mathbb{P}^2\setminus \A)\cong \langle G(\mathcal{T})\mid R(\mathcal{T}),  c_1(\mathcal{T}) \cdot \delta_1(\mathcal{T})^{-1}, \cdots, c_g(\mathcal{T}) \cdot \delta_g(\mathcal{T}))^{-1}\rangle, $$
\end{prop}
\begin{proof}
The presentations $\langle G(\mathcal{T})\mid R(\mathcal{T})\rangle$ and $\langle \gamma_1,\ldots, \gamma_{n+1+s_0}, c_1,\ldots, c_g\mid R\rangle$ of $\pi_1(\partial U)$ as in Theorem \ref{thm:Westlund} can also be obtained as graphs of groups (see \cite{hironaka2000}). These graphs of groups are constructed over $\Delta$ as follows: the vertices groups are given as in Lemma \ref{lem:RelsInBundle}, the edges groups are $\mathbb{Z}^2$. To each tree of $\Delta$ there correspond a presentation and the presentations are Tietze-equivalent.

Let us fix $v_{n+1}$, the vertex corresponding to $D_{n+1}$ as a base point for $\pi_1(\Delta)$ and $c_1, \ldots, c_g$ a generating set. Every cycle $c_\iota(\mathcal{T})\in \pi_1(\Delta,v_{n+1})$ can be expressed as 
$c_\iota(\mathcal{T})=c_{\iota 1}\cdots c_{\iota r_\iota}$ where $c_{\iota m}\in \{c_1,\ldots,c_g\}$ with $m=1,\ldots, r_\iota$ and $\iota=1, \ldots,g$. Therefore $i_*(c_\iota(\mathcal{T}))=i_*(c_{\iota 1})\cdots i_*(c_{\iota r_\iota})=\delta_{\iota 1}\cdots \delta_{\iota r_\iota}$. Let us show that 
$$c_\iota(\mathcal{T})\cdot \delta_\iota(\mathcal{T})^{-1}=c_{\iota 1}\cdots c_{\iota r_\iota} {\delta_{\iota r_\iota}'}^{-1}\cdots {\delta_{\iota 1}'}^{-1}\in \langle\langle c_1\cdot {\delta_{\iota 1}'}^{-1},\ldots,  c_g\cdot {\delta_{g}'}^{-1}\rangle\rangle.$$
Note that 
\begin{align*}
(c_{\iota1} {\delta_{\iota 1}'}^{-1})(c_{\iota 2}{\delta_{\iota 2}'}^{-1})^{\delta_{\iota 1}^{-1}}&=c_{\iota 1} c_{\iota 2}{\delta_{\iota 2}'}^{-1}{\delta_{\iota 1}'}^{-1}\\
& \ \ \! \vdots \\
(c_{\iota 1} {\delta_{\iota 1}'}^{-1})\cdots (c_{\iota r_\iota }{\delta_{\iota r_\iota}'}^{-1})^{{\delta_{\iota r_\iota}'}^{-1}\cdots {\delta_{\iota 1}'}^{-1}}&=c_{\iota 1}\cdots c_{\iota \iota_r}{\delta_{\iota r_\iota}'}^{-1}\cdots {\delta_{\iota 1}'}^{-1}
\end{align*}
In a similar way we can prove that $c_\iota \cdot {\delta_{\iota}'}^{-1}\in \langle\langle c_1(\mathcal{T})\cdot \delta_1(\mathcal{T})^{-1}, \ldots, c_g(\mathcal{T})\cdot \delta_g(\mathcal{T})^{-1}\rangle\rangle$. This proves that the presentations $\langle G(\mathcal{T})\mid R(\mathcal{T}),  c_1(\mathcal{T}) \delta_1(\mathcal{T})^{-1}, \cdots, c_g(\mathcal{T})\delta_g(\mathcal{T}))^{-1}\rangle,$ and $\langle \gamma_1,\ldots, \gamma_{n+1+s_0}, c_1,\ldots, c_g\mid R, c_1{\delta_1'}^{-1}, \ldots, c_g{\delta_g'}^{-1}\rangle$ are equivalent. We conclude by Theorem \ref{thm:Guerville}.
\end{proof}

\subsection{Boundary manifold of a partial compactification}
Here we will present another presentation for the fundamental group of certain partial compactifications $M(\A,I,P)$, where $M(\A,I,P)$ is as in \ref{ss:partialCArrangements}, but the lines of $D$ indexed by $I$ correspond only to exceptional divisors, this is, $I\subset \{n+2,\ldots,n+1+s_0\}$ 

\subsubsection{Inclusion of the boundary of a partial compactification} 
Let us recall the notation of section \ref{ss:partialCArrangements}. 

Let $\A\subset \mathbb{P}^2$ be an arrangement of lines and $\bar{X}$ the blow-up at the points $P_0=\{p_1,\ldots,p_{s_0}\}$ of $\Sing \A$ with multiplicity strictly higher than two and let $D=\sum_{k=1}^{n+1+s_0} D_k$ be the reduced total transform of $\A$ in $\bar{X}$.  

Here, \emph{we suppose that} $I\subset \{n+2,\ldots, n+1+s_0\}$ and let $P=\{p_1',\ldots, p_{s_1}'\}\subset \Sing \sum_{k\not \in I} D_k$. Denote by $\pi':\Bl_P\bar{X}\to \bar{X}$ the blow-up map and the dual graph of $\abs{\pi'^* D}$ by $\Delta$. Note that in the previous section $\Delta$ denoted instead the dual graph of $D$. Consider the divisor $D'\subset \Bl_P\bar{X}$ as in \ref{ss:partialCArrangements} and denote by $\Delta'$ the dual graph of $D'$. Recall that $\Delta'$ is obtained from $\Delta$ by removing some vertices and the corresponding adjacent edges. 

In \ref{ss:partialCArrangements} we defined the partial compactification $M(\A,I,P)$ of $M(\A)$ as $\Bl_P\bar{X}\setminus D'$. 

Let us assume that $D'$ \emph{is connected}, which is equivalent to $\Delta'$ being connected. Therefore, there exists a maximal tree $\mathcal{T}_{\Delta'}\subset \Delta'$. Note that every cycle in $\Delta'$ can be seen as a cycle in $\Delta$.

\begin{lem} Any maximal tree $\mathcal{T}_{\Delta'}$ can be completed to a maximal tree $\mathcal{T}_{\Delta',\Delta}$ in $\Delta$.
\end{lem}
\begin{proof}
Let $\{v_1,\ldots,v_k\}$ be the vertices of $\Delta$ which are to be removed along with its adjacent edges in order to obtain $\Delta'$. 

 As $I\subset \{n+2,\ldots, n+1+s_0\}$ and $P\subset \Sing \sum_{\iota\not \in I} D_\iota$, we have that all the vertices in $\{v_1,\ldots, v_k\}$ correspond to exceptional divisors in $\Bl_P\bar{X}$, therefore there is no edge connecting $v_\iota$ and $v_j$ for $\iota\not = j$ and to complete $\mathcal{T}_{\Delta'}$ to a maximal tree of $\Delta$ it suffices to take no matter what edge connecting a vertex in $\mathcal{T}_{\Delta'}$ and $v_\iota$ for $\iota=1,\ldots,k$ because no cycle will be created in this way.
\end{proof}

\begin{cor} Let $g$ denote the number of independent cycles in $\Delta$. Let $c_1(\mathcal{T}_{\Delta'}),$ $\ldots,c_{g'}(\mathcal{T}_{\Delta'})$ be independent cycles in $\Delta'$ each one formed by adjoining one edge in $\Delta'$ to the maximal tree $\mathcal{T}_{\Delta'}$. There exists $c_{g'+1}(\mathcal{T}_{\Delta',\Delta}),\ldots, c_g(\mathcal{T}_{\Delta',\Delta})$ cycles in $\Delta$ that together with $c_1(\mathcal{T}_{\Delta',\Delta})=c_1(\mathcal{T}_{\Delta'}),\ldots,c_{g'}(\mathcal{T}_{\Delta',\Delta})=c_{g'}(\mathcal{T}_{\Delta'})$ complete a generating set of $\pi_1(\Delta,v_{n+1})$.
\end{cor}

Let us denote by $\partial U$ the boundary manifold of the total transform of $D$ in $\Bl_P\bar{X}$. By proceeding as in the proof of Theorem \ref{thm:Westlund}, we have that a presentation for $\pi_1(\partial U)$, by using the maximal tree $\mathcal{T}_{\Delta',\Delta}$, has generators $\gamma_1=\gamma_1(\mathcal{T}_{\Delta',\Delta}),\ldots,$  $\gamma_{n+1+s_0+s_1}=\gamma_{n+1+s_0+s_1}(\mathcal{T}_{\Delta',\Delta}), c_1=c_1(\mathcal{T}_{\Delta',\Delta}),\ldots, c_g=c_g(\mathcal{T}_{\Delta',\Delta})$ and a set of relations 
\begin{equation}\label{eq:RelsBoundaryDouble}
R=\left\lbrace 
 \begin{array}{ll}
 [\gamma_r,\gamma_j^{s_{rj}}], & (r,j)\in \mathcal{E}(\Delta) \\
 \gamma_r^{-w_r'}=\prod_{m=1}^{k_r'} \gamma_{j_{\Delta}(r,m)}^{s_{rj_{\Delta}(i,m)}} & 1\leq r \leq n+1+s_0+s_1
\end{array} \right\rbrace
\end{equation}
where $w_r'=D_r'\cdot D_r'$, for an irreducible component $D_r'$ of  $\pi'^*D$, we denoted by $k_r'$ the number of points in $\Sing \pi'^*D\cap D_r'$ (see the proof of Proposition \ref{prop:Redundant}), and 
$$s_{rj}=\left\lbrace\begin{array}{ll}
c_k^{-1} & \text{ if } (r,j) \text{ equals the $k$-th element in } \Delta\setminus \mathcal{T}_{\Delta',\Delta},\\
c_k & \text{ if } (j,r) \text{ equals the $k$-th element in } \Delta\setminus \mathcal{T}_{\Delta',\Delta},\\
1 & \text{ if } (r,j) \text{ is an edge of } \mathcal{T}_{\Delta',\Delta}.\\	 
\end{array} \right. $$ 

Moreover, let $\partial U'$ denote the boundary manifold of $D'\subset \Bl_P \bar{X}$. Here, if $r\not \in I$ let us denote by $k_r''$ the number of points in $(D_k\cap \sum_{\iota\not \in I} D_\iota)\setminus P$ or equivalently, in $D_k'\cap D'$. By using the maximal tree $\mathcal{T}_{\Delta'}$ of $\Delta'$ and proceeding as in the proof of Theorem \ref{thm:Westlund}, we obtain the following Proposition.

\begin{prop}\label{prop:PresBounPart}  A presentation for $\pi_1(\partial U')$ is given by 
$$\left\langle 
 \begin{array}{llll}
\gamma_\iota, \quad \iota \in J & \big\lvert [\gamma_r,\gamma_j^{s_{rj}}], & (r,j)\in \mathcal{E}(\Delta') \\
c_1, \ldots, c_{g'} & \bigg\lvert \gamma_r^{-w_r'}=\prod_{m=1}^{k_r''} \gamma_{j_{\Delta'}(r,m)}^{s_{rj_{\Delta'}(r,m)}} & r\in J
\end{array} \right\rangle $$
where $J= \{1,\ldots, n+1+s_0\}\setminus I$,  $\mathcal{E}(\Delta')$ denotes the set of edges of $\Delta'$, $w_r'$ the self-intersection number of the strict transform $D_r'$ of $D_r$ in $\Bl_P\bar{X}$ and 
$$s_{rj}=\left\lbrace\begin{array}{ll}
c_k^{-1} & \text{ if } (r,j) \text{ equals the $k$-th element in } \Delta'\setminus \mathcal{T}_{\Delta'},\\
c_k & \text{ if } (j,r) \text{ equals the $k$-th element in } \Delta'\setminus  \mathcal{T}_{\Delta'},\\
1 & \text{ if } (r,j) \text{ is an edge of } \mathcal{T}_{\Delta'}.\\	 
\end{array} \right. $$ 
\end{prop}

For every $\iota\in I$, we have that, as $D_\iota'$ is an exceptional divisor, the following relation is in $R$: 
\begin{equation}\label{eq:FirstWord}
\gamma_\iota=\prod_{m=1}^{k_\iota'}\gamma_{j_{\Delta',\Delta}(\iota,m)}^{s_{\iota j_{\Delta',\Delta}(\iota,m)}}.
\end{equation}

Analogously, if $p=p_\iota\in P$, by abuse of notation we will write $p={n+1+s_0+\iota}$. We have that if $p=D_r\cap D_j$:
\begin{equation}\label{eq:SecondWord}
\gamma_p=\gamma_r^{s_{pr}}\gamma_j^{s_{pj}}
\end{equation}


By using the map $i:\partial U\to M(\A)=\mathbb{P}^2\setminus \A$ as in \ref{sss:Independence},  we can express the image of the cycles $i_*(c_r)$ as a word in the letters $i_*(\gamma_1), \ldots, i_*(\gamma_{n+1})$, for $r=1,\ldots, g$. Let us denote by $\delta_r$ the word obtained by replacing the letters $i_*(\gamma_1),\ldots, i_*(\gamma_{n+1})$ by $\gamma_1,\ldots,\gamma_{n+1}$ in this precedent word associated to $i_*(c_r)$.

By using $\delta_1,\ldots, \delta_g$ and replacing $i_*(\gamma_1),\ldots,i_*(\gamma_{n+1})$ by $\gamma_1,\ldots, \gamma_{n+1}$, we can express the words $i_*(\prod_{m=1}^{k_\iota'}\gamma_{j_{\Delta',\Delta}(\iota,m)}^{s_{\iota j_{\Delta',\Delta}(\iota,m)}})$ and $i_*(\gamma_r^{s_{pr}}\gamma_j^{s_{pj}})$ with $\iota\in I$ and $p\in P$ as words $\gamma(\iota),\gamma(p)\in \pi_1(\partial U')$ respectively.

Let us denote by $R'$ the set of relations in the presentation given by Proposition \ref{prop:PresBounPart}.

\begin{thm}\label{thm:BoundaryLAC} A presentation of $\pi_1(M(\A,I,P))$ is given by  $$\langle c_1,\ldots, c_{g'}, \gamma_\iota \quad \iota\in J\mid R',c_1 \delta_1^{-1}, \ldots, c_{g'}\delta_{g'}^{-1}, \cup_{\iota \in I}\gamma(\iota),\cup_{p\in P}\gamma(p)\rangle$$
with $J=\{1,\ldots,n+1+s_0\}\setminus I$.
\end{thm}
\begin{proof}
Consider the following diagram: 
 \[
\begin{tikzcd}
\pi_1(\partial U)  \arrow[r, twoheadrightarrow]& \pi_1(M(\A))\arrow[d,twoheadrightarrow] & \pi_1(\partial U)/{\langle\langle c_1 \delta_l^{-1},\ldots, c_g \delta_g^{-1}\rangle\rangle} \arrow[l, "\cong "] \arrow[d]  \\
\pi_1(\partial U') \arrow[r, twoheadrightarrow] & \pi_1(M(\A,I,P)) &  {\pi_1(M(\A))}/{\langle\langle i_*(\gamma_\iota),i_*(\gamma_p)\rangle\rangle} \arrow[l, "\cong "]\end{tikzcd}
\]

Where the isomorphism in the right of the first row comes from Theorem \ref{thm:Guerville} and Proposition \ref{prop:EqPresentations}.

From the rightest column we obtain that 
\begin{equation}\label{eq:FirstPresQuo}
\pi_1(\partial U)/{\langle\langle c_1 \delta_1^{-1},\ldots, c_g \delta_g^{-1},\gamma(\iota),\gamma(p)\rangle\rangle} \cong \pi_1(M(\A,I,P)).
\end{equation} We will see that this presentation is equivalent to
\begin{equation}\label{eq:PresQuo}
\pi_1(\partial U')/\langle\langle \gamma(\iota),\gamma(p), c_1\cdot \delta_1^{-1},\ldots, c_{g'}\delta_{g'}^{-1}) \rangle\rangle 
\end{equation}

Indeed, by the choice of the maximal tree $\mathcal{T}_{\Delta',\Delta}$, the are only four types of relations in $R$ of the presentation of $\pi_1(\partial U)$ involving the cycles $c_{g'},\ldots, c_{g}$:
\begin{itemize}
\item commutators ${[\gamma_r,\gamma_{r'}^{s_{rr'}}]}$ with $r'=\iota,p$,
\item those relations as in (\ref{eq:FirstWord}), 
\item those relations as in (\ref{eq:SecondWord}), and
\item relations $ \gamma_r^{-w_r'}=\prod_{m=1}^{k_r'}\gamma_{j_{\Delta',\Delta}(r,m)}^{s_{r j_{\Delta',\Delta}(r,m)}}$
 with $(r,\iota)$ or $(r,p)$ an edge in $\Delta$.
\end{itemize}
By adding the relations $c_1=\delta_1,\ldots, c_g=\delta_g$, we can see these relations as expressed in terms of $\gamma_1,\ldots, \gamma_{n+1+s_0}$.

Note that the commutator-relation as in the first point above becomes trivial in $\pi_1(\partial U)/{\langle\langle c_1 \delta_1^{-1},\ldots, c_g \delta_g^{-1},\gamma(\iota),\gamma(p)\rangle\rangle}$.

The relations in the points two and three above, are by construction, equivalent to the words $\gamma(\iota), \gamma(p)$.

For the relations as in the fourth point, note that $k_r''=k_r'-\abs{P\cap D_r} -\{\iota\in I\mid D_\iota\cap D_r\not = \varnothing\}$.

\end{proof}

\section{Applications and examples}
\subsection{Preliminary results in homology planes}
As shown in section \ref{sec:Wirtinger}, the wiring diagram $\mathcal{W}$ of an arrangement $\A$ can be used to determine the meridians around the exceptional divisors $E_p$, corresponding to a point $p\in \Sing \A$, in terms of the meridians of the lines in $\A$.

Here, we apply Theorem \ref{thm:Arvola} to obtain presentations for the fundamental group of a very special type of partial compactifications, which we proceed to describe.

\begin{defn} A ($\mathbb{Q}$-)homology plane $X$ is an affine smooth complex surface such that the $i$-th group of (rational) integer homology ($H_i(X,\mathbb{Q})$) $H_i(X,\mathbb{Z})$ vanishes for $i=1,\ldots,4$.
\end{defn}

It was proved in \cite{gurjarshastri} that homology planes are rational. This result also holds for $\mathbb{Q}$-homology planes, see \cite{gurjarpradeep}.

Let $X$ be a homology plane. There exists a projective smooth rational surface $\bar{X}$, a birational morphism $\pi:\bar{X} \to \mathbb{P}^2$, and a divisor $D\subset \bar{X}$ such that $\bar{X}\setminus D=X$.

A classification of the arrangements of lines $\A\subset \mathbb{P}^2$ such that there exists a pair $(\bar{X},\pi)$ as above and $\pi(D)=\A$ is given in \cite{Dieck1989}. We call such an arrangement $\A$ a \emph{linear plane divisor} of the homology plane $X$. There exists six arrangements of lines $L(2),\ldots, L(7)$ and an infinite family of arrangements $L(1,n+1)$ which are linear plane divisors, each one for an infinite family of homology planes \cite[Theorem D]{Dieck1989}.

In \cite{tomdieck1993}, an algorithm for constructing homology planes out of these arrangements is given. We describe it briefly. 

\subsubsection{}\label{sss:algorithm} Let $\A=\{L_1,\ldots,L_{n+1}\}$ be an arrangement of lines. We denoted by $P_0=\{p_1,\ldots, p_{s_0}\}\subset \Sing \A$ the set of points of multiplicity strictly higher than two, by $\pi^{(0)}:\bar{X}=\Bl_{P_0}\mathbb{P}^2\to \mathbb{P}^2$ the blow-up map, by $D_i$ the strict transform of $L_i$ for $i=1,\ldots, n+1$ and by $D_{n+1+j}$ the exceptional divisor associated to $p_j$ for $j=1,\ldots,s_0$. Let $I\subset \{n+2,\ldots,n+1+s_0\}$ and consider the divisor $D'=\sum_{i=1}^{n+1+s_0}D_i-\sum_{i\in I}D_i$.

Denote by $\Delta'$ the dual graph of $D'$ and suppose that the number of independent cycles in $\Delta'$ equals $m$, this is $b_1(\Delta')=m$. Consider $P_1\subset \Sing D'$ a subset of $m$-points such that when we remove from $\Delta'$ the edges corresponding to $P_1$ we obtain a maximal tree. Let $\pi^{(1)}:\Bl_{P_1}\bar{X}\to \bar{X}$ be the blow-up at $P_1$ and denote by $D''$ the strict transform of $D'$ in $\Bl_{P_1}\bar{X}$. 

\begin{prop}{\cite[Proposition 2.1]{tomdieck1993}}\label{prop:countthelines}  The surface $\Bl_{P_1}\bar{X}\setminus D''$ is a $\mathbb{Q}$-homology plane if and only if the inclusion $D''\hookrightarrow \Bl_{P_1}\bar{X}$ induces an isomorphism $H_2(D'',\mathbb{Q})\to H_2(\Bl_{P_1}\bar{X},\mathbb{Q})$.
\end{prop}
We obtain that a necessary condition for $\Bl_{P_1}\bar{X}\setminus D''$ as above, to be a $\mathbb{Q}$-homology plane is that the number $n+1+s_0-\abs{I}$ of irreducible components in $D''$ must be equal to $b_2(\Bl_{P_1}\bar{X})=\dim(H_2(\Bl_{P_1}\bar{X},\mathbb{Q}))=m+\abs{P_0}+1$. It follows that in this case, $n=m+\abs{I}$. We will describe when this condition is as well sufficient.

Let $\lambda_1,\ldots,\lambda_{n+1}$ be meridians of the lines $L_1,\ldots, L_{n+1}\in \A$ respectively. A basis for $H_1(\mathbb{P}^2\setminus \A)$ is given by the homology clases $[\lambda_1],\ldots,[\lambda_{n}]$ and they satisfy that
$\sum_{i=1}^{n+1}[\lambda_i]=0$. Given any exceptional divisor $D_{n+1+j}$ in $\bar{X}$ such that the corresponding point $p_j\in P_0$ satisfies $p_j=L_{j_1}\cap L_{j_2}\cap \ldots\cap L_{j_r}$, we have that the homology class of a meridian $\lambda_{n+1+j}$ of $D_{n+1+j}$ is given by $[\lambda_{n+1+j}]=[\lambda_{j_1}]+\ldots+[\lambda_{j_r}]$. Similarly, for $p\in P_1$ such that $p	=D_i\cap D_j$ we have that the homology class of a meridian $\lambda_p$ of the exceptional divisor $E_{p}$ in $\Bl_{P_1}\bar{X}$ is given by $[\lambda_p]=[\lambda_i]+[\lambda_j]$. It follows that for every $k\in I$ we can express $[\lambda_{n+1+k}]=\sum_{r=1}^n a_{k,r}[\lambda_r]$ and for $p\in P_1$ we have $[\lambda_p]=\sum_{r=1}^n a_{p,r} [\lambda_r]$. Define a matrix $M=(a_{q,r})$ with $q\in I\cup P_1$.

Now, suppose that a surface $X$ is constructed as in \ref{sss:algorithm} by choosing $I$ and $P_1\subset \Sing D'$ such that $X=\Bl_{P_1}\bar{X}\setminus D''$ and that the number of irreducible components in $D''$ equals $b_2(\Bl_{P_1}\bar{X})$. As $n=m+\abs{I}$, the matrix $M$ is a square matrix.

\begin{prop}{\cite[Theorem B]{tomdieck1993}}\label{prop:MatrixCrit} The surface $X$ is a $\mathbb{Q}$-homology plane if and only if $\det M\not =0$. Moreover, if $\det M\not =0$, $H_1(X,\mathbb{Z})=\abs{\det M}$ and hence it is a homology plane if and only if $\abs{\det M}=1$.
\end{prop}

The surfaces obtained by the construction given in \ref{sss:algorithm} are usually only $\mathbb{Q}$-homology planes. In order to obtain homology planes, further blow-ups are required.

\subsubsection{}\label{sss:expansions} Consider $D_i, D_j$ in $\bar{X}$ and let $p=D_i\cap D_j$. Let $a_i,a_j$ be two coprime positive integers. There exists a sequence of blow-ups
$$\Bl_{p_{k-1},\ldots,p}\bar{X}\overset{\pi_{ij}^{(k)}}\rightarrow  \Bl_{p_{k-2},\ldots,p}\bar{X} \overset{\pi_{ij}^{(k-1)}}\rightarrow  \ldots\overset{\pi_{ij}^{(2)}}\rightarrow\Bl_{p}\bar{X}\overset{\pi_{ij}^{(1)}}\rightarrow \bar{X} $$
where each $\pi_{ij}^{(r)}$ is the blow-up of $\Bl_{p_{r-1},\ldots,p}\bar{X}$ at a point $p_r\in \Bl_{p_{r-1},\ldots,p}\bar{X}$, with $p_r$ lying in the exceptional divisor $E_r$ corresponding to $\pi_{ij}^{(r-1)}$ and in the singular locus of the total transform of $D_i+D_j$ in $\Bl_{p_{r-1},\ldots,p}\bar{X}$, such that the multiplicity of the exceptional divisor $E_k\in \Bl_{p_{k-1},\ldots,p}\bar{X}$ in $(\pi_{ij}^{(k)}\circ\cdots\circ\pi_{ij}^{(1)})^*(D_i)$ and in $(\pi_{ij}^{(k)}\circ\cdots\circ\pi_{ij}^{(1)})^* (D_j)$ is $a_i,a_j$ respectively, see \cite[Theorem 4.11]{tomdieck1993}. By \cite[Lemma 7.18]{fujita} the meridian around $E_k$ is given by $\lambda_i^{a_i}\lambda_j^{a_j}$. See also \cite[Appendix]{AST_1993__217__251_0}.   Moreover, if $p=D_i\cap D_j\in P_1$, using this construction we can replace the line in the matrix $M$ corresponding to $[\lambda_i]+[\lambda_j]$ by $a_i[\lambda_i]+b_i[\lambda_j]$. 

Let $e_p$ be the edge in the dual graph $\Delta'$ of the divisor $D'$ corresponding to $p$. We call this construction \emph{expanding the edge} $v_p$. In order to describe the change in the dual graph, let us denote a continuous fraction by $[c_1,\ldots,c_r]$ defined by $[c_1]=c_1$ and $[c_1,\ldots,c_r]=c_1-[c_2,\ldots,c_r]^{-1}$. We have that $c_i\geq 2$. Let
\begin{equation*}
\frac{a_i+a_j}{a_j}=[c_{-r},\ldots,c_{-1}], \quad
\frac{a_i+a_j}{a_i}=[c_s,\ldots,c_1]
\end{equation*}
Suppose that $D_i^2=b_i, D_j^2=b_j$. The edge connecting $v_i$ and $v_j$ changes as in Figure \ref{fig:ExpandingEdge}.
\begin{figure}[h]
       \centering
      
\includegraphics[scale=1]{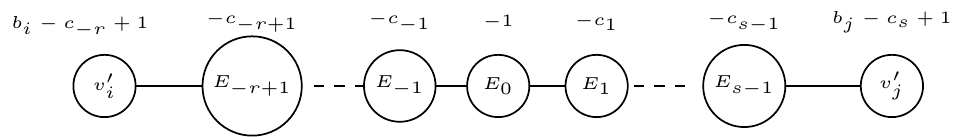}
\caption{Expanding an edge}
        \label{fig:ExpandingEdge}
\end{figure}

This will be abbreviated as in Figure \ref{fig:ExpandingEdgeReduced}. 
\begin{figure}[h]
       \centering
\includegraphics[scale=1]{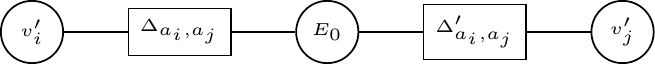}
\caption{Notation for expanding an edge}
        \label{fig:ExpandingEdgeReduced}
\end{figure}

\subsubsection{Absolutely minimal graphs}
\begin{defn} Let $\Delta$ be a weighted graph. We say that $\Delta$ is absolutely minimal if the weight of every linear or terminal vertex does not exceed $-2$. See \cite{AST_1993__217__251_0,Zai94}.
\end{defn}

\begin{defn} Let $(\bar{X},D)$ be a pair with $\bar{X}$ a smooth projective surface and $D$ a simple normal crossing divisor such that all its irreducible components are rational curves. We say that $(\bar{X},D)$ is absolutely minimal if the dual graph $\Delta$ of $D$ is absolutely minimal.
\end{defn}

\begin{prop}[\cite{Zai94}] \label{prop:ZaidUniqueness} Let $X$ be a quasi-projective smooth surface and $(\bar{X},D)$ an absolutely minimal pair such that $\bar{X}\setminus D=X$. Any other absolutely minimal completion of $X$ is isomorphic to $(\bar{X},D)$.
\end{prop}
\subsection{Examples}
\subsubsection{The arrangement $L(1,n+1)$} For $n \in \mathbb{N}$ such that $n\geq 3$ consider the arrangement $L(1,n+1)$ of $n+1$ lines where $n$ of them intersect in a point and the other one is in general position. See figure \ref{fig:LinesL1} for the representation in the projective plane and figure \ref{fig:WiringL1} for its wiring diagram.

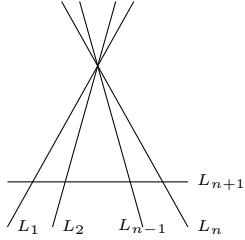
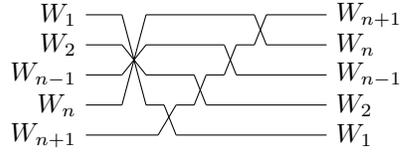
\begin{figure}[h]
\begin{subfigure}[b]{0.4\textwidth}
       \centering
     \begin{tikzpicture}[scale=1.2]

\draw (-.4,2)  -- (1,-.5)node[right]{\tiny $L_n$};
\draw (-.2,2) -- (.5,-.5)node{\tiny $L_{n-1}$};
\draw (.2,2) -- (-.5,-.5)node[above,right]{\tiny $L_2$};
\draw (.4,2) -- (-1,-.5)node[above,right]{\tiny $L_1$};

\draw (-1,0) -- (1,0) node[right]{\tiny $L_{n+1}$};



\end{tikzpicture}
\caption{$L(1,5)$ in the projective plane}
\label{fig:LinesL1}
    \end{subfigure}
~ ~ ~~~  
\quad         
\begin{subfigure}[b]{0.45\textwidth}
        \centering
        \begin{tikzpicture}[scale=.8]
\draw(3.5,0)--(6,0)node [anchor=west] {$W_1$};
\draw(3.3,.5)-- (3.5,0);
\draw(3.3,.5)-- (3,.5);
\draw(2.6,2)-- (3,.5);
\draw(2,2) node[anchor=east]{$W_1$} --(2.6,2);

\draw (4,.5)--(6,.5)node [anchor=west]{$W_2$} ;
\draw (3.8,1)  -- (4,.5);
\draw (3.8,1)  -- (3,1);
\draw (2.6,1.5)  -- (3,1);
\draw(2,1.5)node[anchor=east]{$W_2$}--(2.6,1.5);

\draw (4.5,1)--(6,1)node [anchor=west]{$W_{n-1}$};
\draw (4.3,1.5)  -- (4.5,1);
\draw (3,1.5)  -- (4.3,1.5);
\draw (3,1.5)  -- (2.6,1);
\draw(2,1) node[anchor=east]{$W_{n-1}$} --(2.6,1);

\draw(5,1.5)--(6,1.5)node[anchor=west]{$W_n$} ;
\draw(5,1.5)  -- (4.8,2);
\draw(3,2)  -- (4.8,2);
\draw(3,2)  -- (2.6,.5);
\draw (2,.5)node[anchor=east]{$W_{n}$}--(2.6,.5);

\draw(5,2)--(6,2)node[anchor=west]{$W_{n+1}$};
\draw(5,2)--(4.8,1.5);
\draw(4.5,1.5)--(4.8,1.5);
\draw(4.5,1.5)--(4.3,1);
\draw(4,1)--(4.3,1);
\draw(4,1)--(3.8,.5);
\draw(3.5,.5)--(3.8,.5);
\draw(3.5,.5)--(3.2,0);
\draw(2,0) node[anchor=east]{$W_{n+1}$} --(3.2,0);

\end{tikzpicture}
        \caption{Wiring Diagram L(1,5)}
        \label{fig:WiringL1}
    \end{subfigure}

    \caption{The arrangement L(1,5)}\label{fig:L1}
\end{figure}

A presentation for $\pi_1(\mathbb{P}^2\setminus L(1,n+1))$ is given by: 
\begin{equation}\label{eq:PresentationL(1,n+1)}
\langle \lambda_1, \ldots, \lambda_{n+1}\mid \lambda_{n+1}\cdots \lambda_1, {[\lambda_n,\ldots, \lambda_1]},  [\lambda_r,\lambda_{n+1}] \  r=1,\ldots, n \rangle 
\end{equation}

An expression for the meridian $\lambda_{n+2}$ around the exceptional divisor $D_{n+2}$, obtained by blowing-up the unique point $p$ of multiplicity $n$ in $\A$, is given by $\lambda_{n+2}=\lambda_n\cdots \lambda_1$. Note that $\lambda_{n+2}=\lambda_{n+1}^{-1}$.

Let $\pi:\bar{X}=\Bl_{p}\mathbb{P}^2\to \mathbb{P}^2$, $D=\abs{\pi^*\A}=\sum_{i=1}^{n+2} D_i$ and $\Delta$ the dual graph of $D$. In order to obtain a maximal tree of $\Delta$, that give rise to homology planes, we need to remove $n-1$ edges from $\Delta$: each one corresponding either to $D_i\cap D_{n+1}$ or to $D_i\cap D_{n+2}$ for $i=2,\ldots, n$. We can expand these edges as in \ref{sss:expansions}. In doing so, for every edge that we expand and every pair of coprime positive integers $a_{i1},a_{i_2}$ we need to add the relations $\lambda_i^{a_{i1}}\lambda_{n+1}^{\pm a_{i2}}$ to the presentation of $\pi_1(\mathbb{P}^2\setminus L(1,n+1))$.

Using the notation of \ref{sss:algorithm}, as $\abs{I}=0$, we have that $b_1(\Delta)=n-1$. In order to obtain that $n=m$, we have to blow-up a smooth point in the line $D_1$ (and possibly several times in a point infinitely near). We have to add the relation $\lambda_1^{a_{11}}$. 

\begin{prop} Let $\Gamma$ be a group presented by $<\lambda_1, \ldots, \lambda_n\mid \lambda_n\cdots\lambda_1, \lambda_i^{a_{i1}},i=1,\ldots, n >$. The fundamental group of any homology plane $X$ arising from the arrangement $L(1,n+1)$ as above, admits an exact sequence $$1 \to N \to \pi_1(X)\to \Gamma\to 1 $$ with $N$ a cyclic group.
\end{prop}
\begin{proof}
Note that from the presentation in (\ref{eq:PresentationL(1,n+1)}), it follows that $\lambda_{n+1}$ generates a central group. Denote by $N$ the image of the cyclic group generated by $\lambda_{n+1}$. By taking the quotient of $\pi_1(X)$ by this group we obtain the presentation given by $\Gamma$.
\end{proof}

In fact, all the homology planes of logarithmic Kodaira dimension one arise in this way, see \cite{GURJAR198899,weko_39408_1} and  \cite[Chapter 3.4]{miyanishiopen}.
\begin{lem}\label{lem:DualGraphL1} The dual graph of the divisor at infinity for a homology plane $X$ arising from $L(1,n+1)$ as above, is absolutely minimal and has the first form if $k>1$ and the second if $k=1$:

\begin{figure}[h]
       \centering
\begin{subfigure}[b]{0.45\textwidth}
        \centering     
        \scalebox{0.85}{ 
\includegraphics[scale=1]{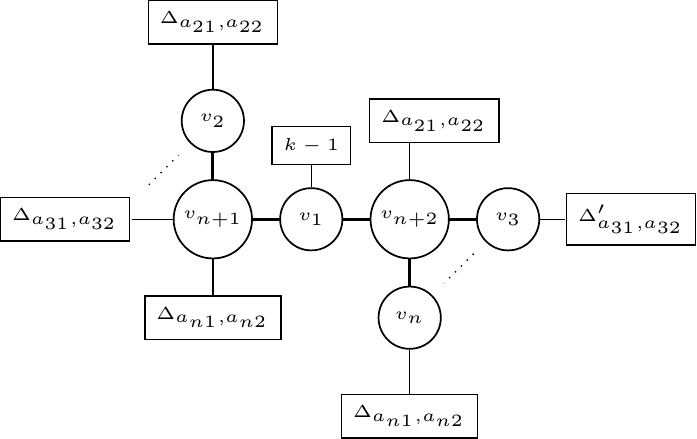}}
\caption{Dual graph of homology planes from L(1)}
        \label{fig:Dual Graph L1}
\end{subfigure}
\quad         
\begin{subfigure}[b]{0.45\textwidth}
        \centering
        \scalebox{0.85}{\includegraphics[scale=1]{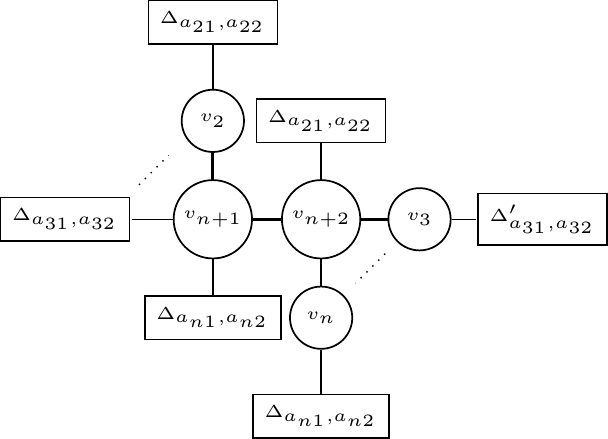}}
\caption{Second type of dual graph of homology planes from L(1)}
        \label{fig:Second Dual Graph L1}     
    \end{subfigure}
             
\end{figure}
with the square above $v_1$ representing a linear chain of $k-1$-vertices with weight $-2$.
\end{lem}
These graphs include those corresponding to contractible homology planes arising from $L(1,n+1)$ studied in \cite{AST_1993__217__251_0}.
\subsubsection{The arrangement $L(2)$}
The arrangement $L(2)$ is constructed by four lines in general position. 

\begin{figure}[h]
\begin{subfigure}[b]{0.4\textwidth}
       \centering
     \begin{tikzpicture}[scale=.8]

\draw (-1.2,-1.1017) node{\tiny $L_1$} -- (1,.5) ;
\draw (0.5,1.5) -- (-1,-1.3)node[above,right]{\tiny $L_2$};
\draw (-.5,1.5) -- (1,-1.3)node[above,right]{\tiny $L_3$};

\draw (-1.666,.6192) -- (1.564,-.8187)node[above,right]{\tiny $L_4$};
\end{tikzpicture}
\caption{$L(2)$ in the projective plane}
\label{fig:LinesL2}
    \end{subfigure}
~ ~ ~~~  
\quad         
\begin{subfigure}[b]{0.45\textwidth}
        \centering
        \begin{tikzpicture}[scale=.8]
\draw(4.5,0)--(5.5,0)node [anchor=west] {$W_1$};
\draw(4.5,0)--(4.3,.5);
\draw(4,.5)--(4.3,.5);
\draw(4,.5)--(3.8,1);
\draw(3.5,1)--(3.8,1);
\draw(3.5,1)--(3.2,1.5);
\draw(2,1.5) node[anchor=east]{$W_1$} --(3.2,1.5);

\draw (4.5,.5)--(5.5,.5)node [anchor=west]{$W_2$} ;
\draw (4.5,.5)--(4.3,0) ;
\draw (3,0)--(4.3,0) ;
\draw(3,0)--(2.8,.5);
\draw(2.5,.5)--(2.8,.5);
\draw(2.3,1)--(2.5,.5);
\draw(2,1)node[anchor=east]{$W_2$}--(2.3,1);

\draw(5,1)--(5.5,1)node[anchor=west]{$W_3$} ;
\draw(5,1)  -- (4.8,1.5);
\draw(3.5,1.5)  -- (4.8,1.5);
\draw(3.5,1.5)  -- (3.3,1);
\draw(3.3,1)  -- (3,1);
\draw(3,1)  -- (2.5,1);
\draw(2.5,1)  -- (2.3,.5);
\draw (2,.5)node[anchor=east]{$W_3$}--(2.3,.5);

\draw(5,1.5)--(5.5,1.5)node[anchor=west]{$W_4$};
\draw(5,1.5)--(4.8,1);
\draw(4,1)--(4.8,1);
\draw(4,1)--(3.8,.5);
\draw(3,.5)--(3.8,.5);
\draw(3,.5)--(2.8,0);
\draw(2,0) node[anchor=east]{$W_4$} --(2.8,0);

\end{tikzpicture}
        \caption{Wiring Diagram L(2)}
        \label{fig:WiringL2}
    \end{subfigure}
    
    \caption{The arrangement L(2)}\label{fig:L2}
\end{figure}
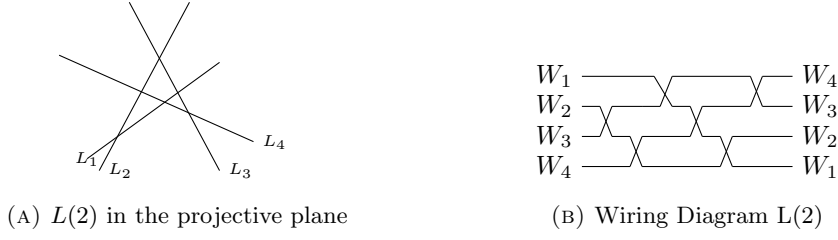

The fundamental group $\pi_1(\mathbb{P}^2\setminus L(2))$ is isomorphic to $\mathbb{Z}^3$ and admits the following presentation 
$$\left\langle\lambda_1,\ldots,\lambda_4\mid \lambda_4\cdots \lambda_1=1, [\lambda_i,\lambda_j] \ i<j, i=1,\ldots,3, j=2,\ldots, 4  \right\rangle $$

The homology planes arising from this arrangement are treated in \cite[3.16]{tomdieck1993}.

\subsubsection{The arrangement $L(3)$} This arrangement is constructed from $L(2)$ by adding a line $L_5$ passing through two double points of $L(2)$. See figure \ref{fig:LinesL3}.

The fundamental group $\pi_1(\mathbb{P}^2\setminus L(3))$ is isomorphic to $\mathbb{F}_2\times \mathbb{F}_2$ and admits the following presentation that can be read from the wiring diagram \ref{fig:WiringL3}:
\begin{equation}\label{eq:presL3}
\left\langle \lambda_1,\ldots, \lambda_5 \mid {[\lambda_5,\lambda_4,\lambda_3]}, [\lambda_5,\lambda_2,\lambda_1],  [\lambda_4,\lambda_1], {[\lambda_4,\lambda_2]}, 
{[\lambda_3,\lambda_1], [\lambda_3, \lambda_2]}, \lambda_5\cdots\lambda_1 
\right\rangle 
\end{equation}
\begin{figure}[h]

\begin{subfigure}[b]{0.4\textwidth}
       \centering
     \begin{tikzpicture}[scale=.8]

\draw (-1.2,-1.1017) node{\tiny $L_1$} -- (1,.5) ;
\draw (0.5,1.5) -- (-1,-1.3)node[above,right]{\tiny $L_2$};
\draw (-.5,1.5) -- (1,-1.3)node[above,right]{\tiny $L_3$};
\draw (-1.666,.6192) -- (1.564,-.8187)node[above,right]{\tiny $L_4$};
\draw (-1.5,-1) -- (1.95,.15)node[above,right]{\tiny $L_5$};



\end{tikzpicture}
\caption{$L(3)$ in the projective plane}
\label{fig:LinesL3}
    \end{subfigure}
~ ~ ~~~  
\quad         
\begin{subfigure}[b]{0.45\textwidth}
        \centering
        \begin{tikzpicture}[scale=.8]
\draw(4.7,0)--(6,0)node [anchor=west] {$W_1$};
\draw (4.4,1)  -- (4.7,0);
\draw(4.1,1)--(4.4,1);
\draw (3.9,1.5)  -- (4.1,1);
\draw (3.1,1.5)--(3.9,1.5);
\draw(2.9,2)  -- (3.1,1.5);
\draw(0.4,2) node[anchor=east]{$W_1$} --(2.9,2);

\draw (2.1,.5)--(6,.5)node [anchor=west]{$W_2$} ;
\draw (1.9,1)  -- (2.1,.5);
\draw (1.1,1)--(1.9,1);
\draw(0.9,1.5)  -- (1.1,1);
\draw(0.4,1.5)node[anchor=east]{$W_2$}--(0.9,1.5);

\draw (5.5,1)--(6,1)node [anchor=west]{$W_3$};
\draw (5.2,2)  -- (5.5,1);
\draw (3.1,2)  -- (5.2,2);
\draw (2.9,1.5)  -- (3.1,2);
\draw (1.1,1.5)--(2.9,1.5);
\draw (.9,1)  -- (1.1,1.5);
\draw(0.4,1) node[anchor=east]{$W_3$} --(.9,1);

\draw(4.1,1.5)--(6,1.5)node[anchor=west]{$W_4$} ;
\draw(3.9,1)  -- (4.1,1.5);
\draw (2.1,1)--(3.9,1);
\draw(1.9,.5)  -- (2.1,1);
\draw (0.4,.5)node[anchor=east]{$W_4$}--(1.9,.5);

\draw(5.5,2)--(6,2)node[anchor=west]{$W_5$};
\draw(5.2,1)--(5.5,2);
\draw(4.7,1)--(5.2,1);
\draw(4.4,0)--(4.7,1);
\draw(0.4,0)node[anchor=east]{$W_5$}--(4.4,0);
\end{tikzpicture}
        \caption{Wiring Diagram L(3)}
        \label{fig:WiringL3}
    \end{subfigure}

    \caption{The arrangement L(3)}\label{fig:L3}
\end{figure}
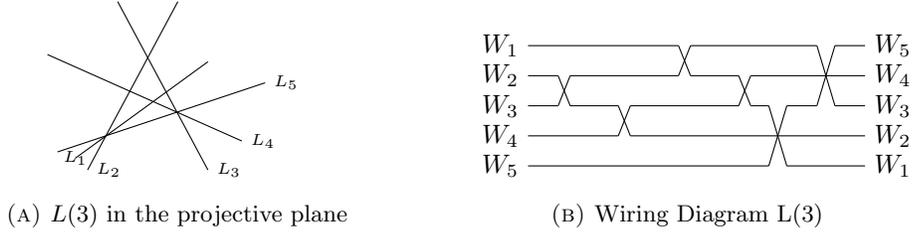

Let $\gamma_6$ and $\gamma_7$ be meridians around the exceptional divisors $D_6$ and $D_7$, obtained by blowing up $L_5\cap L_4\cap L_3$ and $L_5\cap L_2\cap L_1$ respectively. Then we have that $\gamma_6=\gamma_5\gamma_4\gamma_3$ and  $\gamma_7=\gamma_5\gamma_2\gamma_1$.

Denote by $f_1,f_2$ the fibrations given by the pencils of lines obtained by blowing up the points $L_5\cap L_4\cap L_3$ and $L_5\cap L_2\cap L_1$ respectively. Let $X$ be a homology plane arising from $L(3)$ as in \ref{sss:algorithm}. Denote by $f_1',f_2'$ the respective fibrations induced in $X$. The multiple fibers of $f_1',f_2'$ induce orbifold structures over $\mathbb{P}^1$ that we denote by $\mathcal{X}_1=\mathcal{X}_1(\mathbb{P}^1,D,r_1)$ and $\mathcal{X}_2=\mathcal{X}_2(\mathbb{P}^1,D,r_2)$ with $D=0+1+\infty$ and $r_1=(a_1,b_1,c_1), r_2=(a_2,b_2,c_2)$ depending on $f_1',f_2'$.

\begin{prop}\label{prop:L4Structure}
If $X$ is a homology plane arising from $L(3)$ as in \ref{sss:algorithm} such that both $\pi_1(\mathcal{X}_1),\pi_1(\mathcal{X}_2)$ are finite, then $\pi_1(X)$ is also finite. Moreover, if $\pi_1(\mathcal{X}_1)=\pi_1(\mathcal{X}_2)=\langle 1\rangle$ then $\pi_1(X)=\langle 1\rangle$.

\end{prop}
\begin{proof}

Suppose that $\pi_1(\mathcal{X}_1)$ and $\pi_1(\mathcal{X}_2)$ are finite. Recall that $\pi_1(\mathbb{P}^2\setminus L(3))=\mathbb{F}_2\langle\lambda_1,\lambda_2\rangle\times\mathbb{F}_2\langle\lambda_3,\lambda_4\rangle$. Denote by $i:\mathbb{P}^2\setminus L(3)\hookrightarrow X$ the inclusion and consider $H_1:=i_*(\langle \lambda_1,\lambda_2\rangle )$ and $H_2:=i_*(\langle\lambda_3,\lambda_4\rangle)$. We have the following exact sequences induced by $f_1',f_2'$: 
\begin{align*}
1\to H_1\to\pi_1(X)\to \pi_1(\mathcal{X}_1)\to 1, \quad
1\to H_2\to\pi_1(X)\to \pi_1(\mathcal{X}_2)\to 1.
\end{align*}
Consider the subgroup $A=H_1\cap H_2$, it is normal and has finite index. Moreover, $A$ is abelian: let $x,y\in A$ and write $x=w_1(i_*(\lambda_1),i_*(\lambda_2))$ and $y=w_2(i_*(\lambda_3),i_*(\lambda_4))$ with $w_1$ and $w_2$ words in the letters $i_*(\lambda_1),i_*(\lambda_2)$ and $i_*(\lambda_3),i_*(\lambda_4)$ respectively. As $\lambda_j$ commutes with $\lambda_k$ for $j=1,2$ and $k=3,4$, we obtain that $x$ commutes with $y$. The same reasoning shows that $A$ is central.

Now, as $\Gamma$ is a perfect group, we have from the universal coefficient Theorem that $H^2(\Gamma,\mathbb{Z})\cong \Hom(H_2 \Gamma, \mathbb{Z})$. Then, as $\Gamma$ is finite, we have that $H^2(\Gamma,\mathbb{Z})=H^1(\Gamma,\mathbb{Q}/\mathbb{Z})$. As $\Gamma$ is perfect, it operates trivially on $\mathbb{Q}/\mathbb{Z}$ and therefore $H^1(\Gamma,\mathbb{Q}/\mathbb{Z})=\Hom(\Gamma,\mathbb{Q}/\mathbb{Z})=0$, which shows that $H_2(\Gamma,\mathbb{Z})$ is a torsion group. 

Consider the central extension $1\to A\to \pi_1(X)\to \Gamma\to 1$. Write $A\cong \mathbb{Z}^n\oplus A_{\text{tors}}$ and consider the universal central extension $1\to H_2(\Gamma,\mathbb{Z})\to E \to \Gamma\to 1$. This extension comes with a map $H_2(\Gamma,\mathbb{Z})\to A_{\text{tors}}$ that factors through $A$. Now, we have that $$\pi_1(X)= [E\times A]/\Delta H_2(\Gamma,\mathbb{Z})\cong \mathbb{Z}^n\times [E\times A_{\text{tors}}]/\Delta H_2(\Gamma,\mathbb{Z})$$
but as $\pi_1(X)$ is perfect, we have that $n=0$ and therefore $A$ is a torsion group.

Now, if $\pi_1(\mathcal{X}_1)=\pi_1(\mathcal{X}_2)=\langle 1\rangle$ then $H_1=\pi_1(X)=H_2$ and therefore $A=\pi_1(X)$ is an abelian group and perfect, hence $\pi_1(X)=\langle 1 \rangle$.
\end{proof}

Let $P_0=\{L_4\cap L_4\cap L_3, L_5\cap L_2\cap L_1\} $ and consider $\pi:\bar{X}=\Bl_{P_0}\mathbb{P}^2\to \mathbb{P}^2$.  Denote by $D$ the reduced divisor $\pi^*(L(3))$ and its dual graph by $\Delta(3)$. Note that $b_1(\Delta(3))=\abs{E}-\abs{V}+1=4$. The maximal trees contained in $\Delta(3)$ are encoded by subsets $P_1\subset \Sing D$ of four elements, corresponding to independent edges.

\begin{exmp} We can consider the homology planes $X$ arising from expandig the edges corresponding to the double points $P_1=\{D_1\cap D_3, D_1\cap D_4, D_2\cap D_3, D_2\cap D_4\}$ in $L(3)$. All the homology planes arising in this way will satisfy that $\pi_1(\mathcal{X}_1)=\pi_1(\mathcal{X}_2)=\langle 1\rangle$ and by Proposition \ref{prop:L4Structure}, we have that $\pi_1(X)=\langle 1\rangle$.

 There are indeed an infinite collection of homology planes arising from the maximal tree associated to these double points $P_1$. They were studied by Zaidenberg in \cite{AST_1993__217__251_0}. He also obtained that an infinite number of these homology planes have logarithmic Kodaira dimension equal to two.
\end{exmp}

Now we present examples of homology planes arising from $L(3)$ with infinite fundamental group 

\begin{exmp}\label{exmp:L3}
Consider $P_1=\{D_1\cap D_3,D_1\cap D_4, D_2\cap D_4, D_5\cap D_6\}$, a presentation for the fundamental group of any homology plane $X$ arising from this configuration has the following form: 

\begin{equation}\label{eq:expL3.1}
\langle \lambda_1,\ldots,\lambda_5 \mid\lambda_1^a\lambda_3^b, \lambda_1^c\lambda_4^d, \lambda_2^e\lambda_4^f,\lambda_5^g\lambda_6^h,[\lambda_i,\lambda_j] \text{ for } i=1,2, \ j=3,4  \rangle. \end{equation}
Note that $\lambda_5\cdot \lambda_6=(\lambda_1^{-1}\cdots\lambda_4^{-1})(\lambda_1^{-1}\lambda_2^{-1})$.

The fundamental group of the orbifolds $\mathcal{X}_1, \mathcal{X}_2$ over $\mathbb{P}^1$ induced by $f_1'$ and $f_2'$ have presentations $\pi_1(\mathcal{X}_1)=\langle\lambda_3,\lambda_4 \mid \lambda_3^b,\lambda_4^{\gcd(d,f)},(\lambda_4\lambda_3)^{g}\rangle$ and $\pi_1(\mathcal{X}_2)=\langle\lambda_1,\lambda_2 \mid \lambda_1^{\gcd(a,c)},\lambda_2^{e},(\lambda_2\lambda_1)^{g+h} \rangle$ respectively.

By Proposition \ref{prop:MatrixCrit}, we have that $X$ is a homology plane is the following determinant equals one
$$\det \left(\begin{array}{rrrr}
a & 0 & b & 0 \\
c & 0 & 0 &d \\
0 & e & 0 & f \\
-g-h & -g-h & -g & -g 
\end{array}\right) = -a d e g - b c e g + b c f g + b c f h + b d e g + b d e h $$

\begin{enumerate}
\item 
We can choose the values $a,\ldots, h$ such that $\pi_1(\mathcal{X}_1)$ and $\pi_1(\mathcal{X}_2)$ are both infinite hyperbolic triangle groups: 
Consider $b=3, d=f=5$ and $g=7$. As above, $\pi_1(\mathcal{X}_1)=\langle \lambda_3,\lambda_4 \mid \lambda_3^3, \lambda_4^5, (\lambda_3^{-1}\lambda_4^{-1})^7 \rangle$.

There exists indeed solutions for $\det M=1$ and the above values of $b,d,f,g$. For example $a=c=11, e=14, h=16$. We have that $\pi_1(\mathcal{X}_2)=\langle \lambda_1, \lambda_2 \mid \lambda_1^{11}, \lambda_2^{14}, (\lambda_1^{-1}\lambda_2^{-1})^{23} \rangle$.

A presentation for $\pi_1(X)$ is obtained readily from (\ref{eq:expL3.1}) by replacing the values of the exponents.
\item We can choose as well the values $a,\ldots, h$ in such a way that $\pi_1(\mathcal{X}_1)$ is finite non-trivial but $\pi_1(\mathcal{X}_2)$ is infinite: for $b=2,d=f=3,g=5$ we have that $\pi_1(\mathcal{X}_1)=\langle\lambda_3,\lambda_4\mid \lambda_3^2,\lambda_4^3, (\lambda_3\lambda_4)^5 \rangle$.

In order to obtain a solution for $\det M=1$ we can choose: $a=c=25, e=47, h=63$ and therefore $\pi_1(\mathcal{X}_2)=\langle\lambda_1,\lambda_2\mid \lambda_1^{25},\lambda_2^{47},(\lambda_1\lambda_2)^{68} \rangle$

\end{enumerate}

For the proof of the following Lemma, we follow closely the arguments of \cite{AST_1993__217__251_0}.
\begin{lem} The homology planes constructed in this example are of log-general type and the divisor at infinity has dual graph that are absolutely minimal as in Figure \ref{fig:DualGraphL3}.
\begin{figure}[h]
       \centering
\includegraphics[scale=1]{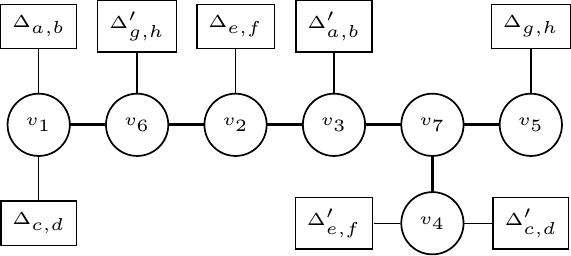}
\caption{Dual graph}
        \label{fig:DualGraphL3}
\end{figure}
\end{lem}
\begin{proof}
The squares with $\Delta_{i,j}$ inside represents a linear chain with vertices of weights at most $-2$, therefore the graph are absolutely minimal. By counting the vertices of valence $3$, we can see that this graphs are not equivalent to that of Lemma \ref{lem:DualGraphL1}.

There are not homology planes $X$ with $\bar{k}(X)=-\infty$ other that $\mathbb{C}^2$ \cite[Theorem 3.2]{fujita1979zariski} and no homology planes at all with $\bar{k}(X)=0$ \cite{miyanishiopen}. As all the homology planes of log-kodaira dimension equal one can be obtained from $L(1,n+1)$, we conclude using Proposition \ref{prop:ZaidUniqueness}.
\end{proof}

\end{exmp}



\subsubsection{The arrangement $L(4)$}
We consider the real arrangement $L(4)$ defined by adding a line $L_6$ to $\Ceva(2)=\{(x:y:z)\in \mathbb{P}^2\mid (x^2-y^2)(x^2-z^2)(y^2-z^2)=0\}$ passing through $2$ double points, see Figure \ref{fig:L(4)}.

The fundamental group $\pi_1(\mathbb{P}^2\setminus L(4))$ admits the following presentation that can be obtained from the wiring diagram in Figure \ref{fig:WiringL4}
$$\left\langle \lambda_1,\ldots, \lambda_7 \mid \begin{array}{l}
[\lambda_7,\lambda_6], [\lambda_7,\lambda_5,\lambda_4], [\lambda_7,\lambda_3], [\lambda_7,\lambda_2,\lambda_1], [\lambda_5^{\lambda_4},\lambda_3,\lambda_1],[\lambda_6,\lambda_4,\lambda_1] \\
{[\lambda_6,\lambda_3^{\lambda_1}]}, [\lambda_6,\lambda_5^{\lambda_4},\lambda_2^{\lambda_1}], [\lambda_4,\lambda_3,\lambda_2], \lambda_7\lambda_6\ldots \lambda_2\lambda_1=1
\end{array}  \right \rangle $$

\begin{figure}[h]
\begin{subfigure}[b]{0.4\textwidth}
       \centering
     \begin{tikzpicture}[scale=1]

\draw(-1.2,-1.1017) node{\tiny $L_1$} -- (1,.5) ;
\draw(0.5,1.5) -- (-1,-1.3)node[above,right]{\tiny $L_2$};
\draw (-0.115,1.51) -- (.254,-1.52)node[above,right]{\tiny $L_3$};
\draw (-.5,1.5) -- (1,-1.3)node[above,right]{\tiny $L_4$};

\draw (-1.666,.6192) -- (1.564,-.8187)node[above,right]{\tiny $L_5$};
\draw (-1.7,0) -- (2,0) node[right,below]{\tiny $L_6$};
\draw (-1.5,-1) -- (1.95,.15)node[above,right]{\tiny $L_7$};



\end{tikzpicture}
\caption{$L(4)$ in the projective plane}
\label{fig:LinesL4}
    \end{subfigure}
~ ~ ~~~  
\quad         
\begin{subfigure}[b]{0.45\textwidth}
        \centering
        \begin{tikzpicture}[scale=.8]
\draw(4,0)--(6,0)node [anchor=west] {$W_1$};
\draw(3.7,1)-- (4,0);
\draw(3.5,1)-- (3.7,1);
\draw(3.2,2)-- (3.5,1);
\draw(3,2)-- (3.2,2);
\draw(3,2)-- (2.7,3);
\draw(0.4,3) node[anchor=east]{$W_1$} --(2.7,3);

\draw (2,.5)--(6,.5)node [anchor=west]{$W_2$} ;
\draw (1.7,1.5)  -- (2,.5);
\draw (1.7,1.5)  -- (1.5,1.5);
\draw (1.2,2.5)  -- (1.5,1.5);
\draw(0.4,2.5)node[anchor=east]{$W_2$}--(1.2,2.5);

\draw (4.5,1)--(6,1)node [anchor=west]{$W_3$};
\draw (4.3,1.5)  -- (4.5,1);
\draw (2.5,1.5)  -- (4.3,1.5);
\draw (2.5,1.5)  -- (2.3,2);
\draw(0.4,2) node[anchor=east]{$W_3$} --(2.3,2);

\draw(5,1.5)--(6,1.5)node[anchor=west]{$W_4$} ;
\draw(5,1.5)  -- (4.7,2.5);
\draw(1.5,2.5)  -- (4.7,2.5);
\draw(1.5,2.5)  -- (1.2,1.5);
\draw (0.4,1.5)node[anchor=east]{$W_4$}--(1.2,1.5);

\draw(3.5,2)--(6,2)node[anchor=west]{$W_5$};
\draw(3.5,2)--(3.2,1);
\draw(0.4,1) node[anchor=east]{$W_5$} --(3.2,1);

\draw(5.5,2.5)--(6,2.5)node[anchor=west]{$W_6$};
\draw(5.5,2.5)--(5.3,3);
\draw(3,3)--(5.3,3);
\draw(3,3)--(2.7,2);
\draw(2.5,2)--(2.7,2);
\draw(2.5,2)--(2.3,1.5);
\draw(2,1.5)--(2.3,1.5);
\draw(2,1.5)--(1.7,.5);
\draw (0.4,.5)node[anchor=east]{$W_6$}--(1.7,.5);

\draw(5.5,3)--(6,3)node[anchor=west]{$W_7$};
\draw(5.5,3)--(5.3,2.5);
\draw(5.3,2.5)-- (5,2.5);
\draw(5,2.5)-- (4.7,1.5);
\draw(4.7,1.5)-- (4.5,1.5);
\draw(4.5,1.5)-- (4.3,1);
\draw(4.3,1)-- (4,1);
\draw(3.7,0)-- (4,1);
\draw(0.4,0)node[anchor=east]{$W_7$}--(3.7,0);

\end{tikzpicture}
        \caption{Wiring Diagram L(4)}
        \label{fig:WiringL4}
    \end{subfigure}

    \caption{The arrangement L(4)}\label{fig:L(4)}
\end{figure}
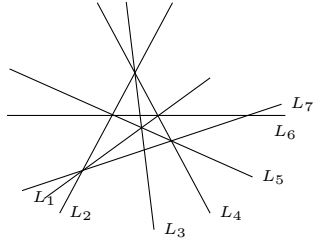
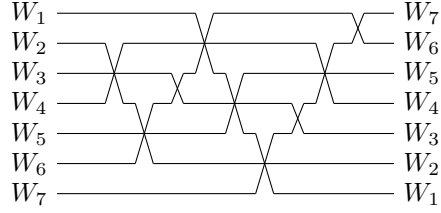

The expressions for the meridians $\lambda_8,\ldots, \lambda_{13}$ around the exceptional divisors $D_8,\ldots, D_{13}$ respectively are: 
\begin{equation}\label{eq:ExceptionalL(4)}
\begin{array}{lll}
\lambda_8=\lambda_7\lambda_5\lambda_4 & \lambda_9=\lambda_7\lambda_2\lambda_1 & \lambda_{10}=\lambda_5^{\lambda_4}\lambda_3 \lambda_1 \\
\lambda_{11}=\lambda_6\lambda_4\lambda_1 & \lambda_{12}=\lambda_6\lambda_5^{\lambda_4}\lambda_2^{\lambda_1} & \lambda_{13}=\lambda_4^{\lambda_1}\lambda_3^{\lambda_1}\lambda_2^{\lambda_1}
\end{array}
\end{equation}

For the following examples of homology planes we will use the computations appearing in \cite{inproceedings} of the first characteristic variety of the arrangements and its associated pencils of lines. We will use these pencils to construct maps to orbicurves with infinite perfect fundamental group.

The arrangement $L(4)$ is called the non-Fano plane in \cite[Example 10.5]{inproceedings}. It has six local pencils corresponding to the six triple points and three pencils corresponding to (braid) $\Ceva(2)$  subarrangements.

\begin{exmp}\label{exmp:L4.1} Consider the pencil $\Pi=C_{(12\mid 45 \mid 36)}$ obtained when we blow-up the singular points corresponding to $(L_1\cup L_2)\cap (L_4\cup L_5)$. Recall that the ordering of the singular points is given by reading the triple points in the wiring diagram in Figure \ref{fig:WiringL4} from right to left. Therefore the exceptional divisors corresponding to $(L_1\cup L_2)\cap (L_4\cup L_5)$ are $D_{10},D_{11},D_{12},D_{13}$. These divisors are sections of $\Pi$ as well as the strict transform $D_7$ of $L_7$.

 With the notation of \ref{sss:algorithm}, let us consider $I=\{11,12,13\}$ and the set of edges to be expanded $P_1=\{D_1\cap D_{9}, D_{5}\cap D_{10}, D_6\cap D_{7}\}$ as in \ref{sss:expansions}. Consider the meridians $\lambda_1^a\lambda_9^b,(\lambda_5^{\lambda_4})^c\lambda_{10}^d,\lambda_6^e\lambda_7^f$ of the respective exceptional divisors coming from $P_1$. 
 
 The pencil $\Pi$ induces an orbifold structure  $\mathcal{X}=\mathcal{X}(\mathbb{P}^1,D,r)$ with $D= 0+ 1+\infty$ and $r=(2b+a, c, e)$.

 From (\ref{eq:ExceptionalL(4)}), we have that $$\det \left(\begin{array}{rrrrrr}
0 & 1 & 1 & 1&0 &0 \\
0 & 1 & 0 &0 &1 &1\\
1 & 0 & 0 & 1 & 0 & 1 \\
a & 0 & -b & -b & -b &-b \\
d &0&d &0 &c+d &0 \\
-f & -f & -f &-f &-f &e-f 
\end{array}\right)=a c e - a c f + 2 a d e - 2 b c f + 4 b d e$$ for $M$ the matrix as in \ref{sss:algorithm}.

By considering the following solution for $|\det M|=1$: $a=1,b=2,c=7,d=1,e=2$ and $f=1$ we obtain that $\pi_1(\mathcal{X})$ is infinite. We can also obtain the following presentation for the fundamental group of the homology plane $X$: 
$$\pi_1(X)=\langle \lambda_2,\lambda_5\mid \lambda_2^4=\lambda_5\lambda_2\lambda_5, \lambda_5^7=\lambda_2\lambda_5\lambda_2^2\lambda_5\lambda_2 \rangle $$
Moreover, the image of the fibers of $\Pi$ in $\pi_1(X)$ is a cyclic group generated by the image of $\lambda_7$.

\begin{lem} The homology plane $X$ in this example is of log-general type with absolutely minimal dual graph of the divisor at infinity as follows
\begin{figure}[h]
       \centering
\includegraphics[scale=1]{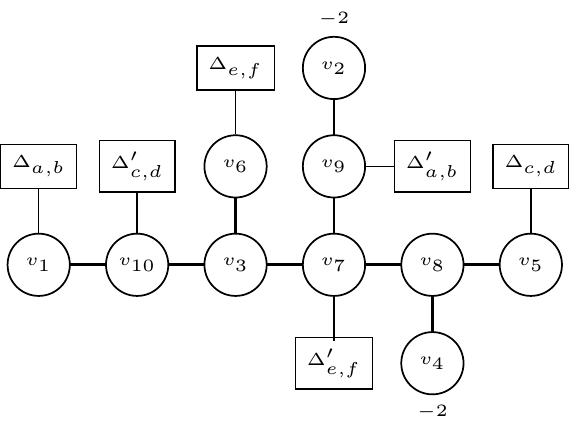}
\caption{Dual graph}
        \label{fig:DualGraphL4.1}
\end{figure}

\end{lem}
\end{exmp}
\begin{exmp}\label{exmp:L4.2} Using again the pencil $\Pi=C_{(12\mid 45 \mid 36)}$, we can construct homology planes fibered over $\mathbb{P}^1$ with general fiber $\mathbb{P}^1\setminus \{3-\text{points}\}$ as follows:
take $I=\{8,11,12\}$ and $P_1=\{D_1\cap D_9, D_3\cap D_{13}, D_6\cap D_7\}$. 

The associated matrix $M$ will have determinant equal to $$\det M=-a c e - a c f + 2 a d f + 4 b d e + 4 b d f$$
The pencil $\Pi$ induces a structure of orbifold $\mathcal{X}=\mathcal{X}(\mathbb{P}^1,D,r)$ with $D$ as before and $r=(a+2b,\gcd(c,e),2)$.

If we take $c=e$, a solution for $\det M=1$ is given by: $a=1,b=2,c=e=9,d=1,f=10$. Note that the meridians $\lambda_7,\lambda_{10},\lambda_{13}$ generate the fundamental group of a fiber of the induced pencil.

The fundamental group $\pi_1(X)$ admits the following presentation 
$$ \langle \lambda_1,\lambda_6,\lambda_7\mid \begin{array}{l}
 (\lambda_7, \lambda_6) = 1,  \quad 
    \lambda_1\lambda_7^{-1}\lambda_6\lambda_1^{-1}\lambda_7^{-1}\lambda_1\lambda_7\lambda_6^{-1}\lambda_1^{-1}\lambda_7 = 1 \\
    \lambda_1\lambda_6\lambda_1\lambda_7^{-1}\lambda_6\lambda_1^{-1}\lambda_6^{-1}\lambda_1^{-1}\lambda_7\lambda_6^{-1} =    1 \\
    \lambda_7\lambda_6^{-1}\lambda_1^{-1}\lambda_7\lambda_6^{-1}\lambda_1^2\lambda_7\lambda_6^{-1}\lambda_1^{-1}   \lambda_6^{-1}\lambda_7\lambda_1 = 1 \\
    \lambda_1\lambda_7\lambda_1^{-1}\lambda_6\lambda_7^{-1}\lambda_1\lambda_6^{-1}\lambda_1^{-1}\lambda_6^{-1}\lambda_7   \lambda_1^{-1}\lambda_6\lambda_7^{-1}\lambda_1 = 1, \
    \lambda_6^7\lambda_7\lambda_6\lambda_7^9\lambda_6 = 1\\
    (\lambda_6\lambda_7^{-1})^6\lambda_1\lambda_6\lambda_7^{-1}\lambda_1   \lambda_7^{-1}\lambda_1^{-1}\lambda_7\lambda_6^{-1}\lambda_1^{-1}(\lambda_6\lambda_7^{-1})^2\lambda_1   \lambda_7^{-1}\lambda_1^{-2}\lambda_7\lambda_6^{-1}\lambda_1^{-1}= 1 \end{array}  \rangle $$
In $\pi_1(X)$ the element $\lambda_{10}$ equals $\lambda_1  \lambda_7^{	-1}  \lambda_6  \lambda_1  \lambda_6  \lambda_7^{-1}
$ and $\lambda_{13}$ equals $\lambda_6^{-1}  \lambda_1^{-1}  \lambda_7  \lambda_6^{-1}  \lambda_7^{-1}  \lambda_1^{-1}$.
\begin{figure}[h]
       \centering
\includegraphics[scale=1]{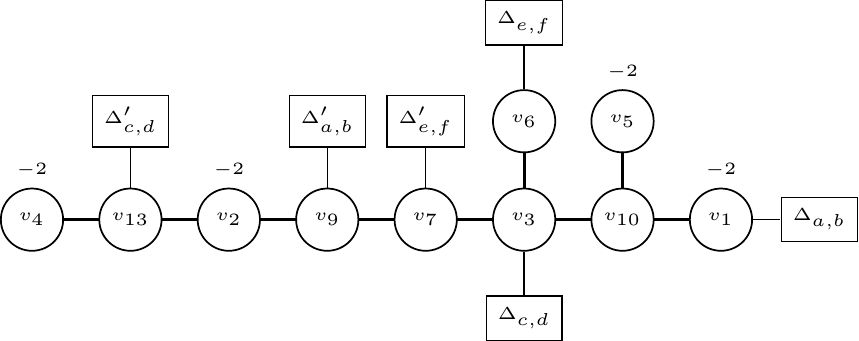}
\caption{Dual graph}
        \label{fig:DualGraphL4.2}
\end{figure}
\end{exmp}

\begin{rem} We can use the two other pencils $\Pi_2=C_{(15\mid 24\mid 67)}$ and $\Pi_3=C_{(14\mid 37 \mid 25)}$ to obtain similar groups. 
\end{rem}
\subsubsection{The arrangement $L(5)$} The arrangement $L(5)$ is the first of two arrangements of $9$ lines with $6$ double points, $6$ triple points and $2$ quadruple points, see Figure \ref{fig:LinesL5}.
 
 The fundamental group $\pi_1(\mathbb{P}^2\setminus L(5))$ admits the following presentation that is obtained using the wiring diagram in Figure \ref{fig:WiringL5}
$$ \langle \lambda_1,\ldots,\lambda_9 \mid \begin{array}{l}
[\lambda_9,\lambda_8,\lambda_7,\lambda_6], [\lambda_9,\lambda_5,\lambda_4],[\lambda_9,\lambda_3,\lambda_2,\lambda_1], [\lambda_4,\lambda_8^{\lambda_7\lambda_6}],[\lambda_8^{\lambda_7\lambda_6},\lambda_5^{\lambda_4},\lambda_1],  \\
{[\lambda_7^{\lambda_6},\lambda_4,\lambda_1],[\lambda_6,\lambda_1],[\lambda_8^{\lambda_7\lambda_6},\lambda_2^{\lambda_1}]}, [\lambda_7^{\lambda_6},\lambda_5^{\lambda_4\lambda_1},\lambda_2^{\lambda_1}],[\lambda_6,\lambda_4^{\lambda_1},\lambda_2^{\lambda_1}],\\
{[\lambda_8^{\lambda_7\lambda_6},\lambda_3^{\lambda_2\lambda_1}]}, [\lambda_7^{\gamma_6},\gamma_3^{\gamma_2\gamma_1}], [\gamma_6,\gamma_5^{\lambda_4\lambda_2\lambda_1},\lambda_3^{\lambda_2\lambda_1}], [\lambda_4,\lambda_3]
\end{array}\rangle $$

\begin{figure}[h]
\begin{subfigure}[b]{0.4\textwidth}
       \centering
     \begin{tikzpicture}[scale=4.5]

\draw[blue] (-0.3,-.033415774658) -- (.1997695,.1332565)node[above,right]{\tiny $L_1$};

\draw[brown] (-0.3,-0.06) -- (.2,.23664)node[above,right]{\tiny $L_2$};
\draw (-.24,-.1) node[below]{\tiny $L_3$} -- (0.04,.6) ;
\draw[brown] (-.07,-.1)node[below,left]{\tiny $L_4$} -- (.2,.3);
\draw[blue] (0,-.1)node[below]{\tiny $L_5$} -- (0,0.6);

\draw[blue, thick] (.24,-.1) node[below]{\tiny $L_6$} -- (-0.04,.6) ;
\draw[red](0.3,-0.06) -- (-.2,.23664)node[above,left]{\tiny $L_7$};

\draw[brown,thick] (0.3,-.033415774658) -- (-.1997695,.1332565)node[above,left]{\tiny $L_8$};

\draw[red] (-.3,0) -- (.3,0) node[right]{\tiny $L_9$};




\end{tikzpicture}
\caption{$L(5)$ in the projective plane}
\label{fig:LinesL5}
    \end{subfigure}
~ ~ ~~~  
\quad         
\begin{subfigure}[b]{0.45\textwidth}
        \centering
        \begin{tikzpicture}[scale=.6]
\draw(4.5,0)--(6,0)node [anchor=west] {$W_1$};
\draw(4.1,1.5)-- (4.5,0);
\draw(4.1,1.5)-- (3.5,1.5);
\draw(3.2,2.5)-- (3.5,1.5);
\draw(3.2,2.5)-- (3,2.5);
\draw(2.7,3.5)-- (3,2.5);
\draw(2.7,3.5)-- (2.5,3.5);
\draw(2.3,4)-- (2.5,3.5);
\draw(-1.5,4) node[anchor=east]{$W_1$} --(2.3,4);

\draw (4.5,.5)--(6,.5)node [anchor=west]{$W_2$} ;
\draw (4.1,1)  -- (4.5,.5);
\draw (4.1,.5)  -- (4.5,1);
\draw (2,1)  -- (4.1,1);
\draw (2,1)  -- (1.8,1.5);
\draw (1.5,1.5)  -- (1.8,1.5);
\draw (1.5,1.5)  -- (1.2,2.5);
\draw (1,2.5)  -- (1.2,2.5);
\draw (1,2.5)  -- (.7,3.5);
\draw(-1.5,3.5)node[anchor=east]{$W_2$}--(.7,3.5);

\draw (4.5,1)--(6,1)node [anchor=west]{$W_3$};
\draw (4.1,.5)  -- (4.5,1);
\draw (4.1,.5)  -- (.5,.5);
\draw (.3,1)  -- (.5,.5);
\draw (.3,1)  -- (.1,1);
\draw (-.1,1.5)  -- (.1,1);
\draw (-.1,1.5)  -- (-.3,1.5);
\draw (-.6,2.5)  -- (-.3,1.5);
\draw (-.6,2.5)  -- (-.8,2.5);
\draw (-1,2.5)  -- (-.8,2.5);
\draw (-1,2.5)  -- (-1.2,3);
\draw(-1.5,3) node[anchor=east]{$W_3$} --(-1.2,3);

\draw(5,1.5)--(6,1.5)node[anchor=west]{$W_4$} ;
\draw(5,1.5)  -- (4.7,2.5);
\draw(4,2.5)  -- (4.7,2.5);
\draw(4,2.5)  -- (3.8,3);
\draw(2.5,3)  -- (3.8,3);
\draw(2.5,3)  -- (.5,3);
\draw(-1,3)  -- (.5,3);
\draw(-1,3)  -- (-1.2,2.5);
\draw (-1.5,2.5)node[anchor=east]{$W_4$}--(-1.2,2.5);

\draw(-1.5,2)node[anchor=east]{$W_5$}--(6,2)node[anchor=west]{$W_5$};


\draw(5.5,2.5)--(6,2.5)node[anchor=west]{$W_6$};
\draw(5.5,2.5)-- (5.1, 4);
\draw(2.5,4)-- (5.1, 4);
\draw(2.3,3.5)-- (2.5, 4);
\draw(2.3,3.5)-- (1, 3.5);
\draw(.7,2.5)-- (1, 3.5);
\draw(.7,2.5)-- (-.3, 2.5);
\draw(-.6,1.5)-- (-.3, 2.5);
\draw (-1.5,1.5)node[anchor=east]{$W_6$}--(-.6,1.5);

\draw(5.5,3)--(6,3)node[anchor=west]{$W_7$};
\draw (5.5,3) -- (5.1,3.5);
\draw (3,3.5) -- (5.1,3.5);
\draw (3,3.5) -- (2.7,2.5);
\draw (2.7,2.5) -- (1.5,2.5);
\draw (1.2,1.5) -- (1.5,2.5);
\draw (1.2,1.5) -- (.1,1.5);
\draw (-.1,1) -- (.1,1.5);
\draw(-1.5,1)node[anchor=east]{$W_7$}--(-.1,1);

\draw(5.5,3.5)--(6,3.5)node[anchor=west]{$W_8$};
\draw(5.5,3.5)--(5.1,3);
\draw(4,3)--(5.1,3);
\draw(4,3)--(3.8,2.5);
\draw(3.8,2.5)--(3.5,2.5);
\draw(3.2,1.5)--(3.5,2.5);
\draw(3.2,1.5)--(2,1.5);
\draw(1.8,1)--(2,1.5);
\draw(1.8,1)--(2,1.5);
\draw(1.8,1)--(.5,1);
\draw(.3,.5)--(.5,1);
\draw(-1.5,.5)node[anchor=east]{$W_8$}--(.3,.5);

\draw(5.5,4)--(6,4)node[anchor=west]{$W_9$};
\draw(5.5,4)--(5.1,2.5);
\draw(5,2.5)--(5.1,2.5);
\draw(5,2.5)--(4.7,1.5);
\draw(4.5,1.5)--(4.7,1.5);
\draw(4.5,1.5)--(4.1,0);
\draw(-1.5,0)node[anchor=east]{$W_9$}--(4.1,0);

\end{tikzpicture}
        \caption{Wiring Diagram L(5)}
        \label{fig:WiringL5}
    \end{subfigure}

    \caption{The arrangement L(5)}\label{fig:L5}
\end{figure}

The expressions for the meridians around the exceptional divisors are given as follows:

$$ \begin{array}{l}
 \lambda_{10}=\lambda_9\lambda_8\lambda_7\lambda_6, \ 
\lambda_{11}=\lambda_9\lambda_5\lambda_4, \ 
\lambda_{12}=\lambda_9\lambda_3\lambda_2\lambda_1\\
\lambda_{13}=\lambda_8^{\lambda_7\lambda_6}\lambda_5^{\lambda_4}\lambda_1, \ 
\lambda_{14}=\lambda_7^{\lambda_6}\lambda_4\lambda_1; \\

\lambda_{15}=\lambda_7^{\lambda_6}\lambda_5^{\lambda_4\lambda_1}\lambda_2^{\lambda_1};
\lambda_{16}=\lambda_6\lambda_4^{\lambda_1}\lambda_2^{\lambda_1};
\lambda_{17}=\lambda_6\lambda_5^{\lambda_4\lambda_2\lambda_1}\lambda_3^{\lambda_2\lambda_1};
\end{array}
$$ 

\begin{exmp}\label{exmp:L5.1} The arrangement $L(5)$ admits as a subarrangement the so-called deleted $B_3$-arrangement \cite[Example 10.6]{inproceedings} that is depicted in Figure \ref{fig:LinesL5} in colors. The deleted $B_3$-arrangement has a special pencil $\Pi$ that is induced by the positive dimensional component of its characteristic variety that does not passes through the origin, see \cite[Example 10.6]{inproceedings} and references there-in. We will use this pencil in order to construct homology planes with infinite fundamental group.

The deleted $B_3$ arrangement is obtained from an arrangement of nine lines called $B_3$ that admits the structure of a $(3,4)$-multinet (see \cite[Example 3.6, Figure 1.b]{falk_yuzvinsky_2007}) by deleting a line of weight two. The pencil $\Pi$ is obtained by restricting that coming from $B_3$. The line that we are removing from $B_3$ will have multiplicity two in the associated pencil.

Take $I=\{10,12,13,14,15,16\}$ and $P_1=\{D_3\cap D_4, D_5\cap D_{11}\}$. Denote by $\Pi$ the pencil described above. It will have as sections $D_3$ and $D_{11}$. The pencil $\Pi$ induces the following orbifold structure on $\mathcal{X}=\mathcal{X}(\mathbb{P}^1,D,r)$ with $D= 0 + 1+\infty$ and $r=(c,2,b)$. 

It can be seen that the determinant of the matrix $M$ constructed as in \ref{sss:algorithm} equals $-2 a c + 3 b c + 6 b d$.

By choosing the weights $a=7,b=3,c=7,d=2$ we obtain a homology plane $X$ with the following presentation for $\pi_1(X)$
$$\langle \lambda_3,\lambda_4,\lambda_5\mid	   (\lambda_3, \lambda_4),
      (\lambda_5\lambda_4)^2=(\lambda_4\lambda_5)^2,
    \lambda_3\lambda_5\lambda_3\lambda_5^{-1}\lambda_4\lambda_5\lambda_3\lambda_4\lambda_5,\lambda_3^4\lambda_4\lambda_3^3\lambda_4^2, (\lambda_5^3\lambda_4\lambda_5\lambda_4)^2\lambda_5^3 \rangle$$
Note that if $\lambda_3=1$ we obtain the presentation of $\pi_1(\mathcal{X})$ a hyperbolic triangle group with weights $(2,3,7)$.

The dual graph of the divisor at infinity is absolutely minimal, see Figure \ref{fig:DualGraphL5}. It is also not equivalent to the graphs of Lemma \ref{lem:DualGraphL1} and therefore the homology plane $X$ is of general type.

\begin{figure}[h]
       \centering
 \includegraphics[scale=1]{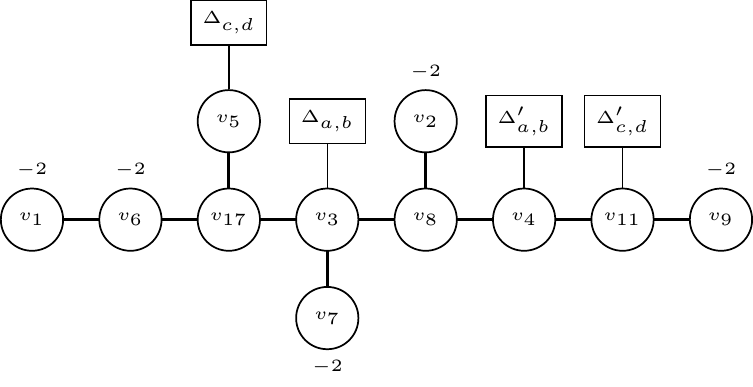}
\caption{Dual graph}
        \label{fig:DualGraphL5}
\end{figure}
\end{exmp}
\begin{rem}
By modifying the maximal tree in example \ref{exmp:L5.1} and using again the pencil induced by the deleted $B3$-arrangement we can obtain homology planes from $L(6)$ that admits a fibration with general fiber $\mathbb{P}^1\setminus\{3-\text{points}\}$ and inducing a surjective map to an infinite fundamental group of an orbicurve. This is similar to the construction in example \ref{exmp:L4.2}. The presentation of the fundamental group is complicated and we omit it.
\end{rem}
\subsubsection{The arrangement $L(6)$} The arrangement $L(6)$ is the second arrangement of nine lines with six double points, six triples and two quadruple points. It is the only arrangement in the list of \cite{Dieck1989} that is not defined over the reals, see figure \ref{fig:Lines6}.
\begin{figure}
 \centering
         \scalebox{0.90	}{

        \begin{tikzpicture}[scale=.5]

\draw(2.3,0)--(5.7,0)node [anchor=west] {\tiny $W_1$};
\draw (2.1,.5) -- (2.3,0);
\draw (1.5,.5) -- (2.1,.5);
\draw (1.3,0) -- (1.5,.5);
\draw (1.3,0) -- (-1.4,0);
\draw (-1.6,.5) -- (-1.4,0);
\draw (-1.6,.5) -- (-1.8,.5);
\draw (-2,1) -- (-1.8,.5);
\draw (-3.8, 1) -- (-2,1);
\draw (-3.8, 1) -- (-4, 1.5);
\draw (-4, 1.5) -- (-4.2,1.5);
\draw (-4.4,2) -- (-4.2,1.5);
\draw (-5,2) -- (-4.4,2);
\draw (-5,2) -- (-5.2,2.5);
\draw (-7.4,2.5) -- (-5.2,2.5);
\draw (-7.4,2.5) -- (-7.46667,2.333);
\draw (-7.6,2)--(-7.5333,2.16667);
\draw (-9.1,2)--(-7.6,2);
\draw (-9.1,2)--(-9.3,1.5);
\draw (-11,1.5)--(-9.3,1.5);
\draw (-11,1.5)--(-11.2,1);
\draw (-12.2,1)--(-11.2,1);
\draw (-12.2,1)--(-12.4,.5);
\draw (-13.5,.5)--(-12.4,.5);
\draw (-13.5,.5)--(-13.7,1);
\draw (-15,1)--(-13.7,1);
\draw (-15,1)--(-15.2,1.5);
\draw (-15.9,1.5)--(-15.2,1.5);
\draw (-15.9,1.5)--(-16.1,1);
\draw (-16.2,1)--(-16.1,1);
\draw (-16.2,1)--(-16.5,0);
\draw (-16.6,0)--(-16.5,0);
\draw (-16.6,0)--(-16.6667,.1333);
\draw (-16.8,.5)--(-16.733,.3333);
\draw (-16.9,.5)--(-16.8,.5);
\draw (-16.9,.5)--(-16.9667,.6667);
\draw (-17.1,1)--(-17.0333,.8333);
\draw (-17.2,1)--(-17.1,1);
\draw (-17.2,1)--(-17.6,2.5);
\draw (-17.7,2.5)--(-17.6,2.5);
\draw (-17.7,2.5)--(-17.7667,2.333);
\draw (-17.9,2)--(-17.8333,2.16667);
\draw (-18,2)--(-17.9,2);
\draw (-18,2)--(-18.06667,1.8333);
\draw (-18.2,1.5)--(-18.1333,1.6667);
\draw (-18.3,1.5)--(-18.2,1.5);
\draw (-18.3,1.5)--(-18.36667,1.333);
\draw (-18.5,1)--(-18.4333,1.16667);
\draw (-18.9,1)--(-18.5,1);
\draw (-18.9,1)--(-18.96667,.8333);
\draw (-19.1,.5)--(-19.0333,.66667);
\draw (-20.7,.5)--(-19.1,.5);
\draw (-20.7,.5)--(-20.76667,.3333);
\draw (-20.9,0)--(-20.8333,.16667);
\draw(-23.1,0) node[anchor=east]{\tiny $W_1$} --(-20.9,0);

\draw (4,.5)--(5.7,.5)node [anchor=west]{\tiny $W_2$} ;
\draw (3.9333,0.6667)  -- (4,.5);
\draw (3.8,1) -- (3.86667,0.8333);
\draw (3.8,1) -- (3.6,1);
\draw (3.6,1) -- (3.5333,1.16667);
\draw (3.4,1.5)-- (3.46667, 1.3333);
\draw (3.2,1.5) -- (3.4,1.5);
\draw (3.2,1.5) -- (3.1333,1.6667);
\draw (3,2) -- (3.06667,1.8333);
\draw (2.8,2) -- (3,2);
\draw (2.8,2) -- (2.7333,2.16667);
\draw (2.6,2.5) -- (2.66667,2.333);
\draw (.3,2.5) -- (2.6,2.5);
\draw (.3,2.5) -- (.23333,2.3333);
\draw (.1,2) -- (0.16667,2.16667);
\draw (.1,2) -- (-.1,2);
\draw (-.4,1) -- (-.1,2);
\draw (-.4,1) -- (-1.8,1);
\draw (-1.86667,.8333) -- (-1.8,1);
\draw (-2,.5) -- (-1.9333,.6667);
\draw (-2.2,.5) -- (-2,.5);
\draw (-2.2667,.3333) -- (-2.2,.5);
\draw (-2.4,0) -- (-2.333,.16667); 
\draw (-5.4,0) -- (-2.4,0);
\draw (-5.4,0) -- (-5.6,.5);
\draw (-5.8,.5) -- (-5.6,.5);
\draw (-5.8,.5) -- (-6,1);
\draw (-6.6,1) -- (-6,1);
\draw (-6.6,1) -- (-6.8,1.5);
\draw (-7,1.5) -- (-6.8,1.5);
\draw (-7,1.5) -- (-7.2,2);
\draw (-7.4,2) -- (-7.2,2);
\draw (-7.4,2) -- (-7.6,2.5);
\draw (-9.4,2.5) -- (-7.6,2.5);
\draw (-9.4,2.5) -- (-9.6,2);
\draw (-9.7,2) -- (-9.6,2);
\draw (-9.7,2) -- (-10,1);
\draw (-10.7,1) -- (-10,1);
\draw (-10.7,1) -- (-10.9,.5);
\draw (-11.6,.5) -- (-10.9,.5);
\draw (-11.6,.5) -- (-11.8,0);
\draw (-11.9,0) -- (-11.8,0);
\draw (-11.9,0) -- (-11.9667,0.1667);
\draw (-12.1,.5)--(-12.0333,.3333);
\draw (-12.2,.5)--(-12.1,.5);
\draw (-12.2,.5)--(-12.2667,.6667);
\draw (-12.4,1)--(-12.333,.8333);
\draw (-12.4,1)--(-12.5,1);
\draw (-12.7,1.5)--(-12.5,1);
\draw (-12.7,1.5)--(-12.8,1.5);
\draw (-13.1,2.5)--(-12.8,1.5);
\draw (-13.8,2.5)--(-13.1,2.5);
\draw (-13.8,2.5)--(-14,3);
\draw (-14,3)--(-14.4,3);
\draw (-14.6,3.5)--(-14.4,3);
\draw(-23.1,3.5)node[anchor=east]{\tiny $W_2$}--(-14.6,3.5);

\draw (4.5,1)--(5.7,1)node [anchor=west]{\tiny $W_3$};
\draw (4.2,2)--(4.5,1);
\draw (4.2,2)--(3.2,2);
\draw (3,1.5)--(3.2,2);
\draw (1.9,1.5) -- (3,1.5);
\draw (1.9,1.5) -- (1.7,1);
\draw (1.1,1) -- (1.7,1);
\draw (1.1,1) -- (.9,.5);
\draw (-1.4,.5) -- (.9,.5);
\draw (-1.46667,.333) -- (-1.4,.5);
\draw (-1.6,0) -- (-1.5333,.16667);
\draw (-2.2,0) -- (-1.6,0);
\draw (-2.2,0) -- (-2.4,.5);
\draw (-5.4,.5) -- (-2.4,.5);
\draw (-5.46667,.333)--(-5.4,.5);
\draw (-5.6,0)--(-5.5333,.16667);
\draw (-7.8,0)--(-5.6,0);
\draw (-7.8,0)--(-8.1,1);
\draw (-8.2,1)--(-8.1,1);
\draw (-8.2,1)--(-8.2667,.8333);
\draw (-8.4,.5)--(-8.3333,.66667);
\draw (-8.4,.5)--(-8.5,.5);
\draw (-8.5667,.3333)--(-8.5,.5);
\draw (-8.7,0)--(-8.6333,.16667);
\draw (-11.6,0)--(-8.7,0);
\draw (-11.6,0)--(-11.8,0.5);
\draw (-11.9,.5)--(-11.8,0.5);
\draw (-11.9,.5)--(-12.1,0);
\draw (-16.2,0)--(-12.1,0);
\draw (-16.2,0)--(-16.5,1);
\draw (-16.9,1)--(-16.5,1);
\draw (-16.9,1)--(-17.1,.5);
\draw (-18.9,.5)--(-17.1,.5);
\draw (-18.9,.5)--(-19.1,1);
\draw (-19.2,1)--(-19.1,1);
\draw (-19.2,1)--(-19.4,1.5);
\draw (-20.1,1.5)--(-19.4,1.5);
\draw (-20.1,1.5)--(-20.3,2);
\draw (-21.6,2)--(-20.3,2);
\draw (-21.6,2)--(-21.6667,1.8333);
\draw (-21.8,1.5)--(-21.7333,1.6667);
\draw (-22.5,1.5)--(-21.8,1.5);
\draw (-22.5,1.5)--(-22.7,2);
\draw (-22.8,2)--(-22.7,2);
\draw (-22.8,2)--(-23,1.5);
\draw(-23.1,1.5) node[anchor=east]{\tiny $W_3$} --(-23,1.5);

\draw(3.6,1.5)--(5.7,1.5)node[anchor=west]{\tiny $W_4$} ;
\draw(3.6,1.5)  -- (3.4,1);
\draw(1.9,1)  -- (3.4,1);
\draw(1.9,1)  -- (1.8333,1.16667);
\draw(1.7,1.5) -- (1.76667,1.3334);
\draw(.7,1.5) -- (1.7,1.5);
\draw(.7,1.5) -- (.5,2);
\draw(.3,2) -- (.5,2);
\draw(.1,2.5)--(.3,2);
\draw(.1,2.5)--(-.6,2.5);
\draw(-.8,2)--(-.6,2.5);
\draw(-.8,2)--(-4.2,2);
\draw(-4.26667,1.8333)--(-4.2,2);
\draw(-4.4,1.5)--(-4.3333,1.6667);
\draw(-6.2,1.5)--(-4.4,1.5);
\draw(-6.2,1.5)--(-6.4,2);
\draw(-7,2)--(-6.4,2);
\draw(-7,2)--(-7.06667,1.8333);
\draw(-7.2,1.5)--(-7.1333,1.66667);
\draw(-9.1,1.5)--(-7.2,1.5);
\draw(-9.1,1.5)--(-9.16667,1.6667);
\draw(-9.3,2)--(-9.2333,1.8333);
\draw(-9.3,2)--(-9.4,2);
\draw(-9.4667,2.16667)--(-9.4,2);
\draw(-9.6,2.5)--(-9.5333,2.333);
\draw(-9.6,2.5)--(-10.1,2.5);
\draw(-10.1667,2.333)--(-10.1,2.5);
\draw(-10.3,2)--(-10.2333,2.16667);
\draw(-17.2,2)--(-10.3,2);
\draw(-17.2,2)--(-17.6,1.5);
\draw(-18,1.5)--(-17.6,1.5);
\draw(-18,1.5)--(-18.2,2);
\draw(-18.6,2)--(-18.2,2);
\draw(-18.6,2)--(-18.6667,1.8333);
\draw(-18.8,1.5)--(-18.7333,1.6667);
\draw(-19.2,1.5)--(-18.8,1.5);
\draw(-19.2,1.5)--(-19.26667,1.333);
\draw(-19.4,1)--(-19.3333,1.16667);
\draw(-20.4,1)--(-19.4,1);
\draw(-20.4,1)--(-20.6,1.5);
\draw(-21.3,1.5)--(-20.6,1.5);
\draw(-21.3,1.5)--(-21.36667,1.333);
\draw(-21.5,1)--(-21.4333,1.16667);
\draw (-23.1,1)node[anchor=east]{\tiny $W_4$}--(-21.5,1);

\draw(4.9,2)--(5.7,2)node[anchor=west]{\tiny $W_5$};
\draw(4.8333,2.16667)--(4.9,2);
\draw(4.7,2.5)--(4.766667,2.33334);
\draw (4.7, 2.5) -- (2.8,2.5);
\draw (2.6, 2) -- (2.8,2.5);
\draw (.7,2)--(2.6,2);
\draw (.7,2)--(.63333,1.8333);
\draw (.5,1.5) --(.56667, 1.6667);
\draw (-3.8,1.5) -- (.5,1.5);
\draw (-3.8,1.5) -- (-3.86667,1.3333);
\draw (-3.9333,1.16667) -- (-4,1);
\draw (-5.8,1) -- (-4,1);
\draw (-5.8,1) -- (-5.86667,.8333);
\draw (-6,.5) -- (-5.9333,.6333);
\draw (-8.2,.5)--(-6,.5);
\draw (-8.2,.5)--(-8.4,1);
\draw (-8.8,1)--(-8.4,1);
\draw (-8.8,1)--(-8.8667,.8333);
\draw (-9,.5)--(-8.9333,.6667);
\draw (-10.7,.5)--(-9,.5);
\draw (-10.7,.5)--(-10.7667,.6667);
\draw (-10.9,1)--(-10.8333,.8333);
\draw (-10.9,1)--(-11,1);
\draw (-11,1)--(-11.06667,1.16667);
\draw (-11.2,1.5)--(-11.1333,1.333);
\draw (-11.2,1.5)--(-12.5,1.5);
\draw (-12.5667,1.333)--(-12.5,1.5);
\draw (-12.7,1)--(-12.6333,1.1667);
\draw (-13.2,1)--(-12.7,1);
\draw (-13.2,1)--(-13.2667,1.1667);
\draw (-13.4,1.5)--(-13.333,1.333);
\draw (-15,1.5)--(-13.4,1.5);
\draw (-15,1.5)--(-15.0667,1.333);
\draw (-15.2,1)--(-15.1333,1.16667);
\draw (-15.2,1)--(-15.3,1);
\draw (-15.5,.5)--(-15.3,1);
\draw (-15.5,.5)--(-15.6,.5);
\draw (-15.8,1)--(-15.6,.5);
\draw (-15.8,1)--(-15.9,1);
\draw (-15.9667,1.1667)--(-15.9,1);
\draw (-16.1,1.5)--(-16.0333,1.333);
\draw (-17.2,1.5)--(-16.1,1.5);
\draw (-17.2,1.5)--(-17.6,2);
\draw (-17.7,2)--(-17.6,2);
\draw (-17.7,2)--(-17.9,2.5);
\draw (-19.5,2.5)--(-17.9,2.5);
\draw (-19.5,2.5)--(-19.7,3);
\draw (-22.2,3)--(-19.7,3);
\draw (-22.2,3)--(-22.26667,2.8333);
\draw (-22.4,2.5)--(-22.3333,2.6667);
\draw(-23.1,2.5) node[anchor=east]{\tiny $W_5$} --(-22.4,2.5);

\draw(5.5,2.5)--(5.7,2.5)node[anchor=west]{\tiny $W_6$};
\draw(5.5,2.5)-- (5.1, 4);
\draw(-2.6,4)-- (5.1, 4);
\draw(-2.6,4)-- (-2.8, 3.5);
\draw(-3,3.5)-- (-2.8, 3.5);
\draw(-3,3.5)-- (-3.2, 3);
\draw(-4.6,3)-- (-3.2, 3);
\draw(-4.8,2.5)-- (-4.6, 3);
\draw(-5,2.5)--(-4.8,2.5);
\draw(-5,2.5)--(-5.06667,2.333);
\draw(-5.2,2)--(-5.1333,2.16667);
\draw(-6.2,2)--(-5.2,2);
\draw(-6.2,2)--(-6.4,1.5);
\draw(-6.6,1.5)--(-6.4,1.5);
\draw(-6.6,1.5)--(-6.6667,1.333);
\draw(-6.8,1)--(-6.7333,1.16667);
\draw(-7.8,1)--(-6.8,1);
\draw(-7.8,1)--(-8.1,0);
\draw(-8.5,0)--(-8.1,0);
\draw(-8.7,.5)--(-8.5,0);
\draw(-8.7,.5)--(-8.8,.5);
\draw(-9,1)--(-8.8,.5);
\draw(-9,1)--(-9.7,1);
\draw(-10,2)--(-9.7,1);
\draw(-10,2)--(-10.1,2);
\draw(-10.3,2.5)--(-10.1,2);
\draw(-10.3,2.5)--(-10.4,2.5);
\draw(-10.6,3)--(-10.4,2.5);
\draw(-11.3,3)--(-10.6,3);
\draw(-11.3,3)--(-11.5,3.5);
\draw(-14.1,3.5)--(-11.5,3.5);
\draw(-14.1,3.5)--(-14.3,4);
\draw(-14.4,4)--(-14.3,4);
\draw (-23.1,4)node[anchor=east]{\tiny $W_6$}--(-14.4,4);

\draw(5.5,3)--(5.7,3)node[anchor=west]{\tiny $W_7$};
\draw (5.5,3) -- (5.1,3.5);
\draw (-2.6,3.5) -- (5.1,3.5);
\draw (-2.6,3.5) -- (-2.66667,3.6667);
\draw (-2.8,4) -- (-2.7333,3.8333);
\draw (-14.1,4) -- (-2.8,4);
\draw (-14.1,4) -- (-14.1667,3.833);
\draw (-14.3,3.5)--(-14.233,3.667);
\draw (-14.3,3.5)--(-14.4,3.5);
\draw (-14.6,3)--(-14.4,3.5);
\draw (-14.6,3)--(-14.7,3);
\draw (-14.9,2.5)--(-14.7,3);
\draw (-14.9,2.5)--(-17.2,2.5);
\draw (-17.6,1)--(-17.2,2.5);
\draw (-17.6,1)--(-18.3,1);
\draw (-18.5,1.5)--(-18.3,1);
\draw (-18.6,1.5)--(-18.5,1.5);
\draw (-18.6,1.5)--(-18.8,2);
\draw (-19.8,2)--(-18.8,2);
\draw (-19.8,2)--(-20,2.5);
\draw (-20,2.5)--(-21.9,2.5);
\draw (-21.9667,2.333)--(-21.9,2.5);
\draw (-22.1,2)--(-22.0333,2.16667);
\draw (-22.5,2)--(-22.1,2);
\draw (-22.5,2)--(-22.5667,1.8333);
\draw (-22.7,1.5)--(-22.6333,1.6667);
\draw (-22.8,1.5)--(-22.7,1.5);
\draw (-22.8,1.5)--(-23,2);
\draw(-23.1,2)node[anchor=east]{\tiny $W_7$}--(-23,2);

\draw(5.5,3.5)--(5.7,3.5)node[anchor=west]{\tiny $W_8$};
\draw(5.5,3.5)--(5.1,3);
\draw(-1,3)--(5.1,3);
\draw(-1,3)--(-1.2,2.5);
\draw(-4.6,2.5) -- (-1.2,2.5);
\draw(-4.6,2.5) -- (-4.6667,2.6667);
\draw(-4.8,3) -- (-4.7333,2.8333);
\draw(-10.4,3)--(-4.8,3);
\draw(-10.4,3)--(-10.4667,2.8333);
\draw(-10.6,2.5)--(-10.533,2.6667);
\draw(-12.8,2.5)--(-10.6,2.5);
\draw(-12.8,2.5)--(-13.1,1.5);
\draw(-13.2,1.5)--(-13.1,1.5);
\draw(-13.2,1.5)--(-13.4,1);
\draw(-13.5,1)--(-13.4,1);
\draw(-13.5,1)--(-13.5667,.8333);
\draw(-13.7,.5)--(-13.6333,.6667);
\draw(-13.7,.5)--(-15.3,.5);
\draw(-15.5,1)--(-15.3,.5);
\draw(-15.5,1)--(-15.6,1);
\draw(-15.667,.8333)--(-15.6,1);
\draw(-15.8,.5)--(-15.7333,.6667);
\draw(-16.6,.5)--(-15.8,.5);
\draw(-16.6,.5)--(-16.8,0);
\draw(-20.7,0)--(-16.8,0);
\draw(-20.7,0)--(-20.9,.5);
\draw(-21,.5)--(-20.9,.5);
\draw(-21,.5)--(-21.2,1);
\draw(-21.3,1)--(-21.2,1);
\draw(-21.3,1)--(-21.5,1.5);
\draw(-21.6,1.5)--(-21.5,1.5);
\draw(-21.6,1.5)--(-21.8,2);
\draw(-21.9,2)--(-21.8,2);
\draw(-21.9,2)--(-22.1,2.5);
\draw(-22.2,2.5)--(-22.1,2.5);
\draw(-22.2,2.5)--(-22.4,3);
\draw(-23.1,3)node[anchor=east]{\tiny $W_8$}--(-22.4,3);

\draw(5.5,4)--(5.7,4)node[anchor=west]{\tiny $W_9$};
\draw(5.5,4)--(5.1,2.5);
\draw(4.9,2.5)--(5.1,2.5);
\draw(4.9,2.5)--(4.7,2);
\draw(4.5,2)--(4.7,2);
\draw(4.5,2)--(4.2,1);
\draw(4,1)--(4.2,1);
\draw(4,1)--(3.8,.5);
\draw(2.3,.5) -- (3.8,.5);
\draw(2.3,.5) -- (2.1,0);
\draw(1.5,0) -- (2.1,0);
\draw(1.5,0) -- (1.43334,.16667);
\draw (1.3,.5) -- (1.3667,.3333);
\draw (1.1,.5) -- (1.3,.5);
\draw (1.03334,.6667) -- (1.1,.5);
\draw(.9,1) -- (.96667,.8333); 
\draw(-.1,1) -- (.9,1);
\draw(-.1,1) -- (-.4,2);
\draw(-.6,2) -- (-.4,2);
\draw(-.66667,2.16667) -- (-.6,2);
\draw (-.8,2.5) -- (-.7333,2.333);
\draw (-1,2.5) -- (-.8,2.5);
\draw (-1.0667,2.6667) -- (-1,2.5);
\draw (-1.2,3) -- (-1.1333,2.8333);
\draw (-3,3) -- (-1.2,3);
\draw (-3,3) -- (-3.06667,3.16667);
\draw (-3.2,3.5)--(-3.1333,3.3333);
\draw (-11.3,3.5)--(-3.2,3.5);
\draw (-11.3,3.5)--(-11.36667,3.333);
\draw (-11.5,3)--(-11.4333,3.16667);
\draw (-13.8,3)--(-11.5,3);
\draw (-13.8,3)--(-13.8667,2.833);
\draw (-14,2.5)--(-13.933,2.6333);
\draw (-14,2.5)--(-14.7,2.5);
\draw (-14.7667,2.6667)--(-14.7,2.5);
\draw (-14.9,3)--(-14.8333,2.8333);
\draw (-14.9,3)--(-19.5,3);
\draw (-19.56667,2.8333)--(-19.5,3);
\draw (-19.7,2.5)--(-19.6333,2.6667);
\draw (-19.8,2.5)--(-19.7,2.5);
\draw (-19.8,2.5)--(-19.86667,2.3333);
\draw (-20,2)--(-19.9333,2.16667);
\draw (-20.1,2)--(-20,2);
\draw (-20.1,2)--(-20.16667,1.8333);
\draw (-20.3,1.5)--(-20.2333,1.66667);
\draw (-20.4,1.5)--(-20.3,1.5);
\draw (-20.4,1.5)--(-20.46667,1.333);
\draw (-20.6,1)--(-20.5333,1.1667);
\draw (-21,1)--(-20.6,1);
\draw (-21,1)--(-21.0667,.8333);
\draw (-21.2,.5)--(-21.1333,.6667);
\draw(-23.1,.5)node[anchor=east]{\tiny $W_9$}--(-21.2,0.5);

\end{tikzpicture}}
        \caption{Wiring Diagram L(6)}
        \label{fig:WiringL6}
\end{figure}
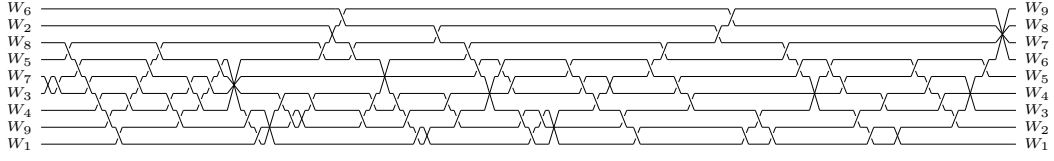
\begin{figure}[h]
\centering
\begin{tikzpicture}[scale=.7]
\draw(-2,3)node[anchor=east]{\tiny $L_6$}--(2,3); 
\draw(2,3)--(3,1);
\draw(-2,2)node[anchor=east]{\tiny $L_9$}--(3,2); 
\draw(-2,1)node[anchor=east]{\tiny $L_8$}--(2,1); 
\draw(2,1)--(3,3);

\draw(-1,-0)node[below]{\tiny $L_4$}--(-1,4);
\draw(-1,4)--(1,5);
\draw(-0,-0)node[below]{\tiny $L_5$}--(-0,5);
\draw(1,-0)node[below]{\tiny $L_1$}--(1,4);
\draw(1,4)--(-1,5);
\draw(-2,0)--(3,5)node[anchor=west]{\tiny $L_2$};
\draw(-2,.5)--(-1.5,0);
\draw[dotted](-2,.5)--(-1.5,1.5);
\draw(-1.5,1.5)--(.5,3.5);
\draw[dotted](-1.5,0)--(.75,.75);
\draw(.75,.75)--(1.25,1.25);
\draw[dotted] (.5,3.5)--(1.25,4.25);
\draw (1.25,4.25)--(2,5)node[anchor=west]{\tiny $L_3$};

\draw(-2,4.5)node[left]{\tiny $L_7$}--(1.5,4.5);
\draw  (1.5,4.5) to[out=0,in=90] 
(2.5,3.5) ;
\draw(2.5,0)--(2.5,3.5);

\end{tikzpicture}
  \caption{$L(6)$ in the projective plane}
        \label{fig:Lines6}
\end{figure}
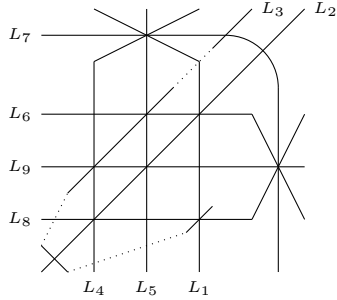

A presentation for $\pi_1(\mathbb{P}^2\setminus L(6))$ can be obtained from the wiring diagram from Figure \ref{fig:WiringL6}.

Using again this diagram we can obtain the expressions for the meridians around the exceptional divisors $D_{10},D_{11},\ldots, D_{17}$. We present here the expression of the first four, the rest can be more easily computed by using a computer (Magma):
$$
\begin{array}{lll}
\lambda_{10}:=\lambda_9\lambda_8\lambda_7\lambda_6 & \lambda_{11}:=\lambda_9^{\lambda_5}\lambda_4\lambda_3 & \lambda_{12}:= \lambda_2\lambda_5^{\lambda_2}\lambda_9^{\lambda_5\lambda_2} \\
\lambda_{13}:= \lambda_6^{\lambda_8\lambda_7\lambda_6\lambda_9^{\lambda_5\lambda_2}}\lambda_5^{\lambda_9\lambda_3\lambda_2	}\lambda_3^{\lambda_4\lambda_3\lambda_9^{\lambda_5}} & \ldots  &  \\

\end{array}
$$

\begin{exmp} Several homology planes arise from this arrangement as follows: let $I=\{10,12,13,14,16,17\}$ and $M(L(6),I)=\bar{X}\setminus D'$. It is not difficult to see that the dual graph of $D'$ has two independent cycles. We can cut these cycles by expanding the edges corresponding to $P_1=\{D_2\cap D_7, D_4\cap D_{11}\}$. In homology, this expansion can be written as $a[\lambda_2]+b[\lambda_7]=0$ and $c[\lambda_4]+d[\lambda_{11}]=0$. By letting the matrix $M$ be as in \ref{sss:algorithm}, we obtain that $\det M=-2 a c + 3 a d + 9 b d$. Thus, several homology planes can arise from this configuration.

Using Magma, the presentation of $\pi_1(\mathbb{P}^2\setminus L(6))$ and the expressions for the meridians around exceptional divisor obtained by the wiring diagram of $L(6)$, we can obtain the following presentation for $\pi_1(M(L(6),I))$.
$$\pi_1(M(L(6),I))\cong \langle \lambda_5,\lambda_6 \mid \lambda_6\lambda_5^{-1}\lambda_6=\lambda_5^{-1}(\lambda_6^2)^{\lambda_1}, \lambda_6\lambda_5^{-1}\lambda_6\lambda_5=(\lambda_6^2)^{\lambda_5^{-1}} \rangle $$

For several values of $a,b,c,d$ such that $\det M=1$ (for example $a=2,b=3,c=8,d=1$), we have checked using Magma that the associated homology planes are in fact contractible.

The dual graph of the divisor associated to the homology plane obtained with $a=2,b=3,c=8,d=1$ is presented in Figure \ref{fig:DualGraphL6}. Note that this graph is absolutely minimal and not equivalent to those of Lemma \ref{lem:DualGraphL1}, neither to those obtained by Zaidenberg \cite{AST_1993__217__251_0}.
\begin{figure}[h]
       \centering
\includegraphics[scale=1]{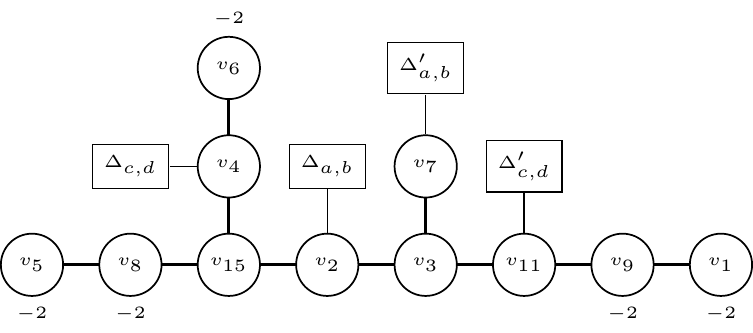}
\caption{Dual graph}
        \label{fig:DualGraphL6}
\end{figure}

We have obtained no homology plane with infinite fundamental group from this arrangement.
\end{exmp}

\subsubsection{The arrangement $L(7)$} The arrangement $L(7)$ consists of ten lines with $8$ double points, $7$ triple points, one quadruple and one quintuple point. See figure \ref{fig:LinesL7}.

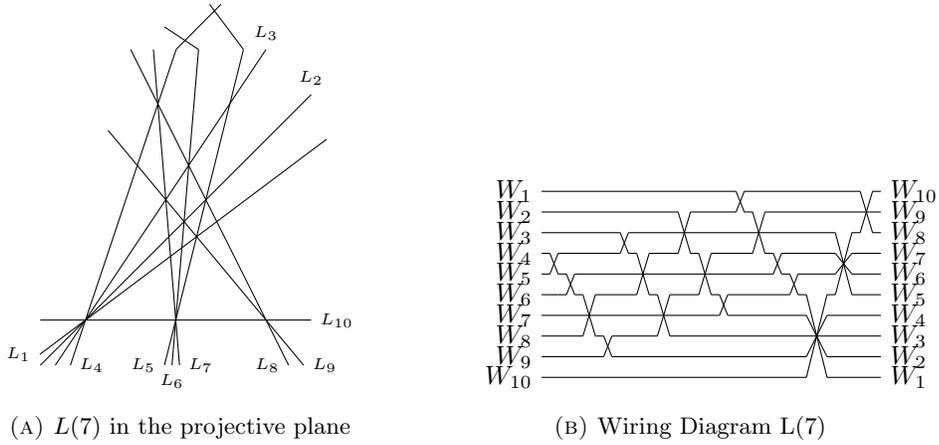
\begin{figure}[h]

\begin{subfigure}[b]{0.4\textwidth}
       \centering
     \begin{tikzpicture}[scale=3]
\draw (-0.2024947,-.15209662752)node[below,left]{\tiny $L_1$} -- (1.0683059970951,.8020136165);
\draw (-0.2,-.2) -- (1,1)node[above]{\tiny $L_2$};
\draw (-0.13371657,-.2) -- (.8,1.2)node[above]{\tiny $L_3$};
\draw (-0.06692637,-.2)node[below,right]{\tiny $L_4$} -- (.40156,1.2);
\draw (.6,1.4) -- (.40156,1.2);

\draw (0.35,-.2)node[below,left]{\tiny $L_5$} -- (0.7,1.2);
\draw (0.55,1.4) -- (0.7,1.2);

\draw (0.383181,-.2)node[below]{\tiny $L_6$} -- (0.5009132773537,1.2);
\draw (0.35,1.3) -- (0.5009132773537,1.2);

\draw (0.416539977841,-.2)node[below,right]{\tiny $L_7$} -- (0.300769,1.2);
\draw (0.9,-.2)node[below,left]{\tiny $L_8$} -- (0.2,1.2);
\draw (0.9663622687779,-.2)node[below,right]{\tiny $L_9$} -- (0.1,0.8415369724666);
\draw (-0.2,0) -- (1,0)node[right]{\tiny $L_{10}$};




\end{tikzpicture}
\caption{$L(7)$ in the projective plane}
\label{fig:LinesL7}
    \end{subfigure}
~ ~ ~~~  
\quad         
\begin{subfigure}[b]{0.45\textwidth}
        \centering
        \begin{tikzpicture}[scale=.55]
\draw(4,0)--(5.3,0)node [anchor=west] {$W_1$};
\draw(3.5,2)--(4,0);
\draw(3.5,2)--(3.3,2);
\draw(3.1,2.5)--(3.3,2);
\draw(3.1,2.5)--(2.9,2.5);
\draw(2.7,3)--(2.9,2.5);
\draw(2.7,3)--(2.5,3);
\draw(2.2,4)--(2.5,3);
\draw(2.2,4)--(2,4);
\draw(1.8,4.5)--(2,4);
\draw (-2.9,4.5) node [anchor=east] {$W_1$}  -- (1.8,4.5);

\draw (4,.5)--(5.3,.5)node [anchor=west]{$W_2$} ;
\draw(3.5,1.5)--(4,.5);
\draw(3.5,1.5)--(1.6,1.5);
\draw(1.4,2)--(1.6,1.5);
\draw(1.4,2)--(1.2,2);
\draw(.9,3)--(1.2,2);
\draw(.9,3)--(.7,3);
\draw(.4,4)--(.7,3);
\draw(.4,4)--(-2.9,4)node [anchor=east] {$W_2$} ;

\draw (0.2,1)--(5.3,1)node [anchor=west]{$W_3$};
\draw (0.2,1)--(-0.1,2);
\draw (-0.3,2)--(-0.1,2);
\draw (-0.3,2)--(-0.6,3);
\draw (-0.8,3)--(-0.6,3);
\draw (-0.8,3)--(-1,3.5);
\draw(-2.9,3.5)node [anchor=east] {$W_3$}--(-1,3.5);

\draw(4,1.5)--(5.3,1.5)node[anchor=west]{$W_4$} ;
\draw(3.5,.5)--(4,1.5);
\draw(3.5,.5)--(-1.2,.5);
\draw(-1.4,1)--(-1.2,.5);
\draw(-1.4,1)--(-1.6,1);
\draw(-1.9,2)--(-1.6,1);
\draw(-1.9,2)--(-2.1,2);
\draw(-2.3,2.5)--(-2.1,2);
\draw(-2.3,2.5)--(-2.5,2.5);
\draw(-2.7,3)--(-2.5,2.5);
\draw(-2.7,3)--(-2.9,3)node [anchor=east] {$W_4$};

\draw(4.6,2)--(5.3,2)node[anchor=west]{$W_5$};
\draw(4.6,2)--(4.2,3.5);
\draw(-.8,3.5)--(4.2,3.5);
\draw(-.8,3.5)--(-1,3);
\draw(-2.5,3)--(-1,3);
\draw(-2.5,3)--(-2.7,2.5);
\draw(-2.7,2.5)--(-2.9,2.5)node[anchor=east]{$W_5$};

\draw(4.6,2.5)--(5.3,2.5)node[anchor=west]{$W_6$};
\draw(4.6,2.5)--(4.2,3);
\draw(2.9,3)--(4.2,3);
\draw(2.7,2.5)--(2.9,3);
\draw(2.7,2.5)--(-2.1,2.5);
\draw(-2.3,2)--(-2.1,2.5);
\draw(-2.3,2)--(-2.5,2);
\draw(-2.5,2)--(-2.9,2)node[anchor=east]{$W_6$};

\draw(5.3,3)node[anchor=west]{$W_7$}--(4.6,3);
\draw(4.6,3)--(4.2,2.5);
\draw(3.3,2.5)--(4.2,2.5);
\draw(3.3,2.5)--(3.1,2);
\draw(1.6,2)--(3.1,2);
\draw(1.6,2)--(1.4,1.5);
\draw(-2.9,1.5)node[anchor=east]{$W_7$}--(1.4,1.5);

\draw(5.1,3.5)--(5.3,3.5)node[anchor=west]{$W_8$};
\draw(5.1,3.5)--(4.8,4.5);
\draw(2,4.5)--(4.8,4.5);
\draw(2,4.5)--(1.8,4);
\draw(.7,4)--(1.8,4);
\draw(.7,4)--(.4,3);
\draw(-.3,3)--(.4,3);
\draw(-.3,3)--(-.6,2);
\draw(-1.6,2)--(-.6,2);
\draw(-1.6,2)--(-1.9,1);
\draw(-1.9,1)--(-2.9,1)node[anchor=east]{$W_8$};

\draw(2.5,4)--(5.3,4)node[anchor=west]{$W_9$};
\draw(2.2,3)--(2.5,4);
\draw(1.2,3)--(2.2,3);
\draw(1.2,3)--(.9,2);
\draw(.2,2)--(.9,2);
\draw(.2,2)--(-.1,1);
\draw(-1.2,1)--(-.1,1);
\draw(-1.2,1)--(-1.4,.5);
\draw(-1.4,.5)--(-2.9,.5) node[anchor=east]{$W_9$};

\draw(5.1,4.5)--(5.3,4.5)node[anchor=west]{$W_{10}$};
\draw(5.1,4.5)--(4.8,3.5);
\draw(4.6,3.5)--(4.8,3.5);
\draw(4.6,3.5)--(4.2,2);
\draw(4.2,2)--(4,2);
\draw(3.5,0)--(4,2);
\draw(3.5,0)--(-2.9,0)node[anchor=east]{$W_{10}$};

\end{tikzpicture}
        \caption{Wiring Diagram L(7)}
        \label{fig:WiringL7}
    \end{subfigure}

    \caption{The arrangement L(7)}\label{fig:L7}
\end{figure}

The fundamental group $\pi_1(\mathbb{P}^2\setminus L(7))$ admits the following presentation (omitting redundant relations) obtained using the wiring diagram \ref{fig:WiringL7}:
$$\left\langle \lambda_1,\ldots, \lambda_9 \mid \begin{array}{l}
{[\lambda_7^{\lambda_6\lambda_5},\lambda_1]}, [\lambda_6^{\lambda_5},\lambda_1], [\lambda_9^{\lambda_8},\lambda_5,\lambda_1],[\lambda_8,\lambda_1], \\
{[\lambda_7^{\lambda_6\lambda_5},\lambda_2^{\lambda_1}]},{[\lambda_9^{\lambda_8},\lambda_6^{\lambda_5},\lambda_2^{\lambda_1}]}, [\lambda_8, \lambda_5^{\lambda_1},\lambda_2^{\lambda_1}] \\
{[\lambda_9^{\lambda_8},\lambda_7^{\lambda_6\lambda_5},\lambda_3^{\lambda_2\lambda_1}]}, [\lambda_8, \lambda_6^{\lambda_5 \lambda_2^{\lambda_1}} ,\lambda_3^{\lambda_2\lambda_1}],  [\lambda_5, \lambda_3], \\
{[\lambda_9^{\lambda_8},\lambda_4^{\lambda_3\lambda_2\lambda_1}]}, [\lambda_8,\lambda_7^{\lambda_6\lambda_5\lambda_3^{\lambda_2\lambda_1}},\lambda_4^{\lambda_3\lambda_2\lambda_1}],\\ {[\lambda_6^{\lambda_5\lambda_1^{-1}},\lambda_4]}, [\lambda_5,\lambda_4]

\end{array} \right \rangle $$
For the meridians we have:

$$\begin{array}{lll}
\lambda_{11}=\lambda_{10}\lambda_9\lambda_8, & \lambda_{12}=\lambda_{10}\lambda_7\lambda_6\lambda_5, & \lambda_{13}=\lambda_{10}\lambda_4\lambda_3\lambda_2\lambda_1, \\ \lambda_{14}=\lambda_9^{\lambda_8}\lambda_5\lambda_1, &
\lambda_{15}=\lambda_9^{\lambda_8}\lambda_6^{\lambda_5}\lambda_2^{\lambda_1} & \lambda_{16}=\lambda_8\lambda_5^{\lambda_1}\lambda_2^{\lambda_1}  \\ \lambda_{17}=\lambda_9^{\lambda_8}\lambda_7^{\lambda_6\lambda_5}\lambda_3^{\lambda_2\lambda_1} & \lambda_{18}=\lambda_8\lambda_6^{\lambda_5\lambda_2^{\lambda_1}}\lambda_3^{\lambda_2\lambda_1} &
\lambda_{19}=\lambda_8\lambda_7^{\lambda_6\lambda_5\lambda_3^{\lambda_2\lambda_1}}\lambda_4^{\lambda_3\lambda_2\lambda_1} 
\end{array} 
$$


Let $P_0\subset \Sing L(7)$ be the points of multiplicity equal or higher than three. Consider $\pi^{(0)}:\bar{X}=\Bl_{P_0}\mathbb{P}^2\to \mathbb{P}^2$ and denote by $D$ the reduced total transform of $L(7)$ by $\pi^{(0)}$. Let us denote by $\Delta$ the dual graph of $D$.
\begin{lem} There exists a unique subgraph $\Delta '\subset \Delta$ obtained by removing vertices corresponding to exceptional divisors of $D$ and their adjacent edges, such that we can obtain from $\Delta'$ maximal trees $\mathcal{T}$ corresponding to homology planes associated to $L(7)$ only by expanding some vertices of $\Delta'$.
\end{lem}
\begin{proof}
Let us first construct $\Delta'$. Remove from $\Delta$ all the vertices corresponding to exceptional divisors in $D$ but $v_{11}$ corresponding to $D_{11}$ (note that otherwise the vertex corresponding to $D_{10}$ would be disconnected). We denote this new subgraph by $\Delta'$.

Note that $\Delta'$ has only one cycle. It has eleven edges: eight corresponding to the double points of $L(7)$ and three coming from $D_{11}$, and eleven vertices: ten coming from the lines in $L(7)$ plus $D_{11}$. We have then that $H_1(\Delta)=e-v+1=1$. We will see later that we can actually have homology planes from this dual graph.

Now, if we connect $v_{10}$ either with $D_{12}$ or $D_{13}$ instead of $D_{11}$ note that $e-v+1>1$, so no homology planes can rise from this graph.

Now, note that if we plug another vertex $v$ corresponding to a triple point with edges $e_1,e_2,e_3$, we have that 
$3=2+H_1(\Delta')=H_1(\Delta\cup\{v,e_1,e_2,e_3\})$ and by \ref{sss:algorithm}, no homology planes can arise from this graph only by expanding vertices of $\Delta\cup\{v,e_1,e_2,e_3\}$.
\end{proof}

\begin{exmp}

If we let $I=\{12, 13,14,15,16,17,18,19\}$ we have that 
\begin{equation}\label{eq:presEx.7}
\pi_1(M(L(7),I))=\left\langle \lambda_3,\lambda_8\mid \lambda_8\lambda_3\lambda_8^{-1}\lambda_3\lambda_8\lambda_3^{-1}, \lambda_8\lambda_3\lambda_8^{-1}\lambda_3^{-1}\lambda_8^{-1}\lambda_3^{-1}\lambda_8\lambda_3 \right\rangle  \end{equation}

Denote by $\varphi:\pi_1(\mathbb{P}^2\setminus L(7))\to \pi_1(M(L(7)),I)$ the quotient map. We have that 
\begin{equation}\label{eq:ExpressionExL7}
 \varphi(\lambda_9)=\lambda_3^2\lambda_8\lambda_3\lambda_8\lambda_3^2\lambda_8\lambda_3^{-2}\lambda_8^{-1}, \quad \varphi(\lambda_{11})=(\lambda_3^2\lambda_8\lambda_3\lambda_8\lambda_3^2\lambda_8\lambda_3^{-2})^2 \end{equation}


Let us compute first the determinant of the matrix $M(i,j	)$ having as the first eight rows the coefficients $a_{k,r}$ in $[\lambda_{10+k}]=\sum_{r=1}^9a_{10+k,r}$ for $k=2,3,4,5,6,7,8,9$, and for the ninth row the expression for expanding an edge $(i,j)$ of $\Delta'$: 
$$\begin{array}{ll}
e_{1,8}: a[\lambda_1]+b[\lambda_8] & \det M(1,8)=-3a-b\\
e_{1,6}:a[\lambda_1]+b[\lambda_6] & \det M(1,6)=-3a \\
e_{4,6}:a[\lambda_4]+b[\lambda_6] & \det M(4,6)=3a \\
e_{4,9}:a[\lambda_4]+b[\lambda_9] & \det M(4,9)=3a+b \\
e_{8,11}:a[\lambda_8]-b[\sum_{r=1}^7 \lambda_r] & \det M(8,11)=-a \\
e_{9,11}:a[\lambda_9]-b[\sum_{r=1}^7 \lambda_r] & \det M(9,11)=a
\end{array} $$
Therefore, only from the last two rows homology planes can arise. By using the expressions in (\ref{eq:presEx.7}), (\ref{eq:ExpressionExL7}) and Magma we can show that for $a=1$ and low values of $b$ ($b\leq 1000$) we have that $\pi_1(X)$ is trivial.

\begin{figure}[h]
       \centering
\includegraphics[scale=1]{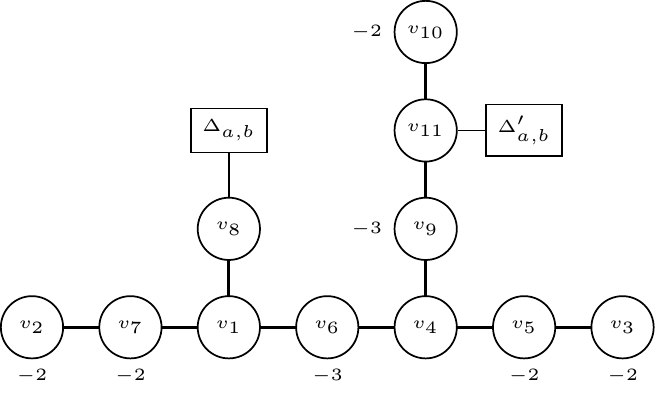}
\caption{Dual graph}
        \label{fig:DualGraphL7}
\end{figure}

The graph obtained from $\Delta'$ by expanding the edge $e_{8,11}$ is shown in Figure \ref{fig:DualGraphL7}. Note that it is absolutely minimal and not equivalent to those of Lemma \ref{lem:DualGraphL1} neither to those of \cite{AST_1993__217__251_0}.

\end{exmp}


\bibliographystyle{smfalpha}
\bibliography{sample}
Rodolfo Aguilar Aguilar \\
International Center for Mathematical Sciences,\\
 Institute of Mathematics and Informatics, \\
  Bulgarian Academy of Sciences, \\
Bulgaria, Sofia 1113, Acad. G. Bonchev St. bl.8, \\
aaguilar.rodolfo@gmail.com

\end{document}